\documentclass[11pt]{amsart}
\usepackage[OT2,T1]{fontenc}
\usepackage{amsmath,amsfonts,amsthm,amssymb,mathrsfs,
amsxtra,amscd,latexsym,color,paralist,bbm,enumitem}
\usepackage[hmargin=2.5cm,vmargin=3.5cm]{geometry}
\usepackage[normalem]{ulem}
\usepackage{soul}
\usepackage{setspace}
\usepackage{hyperref}

\DeclareSymbolFont{cyrletters}{OT2}{wncyr}{m}{n}
\DeclareMathSymbol{\Sha}{\mathalpha}{cyrletters}{"58}

\definecolor{blue}{rgb}{0,0,1}
\definecolor{red}{rgb}{1,0,0}
\definecolor{green}{rgb}{0,.6,.2}
\definecolor{purple}{rgb}{1,0,1}

\long\def\red#1\endred{\textcolor{red}{#1}}
\long\def\blue#1\endblue{\textcolor{blue}{#1}}
\long\def\purple#1\endpurple{\textcolor{purple}{ #1}}
\long\def\green#1\endgreen{\textcolor{green}{#1}}

 \newtheorem{thm}{Theorem}[section]
 \newtheorem*{thm*}{Theorem}
 
 \newtheorem{cor}[thm]{Corollary}
 
 \newtheorem{lem}[thm]{Lemma}

 \newtheorem{rmk}[thm]{Remark}

 \theoremstyle{definition}
 \theoremstyle{remark}
 \numberwithin{equation}{section}

\newcommand{\sm}{\left(\begin{smallmatrix}}
\newcommand{\esm}{\end{smallmatrix}\right)}
\newcommand{\mat}{\left(\begin{matrix}}
\newcommand{\emat}{\end{matrix}\right)}

\newcommand{\bpm}{\begin{pmatrix}}
\newcommand{\ebpm}{\end{pmatrix}}

\def\CC{\mathbb{C}}

\def\Q{\mathbb{Q}}
\def\FF{\mathbb{F}}
\def\RR{\mathbb{R}}
\def\ZZ{\mathbb{Z}}

\def\fin{\mathrm{fin}}
\def\val{\mathrm{val}}
\def\det{\mathrm{det}}

\def\sgn{\mathop{\rm sgn}}

\def\tr{\mathrm{tr}}

\def\G{\mathrm{G}}
\def\Z{\mathrm{Z}}
\def\K{\mathrm{K}}
\def\GL{\mathrm{GL}}

\def\SL{\mathrm{SL}}

\def\vol{\mathrm{Vol}}

\makeatletter
\@namedef{subjclassname@2020}{\textup{2020} Mathematics Subject Classification}
\makeatother

\begin{document}

\title[The number of automorphic representations of $\GL_2$ with exceptional eigenvalues]{The number of automorphic representations of $\GL_2$ with exceptional eigenvalues}

 \author{Dohoon Choi}

    \author{Min Lee}

    \author{Youngmin Lee}

    \author{Subong Lim}

    \address{Department of Mathematics, Korea University, 145 Anam-ro, Seongbuk-gu, Seoul 02841, Republic of Korea}
    \email{dohoonchoi@korea.ac.kr}	

    \address{School of Mathematics, University of Bristol, Bristol BS8 1TW, U.K.}
    \email{min.lee@bristol.ac.uk}

    \address{School of Mathematics, Korea Institute for Advanced Study, 85 Hoegiro, Dongdaemun-gu,
    Seoul 02455, Republic of Korea}
    \email{youngminlee@kias.re.kr}

    \address{Department of Mathematics Education, Sungkyunkwan University, Jongno-gu, Seoul 03063, Republic of Korea}
    \email{subong@skku.edu}


\begin{abstract}
We obtain an upper bound for the dimension of the cuspidal automorphic forms for $\GL_2$ over a number field, whose archimedean local representations are not tempered.     
More precisely, we prove the following result.

Let $F$ be a number field and $\mathbb{A}_{F}$ be the ring of adeles of $F$.
Let $\mathcal{O}_{F}$ be the ring of integers of $F$.
Let $\mathfrak{X}_{F,\mathrm{ex}}$ be the set of irreducible cuspidal automorphic representations $\pi$ of $\GL_2(\mathbb{A}_{F})$ with the trivial central character such that for each archimedean place $v$ of $F$, the local representation of $\pi$ at $v$ is an unramified principal series and is not tempered.
For an ideal $J$ of $\mathcal{O}_{F}$, let $\K_{0}(J)$ be the subgroup of $\GL_2(\mathbb{A}_{F})$ corresponding to $\Gamma_0(J) \subset \mathrm{SL}_2(\mathcal{O}_F)$. 
Let $r_1$ be the number of real embeddings of $F$ and $r_2$ be the number of conjugate pairs of complex embeddings of $F$. 
Using the Arthur-Selberg trace formula, we have
\begin{equation*}
        \sum_{\pi\in \mathfrak{X}_{F,\mathrm{ex}}} \dim \pi^{\K_0(J)}
        \ll_{F} \frac{[\SL_2(\mathcal{O}_{F}) : \Gamma_0(J)]}{(\log (N_{F/\Q}(J)))^{2r_1+3r_2}} \quad \text{ as } \quad |N_{F/\Q}(J)|\to \infty.
    \end{equation*}
From this result, we obtain the result on an upper bound for the number of Hecke-Maass cusp forms of weight $0$ on $\Gamma_0(N)$ which do not satisfy the Selberg eigenvalue conjecture. 
\end{abstract}	

\maketitle 
\section{Introduction}\label{s : intro}

Let $\Gamma$ be a congruence subgroup of $\SL_2(\ZZ)$ and $f$ be a Maass cusp form of weight $0$ on $\Gamma$.
Let $\Delta:=-y^2\left(\frac{\partial^2}{\partial x^2} + \frac{\partial^2}{\partial y^2}\right)$ be the hyperbolic Laplace operator on the upper half plane $\mathbb{H}$, and $\lambda_{f}$ be the eigenvalue of $\Delta$ on $f$.
In \cite{S65}, Selberg proved that $\lambda_{f}\geq \frac{3}{16}$, and asserted that $\lambda_{f}\geq \frac{1}{4}$.
It is called the Selberg eigenvalue conjecture.
The best known lower bound for $\lambda_{f}$ is that 
\begin{equation*}
    \lambda_{f}\geq \frac{1}{4}-\left(\frac{7}{64}\right)^2,
\end{equation*}
which was proved by Kim and Sarnak \cite{KS03}. 
This conjecture has been proven for some congruence subgroups (see \cite{BLS20, BS07,H86}).

For $\lambda\in \RR$, the space $V_{\lambda}$ of Maass cusp forms of weight $0$ on $\Gamma$ with the Laplacian eigenvalue $\lambda$ is a finite-dimensional vector space over $\CC$.
If $\lambda_{f}<\frac{1}{4}$, then we call $\lambda_{f}$ an exceptional eigenvalue, and its multiplicity is equal to the dimension of $V_{\lambda_f}$ over $\CC$.
The number of exceptional eigenvalues with multiplicities also has been studied in various aspects, as shown in \cite{H18,H86,I02,IS85}.
For $\Gamma=\Gamma_0(N)$ with a positive integer $N$, Iwaniec and Szmidt \cite{IS85} and Huxley \cite{H86} independently proved that if $e=2$, then the number of exceptional eigenvalues $\lambda_{f}$ with multiplicities satisfying $0<\lambda_{f}<\frac{1}{4}-\delta^2$ is less than $C_{\epsilon}\cdot N^{1-e\delta+\epsilon}$ for some constant $C_{\epsilon}$ depending on $\epsilon$ with $\epsilon>0$. Later, Iwaniec \cite{I02} proved that it holds for $e=4$.

Let $F$ be a number field and $\mathbb{A}_{F}$ be the ring of adeles of $F$.
In this paper, we consider the number of Maass cusp forms of $\mathrm{GL}_2(\mathbb{A}_{F})$ with exceptional eigenvalues. 
The Selberg eigenvalue conjecture for Maass cusp forms of $\mathrm{GL}_{2}(\mathbb{A}_{F})$ can be stated via automorphic representations as follows. 
Let $\mathfrak{X}_{F}$ denote the set of irreducible cuspidal automorphic representations $\pi \cong \otimes_{v} \pi_v$ of $\GL_2(\mathbb{A}_{F})$ with the trivial central character such that the local representation $\pi_v$ of $\pi$ at each archimedean place $v$ of $F$ is an unramified principal series.
When $F=\Q$ and $\Gamma$ is a congruence subgroup of $\SL_2(\ZZ)$, there is a correspondence between Maass cusp forms of weight $0$ on $\Gamma$ and the fixed vectors under $\Gamma_{\mathbb{A}_{\Q}}$ of automorphic representations $\pi\in \mathfrak{X}_{\Q}$.
Here, $\Gamma_{\mathbb{A}_{\Q}}$ is the subgroup of $\GL_2(\mathbb{A}_{\Q})$ corresponding to $\Gamma$.
For a Maass cusp form $f$ of weight $0$ on $\Gamma$, let $\pi_{f}\in \mathfrak{X}_{\Q}$ be the irreducible cuspidal automorphic representation of $\GL_2(\mathbb{A}_{\Q})$ corresponding to $f$. By the $(\mathfrak{g},K)$-module theory due to Harish-Chandra, $\lambda_{f} \geq \frac{1}{4}$ if and only if $(\pi_f)_{\infty}$ is tempered.
Thus, the Selberg eigenvalue conjecture is equivalent to that $(\pi_{f})_{\infty}$ is tempered for every Maass cusp form $f$ of weight $0$ on $\Gamma$. 
Moreover, the number of exceptional eigenvalues with multiplicities is equal to the sum of $\dim \pi_{f}^{\Gamma_{\mathbb{A}_{\Q}}}$ with respect to $f$, where $f$ is a Maass form of weight $0$ on $\Gamma$ such that $(\pi_f)_{\infty}$ is not tempered.
 Here, for $\pi\in \mathfrak{X}_{F}$ and $H\leq \GL_2(\mathbb{A}_{F})$, $\pi^{H}$ denotes the space of fixed vectors under $H$ of $\pi$, i.e.,   
\begin{equation*}
   \pi^{H}:=\left\{v\in V_{\pi} : \pi(g)v=v \text{ for }g\in H \right\},
\end{equation*}
where $V_{\pi}$ is the underlying vector space of $\pi$.
Let $\mathfrak{X}_{F,\mathrm{ex}}$ be the subset of $\mathfrak{X}_{F}$ consisting of $\pi$ such that for every archimedean place $v$ of $F$, $\pi_v$ is not tempered.


Let $\mathcal{O}_{F}$ be the ring of integers of $F$.
For each place $v$ of $F$, let $F_{v}$ be the completion of $F$ at $v$ and $\mathcal{O}_{F_v}$ be the ring of integers of $F_{v}$.
Let $S_{F,\fin}$ (resp. $S_{F,\infty}$) be the set of non-archimedean (resp. archimedean) places of $F$.
For $v\in S_{F,\fin}$, let $\mathfrak{p}_{v}$ be the prime ideal of $\mathcal{O}_{F}$ corresponding to $v$.
For a non-negative integer $e$, we define $\K_{v,e}$ by
\begin{equation*}
    \K_{v,e}:=\left\{\begin{pmatrix}
        a & b\\
        c & d
    \end{pmatrix}\in \GL_2(\mathcal{O}_{F_v}) : c\in \mathfrak{p}_{v}^{e}\right\}.
\end{equation*}
For each $v\in S_{F,\infty}$, let $\K_{v}^{0}$ be the maximal connected compact subgroup of $\GL_2(F_v)$ defined by
\begin{equation*}
    \K_{v}^{0}:=\begin{cases}
        \mathrm{SO}(2)\quad &\text{if }F_v=\RR,\\
        \mathrm{U}(2) \quad &\text{if }F_v=\CC.
    \end{cases}    
\end{equation*}
Let $J$ be an ideal of $\mathcal{O}_F$, and $N_{F/\Q}(J):= [\mathcal{O}_{F} : J]$ be the absolute norm of $J$. 
Since $\mathcal{O}_{F}$ is a Dedekind domain, there is a unique non-negative integer $\val_{v}(J)$ for each $v\in S_{F,\fin}$ such that 
\begin{equation*}
    J=\prod_{v\in S_{F,\fin}} \mathfrak{p}_{v}^{\val_{v}(J)}.
\end{equation*}
Then, we let 
\begin{equation*}
    \K(J):=\prod_{v\in S_{F,\fin}} \K_{v,\val_v(J)} \cdot \prod_{v\in S_{F,\infty}} \K_{v}^{0}
\end{equation*}
and 
\begin{equation*}
    \mathcal{N}(J):=\sum_{\pi \in \mathfrak{X}_{F,\mathrm{ex}}} \dim \pi^{\K(J)}.
\end{equation*} 
In the following theorem, we obtain an upper bound for $\mathcal{N}(J)$. 

\begin{thm}\label{thm : main}
    Assume that  $F$ is a number field. Assume that $r_1$ is the number of real embeddings of $F$ and that $r_{2}$ is the number of conjugate pairs of complex embeddings of $F$.
    Then, for  non-trivial ideals $J$ in $\mathcal{O}_{F}$, 
    \begin{equation*}
        \mathcal{N}(J) \ll_{F} \frac{[\SL_2(\mathcal{O}_{F}) : \Gamma_0(J)]}{(\log (N_{F/\Q}(J)))^{2r_1+3r_2}}, \quad |N_{F/\Q}(J)|\to \infty.
    \end{equation*}
\end{thm}
Let us consider the case when the number of archimedean places of $F$ is $1$.
In this case, $F$ is the field of rational numbers or an imaginary quadratic field, and $\mathcal{N}(J)$ means the number of exceptional eigenvalues with multiplicities for a congruence subgroup corresponding to $J$. 

When $F$ is an imaginary quadratic field, an automorphic form on $\mathbb{H}^{3}$ corresponding to a fixed vector of $\pi\in \mathfrak{X}_{F}$ is called a Maass cusp form of weight $0$ over $F$ (for details, see \cite{EGM98}).
Note that the Laplace–Beltrami operator $\Delta_{\mathbb{H}^3}$ on $\mathbb{H}^{3}:=\{x+iy+jr : x,y,r\in \RR \text{ and } r>0 \}$ is defined by
\begin{equation*}
    \Delta_{\mathbb{H}^3}:=r^2\left(\frac{\partial^2}{\partial x^2}+\frac{\partial^2}{\partial y^2}+\frac{\partial^2}{\partial r^2} \right)-r\frac{\partial}{\partial r}.
\end{equation*}
Assume that $f$ is a Maass cusp form of weight $0$ over an imaginary quadratic field $F$ and $\pi_{f}\in \mathfrak{X}_{F}$ is an automorphic representation corresponding to $f$.
Let $\lambda_{f}$ be the eigenvalue of $\Delta_{\mathbb{H}^3}$ on $f$. 
Then, the Selberg eigenvalue conjecture is equivalent to that $\lambda_{f}\geq 1$. 

Let $T_{v}$ be the Hecke operator for $v\in S_{F,\fin}$. 
For an ideal $J$ of $\mathcal{O}_{F}$, a Hecke-Maass cusp form of weight $0$ on $\Gamma_0(J)$ over $F$ is a Maass cusp form of weight $0$ on $\Gamma_0(J)$ over $F$ that is an eigenform for all $T_{v}$ with $\val_v(J)=0$.
A Hecke-Maass cusp form $f$ is called a normalized Hecke-Maass cusp form if the $L^2$-norm of $f$ is $1$.
Then, the vector space $V_{\lambda}$ of Maass cusp forms of weight $0$ on $\Gamma_0(J)$ over $F$ with the Laplacian eigenvalue $\lambda$ has a basis consisting of normalized Hecke-Maass cusp forms. The following corollary is immediately implied by Theorem \ref{thm : main}. 

\begin{cor}\label{thm : main - inf place 1}
Assume that the number of archimedean places of $F$ is $1$. 
Let $J$ be a non-trivial ideal of $\mathcal{O}_{F}$ and $\mathcal{N}(\Gamma_0(J))$ be the number of normalized Hecke-Maass cusp forms of weight $0$ on $\Gamma_0(J)$ over $F$ that do not satisfy the Selberg eigenvalue conjecture. 
Then, the followings are true. 
\begin{enumerate}
    \item Assume that $F=\Q$ and $N$ is a positive integer with $N>1$.
    For convenience, we write $\Gamma_0(N)$ for $\Gamma_0((N))$. 
    Then, we have  
    \begin{equation*}
        \mathcal{N}(\Gamma_0(N))\ll \frac{[\SL_2(\ZZ) : \Gamma_0(N)]}{(\log N)^2},\quad N\to \infty.
    \end{equation*}
    
    \item Assume that $F$ is an imaginary quadratic field.
    Then, we have 
    \begin{equation*}
        \mathcal{N}(\Gamma_0(J))\ll_{F} \frac{[\SL_2(\mathcal{O}_{F}) : \Gamma_0(J)]}{(\log (N_{F/\Q}(J)))^3}, \quad |N_{F/\Q}(J)|\to \infty.
    \end{equation*}
\end{enumerate}
\end{cor}

\begin{rmk}
    Recalling the result of Iwaniec \cite{I02}, the number of normalized Hecke-Maass cusp forms $f$ of weight $0$ on $\Gamma_0(N)$ satisfying $0<\lambda_{f}<\frac{1}{4}-\delta^2$ is less than $C_{\epsilon}\cdot N^{1-4\delta+\epsilon}$, where $C_{\epsilon}$ is a constant depending only on $\epsilon>0$.
    If we take $\delta=0$, then we have
    \begin{equation}\label{eq 104}
        \mathcal{N}(\Gamma_0(N))\ll_{\epsilon} N^{1+\epsilon},\quad N\to \infty.
    \end{equation}
    Since $[\SL_2(\ZZ):\Gamma_0(N)]\ll N\log \log N$, Corollary \ref{thm : main - inf place 1} implies 
    \begin{equation*}
        \mathcal{N}(\Gamma_0(N))\ll \frac{N \log \log N}{(\log N)^2},\quad N\to \infty.
    \end{equation*}
\end{rmk}

The rest of this paper is organized as follows. 
In Section \ref{s : Pre}, we review the Arthur-Selberg trace formula which is mainly used to prove Theorem \ref{thm : main} and describe how to obtain an upper bound for $\mathcal{N}(J)$ by using the Arthur-Selberg trace formula.
In Section \ref{s : geometric side}, we compute the geometric side of the Arthur-Selberg trace formula.
In Section \ref{s : pf}, we prove Theorem \ref{thm : main}.

\subsection*{Acknowledgments.} The authors appreciate Andrew Knightly for his kind and helpful comments. The second author is supported by a Royal Society University Research Fellowship.

\section{Preliminaries}\label{s : Pre}
We introduce some notions as follows.
Let $\G:=\GL_2$ and $\overline{\G}:=\G/\Z$, where $\Z$ is defined by the center of $\G$.
Let $F$ be a number field and $S_{F,\fin}$ (resp. $S_{F,\infty}$) be the set of non-archimedean (resp. archimedean) places of $F$, and $S_{F}:=S_{F,\fin}\cup S_{F,\infty}$ be the set of places of $F$.
Let $\mathbb{A}_{F}$ be the ring of adeles of $F$ and $\mathbb{A}_{F,\infty}$ be the ring of infinite adeles of $F$.
For each $v\in S_{F}$, let $F_{v}$ be the completion of $F$ at $v$ and $\mathcal{O}_{F_v}$ be the ring of integers of $F_{v}$.
Let $\K_{v}$ be the maximal compact subgroup of $\G(F_v)$ defined by
\begin{equation*}
    \mathrm{K}_{v}:=\begin{cases}
        \G(\mathcal{O}_{F_v}) \quad &\text{if $v\in S_{F,\fin}$,}\\
        \mathrm{O}(2) \quad &\text{if $v\in S_{F,\infty}$ and $F_{v}=\RR$,}\\
        \mathrm{U}(2) \quad &\text{if $v\in S_{F,\infty}$ and $F_v=\CC$.}
    \end{cases}
\end{equation*}

For each $v\in S_{F,\fin}$, let $\varpi_{v}$ be the uniformizer of $\mathcal{O}_{F_v}$ and $q_{v}:=[\mathcal{O}_{F_v} : (\varpi_v)]$.
Let $\mathfrak{p}_{v}:=(\varpi_v)$ be a unique prime ideal of $\mathcal{O}_{F_v}$.
For $g=\sm a & b\\
c & d\esm\in \G(F_v)$ with $v\in S_{F,\fin}$, a Haar measure $dg$ on $\G(F_v)$ is defined by
\begin{equation}\label{eq 18}
    dg:=\frac{q_v^3}{(q_v-1)^2(q_v+1)}\cdot\frac{d_va \, d_v b \, d_v c \, d_v d}{|ad-bc|_{v}^2},
\end{equation}
where $d_v x$ is a Haar measure on $F_{v}$ satisfying $\vol(\mathcal{O}_{F_v})=1$, and $|\cdot |_{v}$ denotes the $v$-adic absolute value on $F_{v}$.
From \eqref{eq 18}, we have $\vol(\K_v)=1$. 
For $v\in S_{F,\fin}$, a multiplicative Haar measure $d_{v}^{\times}x$ on $F_v^{\times}$ is defined by 
\begin{equation*}
    d_{v}^{\times}x :=\frac{q_v}{q_v-1}\cdot\frac{d_v x}{|x|_{v}}.
\end{equation*}

For $v\in S_{F,\infty}$, a multiplicative Haar measure $d_{v}^{\times}x$ on $F_{v}^{\times}$ is defined by 
\begin{equation*}
    d_{v}^{\times}x:=\begin{cases}
    \frac{d_{\RR} x}{2|x|_{v}} \quad &\text{if }F_v=\RR,\\
    \frac{d_{\CC} x}{\pi |x|_{v}} \quad &\text{if }F_v=\CC,
    \end{cases}
\end{equation*}
and $|y|_{v}:=|y|^{[F_v:\RR]}$.
If $g\in \G(F_v)$ with $v\in S_{F,\infty}$, then by the Iwasawa decomposition, there are $y,z\in F_v^{\times}$, $x\in F_v$ and $\kappa\in \K_v$ such that $g=\sm z & 0 \\
0  & z\esm \sm y & 0\\
0 & 1\esm 
\sm 1 & x\\
0 & 1\esm
\kappa$. 
Then, a Haar measure $dg$ on $\G(F_v)$ is defined by
\begin{equation*}
    dg:= d_{v}^{\times}z \, d_{v}^{\times}y \, d_vx \, d_v\kappa,
\end{equation*}
where $d_v\kappa$ is a Haar measure on $\K_v$.
For any $v\in S_{F}$, we set a Haar measure $d_v \overline{\kappa}$ on $\Z(F_v)\backslash \Z(F_v)\K_v$ such that the volume of $\Z(F_v)\backslash \Z(F_v)\K_v$ is equal to $1$.
Then, a Haar measure $dg$ on $\overline{\G}(F_v)$ is defined by 
\[dg:= d_{v}^{\times}y \, d_v x \, d_v\overline{\kappa}, \]
where $g=\sm y & 0\\
0 & 1 \esm \sm 1 & x\\
0 & 1 \esm \overline{\kappa}$.
Thus, a Haar measure on $\overline{\G}(\mathbb{A}_{F})$ is defined by the product of Haar measures on $\overline{\G}(F_v)$ for all $v\in S_F$.

For a non-archimedean place $v$ of $F$ and a non-negative integer $e$, let 
\begin{equation*}
    \K_{v,e}:=\left\{\begin{pmatrix}
        a & b\\
        c & d
    \end{pmatrix}\in \G(\mathcal{O}_{F_v}) : \varpi_{v}^{e}\mid c \right\}.
\end{equation*}
Assume that $J$ is an ideal of $\mathcal{O}_{F}$. 
Then, for each $v\in S_{F,\fin}$, there is a unique non-negative integer $\val_v(J)$ such that
\begin{equation*}
    J=\prod_{v\in S_{F,\fin}} \mathfrak{p}_{v}^{\val_v(J)}.
\end{equation*}
For convenience, let $\val_{v}(m):=\val_{v}((m))$ for $m\in \mathcal{O}_{F}$, where $(m)=m\mathcal{O}_{F}$ denotes the principal ideal generated by $m$ in $\mathcal{O}_{F}$.   
Let $N_{F/\Q}(J):=[\mathcal{O}_{F} : J]$ be the absolute norm of $J$.
Let $\K_{\fin}(J):=\prod_{v\in S_{F,\fin}} \K_{v,\val_v(J)}$ and $\K(J):= \K_{\fin}(J)\cdot\prod_{v\in S_{F,\infty}} \K_{v}^{0}$, where 
\begin{equation}\label{eq 38}
    \K_{v}^{0}:=\begin{cases}
        \mathrm{SO}(2)\quad &\text{if }F_v=\RR,\\
        \mathrm{U}(2) \quad &\text{if }F_v=\CC
    \end{cases}    
\end{equation}
is the maximal connected compact subgroup of $\G(F_v)$. 

For each $v\in S_{F,\infty}$ and $s\in \mathbb{C}$, let 
\begin{equation*}
    \mathbb{H}_{v}\left(|\cdot|_{v}^{s}, |\cdot|_{v}^{-s}\right):= \left\{f: \G(F_v)\to \CC  \quad \bigg| \int_{\K_v} |f(\kappa_v)|^2 d_v\kappa_v<\infty \text{ and } f\left(\sm a & x\\
    0 & b\esm g\right)= \left|\frac{a}{b}\right|_{v}^{s+\frac{1}{2}}f(g) \right\}.
\end{equation*}
Recall that $\mathfrak{X}_{F}$ is defined by the set of irreducible cuspidal automorphic representations $\pi$ of $\G(\mathbb{A}_{F})$ with the trivial central character such that the local representation $\pi_{v}$ of $\pi$ at each archimedean place $v$ of $F$ is an unramified principal series, and that $\mathfrak{X}_{F,\mathrm{ex}}$ is the subset of $\mathfrak{X}_{F}$ consisting of $\pi$ such that $\pi_v$ is not tempered for every $v\in S_{F,\infty}$.
In other words, if $\pi\in \mathfrak{X}_{F}$, then for each $v\in S_{F,\infty}$, there is a unique spectral parameter $\nu(\pi_v)\in \mathbb{C}$ such that $\pi_v$ is a right regular representation of $\G(F_v)$ on $\mathbb{H}_{v}\left(|\cdot|_{v}^{\nu(\pi_v)}, |\cdot|_{v}^{-\nu(\pi_v)}\right)$.
Note that for $\pi\in \mathfrak{X}_{F}$, the spectral
parameter $\nu(\pi_v)$ is either a purely imaginary number or a non-zero real number with $|\nu(\pi_v)| \leq \frac{1}{2}$.
Since $\pi_v$ is tempered if and only if $\nu(\pi_v)$ is a purely imaginary number, it follows that $\mathfrak{X}_{F,\mathrm{ex}}$ is the subset of $\mathfrak{X}_{F}$ consisting of $\pi$ such that $\nu(\pi_v)\in \RR^{\times}$ for all $v\in S_{F,\infty}$.
For $\pi\in \mathfrak{X}_{F}$, let $V_{\pi}$ be the underlying vector space of $\pi$ and  
\begin{equation*}
   \pi^{\K(J)}:=\left\{v\in V_{\pi} : \pi(g)v=v \text{ for }g\in \K(J) \right\}. 
\end{equation*}
As in Section \ref{s : intro}, we define 
\begin{equation*}
    \mathcal{N}(J):=\sum_{\pi\in \mathfrak{X}_{F,\mathrm{ex}}} \dim \pi^{\K(J)}.
\end{equation*} 
The goal of this section is to derive the formula for an upper bound for $\mathcal{N}(J)$ by using the Arthur-Selberg trace formula.

For each $v\in S_{F,\infty}$, let $\phi_{v}$ be a smooth function on $\G(F_v)$ satisfying the following conditions :
\begin{enumerate}[label=(\alph*)]
    \item $\phi_{v}(z\kappa_1 g \kappa_2)=\phi_{v}(g)$ for every $z\in \Z(F_v)$, $\kappa_1,\kappa_2\in \K_{v}^{0}$ and $g\in \G(F_v)$,
    \item $\phi_{v}$ is compactly supported modulo $\Z(F_v)$,
    \item If $F_v=\RR$, then the support of $\phi_{v}$ is contained in $\GL_2^{+}(\mathbb{R})$.
\end{enumerate}
Since $\phi_{v}$ is a $\Z(F_v)$-invariant function, throughout this paper we see that $\phi_{v}$ is a function on $\overline{\G}(F_v)$ by abusing the notation.
Let $\widehat{h_{\phi_{v}}} : \mathbb{R}\to \mathbb{C}$ and $h_{\phi_{v}} : \mathbb{C} \to \mathbb{C}$ be defined by 
\begin{equation}\label{e:hathphiv_def}
\widehat{h_{\phi_v}}(t)
:= 2\pi^{1+\epsilon_v} \int_{\frac{2\pi |t|}{1+\epsilon_v}}^\infty 
\frac{\phi_v \left(\sm e^{-\frac{r}{2}}& 0 \\ 0 & e^{\frac{r}{2}}\esm\right)}
{(\sinh^2(\frac{r}{2})-\sinh^2(\pi t))^{\frac{1-\epsilon_v}{2}}} \sinh(r)\, dr, 
\end{equation}
where $\epsilon_v:=[F_v:\mathbb{R}]-1$, and 
\begin{equation}
    h_{\phi_{v}}(z) := \int_{\RR} \widehat{h_{\phi_{v}}}(x)e^{-2\pi izx} dx.
\end{equation}

Assume that $\widehat{h}:\mathbb{R}\to \mathbb{R}_{\geq 0}$ and $h(z):=\int_{\RR} \widehat{h}(x)e^{-2\pi i z x}dx$ satisfy the following conditions : 
\begin{enumerate}
\item $\widehat{h}$ is smooth and compactly supported, 
\item $\widehat{h}$ is even,
\item $\widehat{h}(0) = \int_{\RR} h(x) dx = 1$, 
\item $h$ is entire,
\item $h$ is rapidly decreasing on horizontal strips, 
\item $h(x)\geq 0$ on $x\in \RR$,
\item $h(x)>0$ and $h(ix)>0$ on $x\in [-\frac{1}{2}, \frac{1}{2}]$.
\end{enumerate}
To show the existence of $h$ and $\widehat{h}$, let us take $g_0:\RR\to\RR$ such that $g_0$ is smooth, even, non-negative on $\mathbb{R}$ and supported on $[-1,1]$.
Let $g_1:=g_0 *g_0$. 
Then, $g_1$ is supported on $[-2,2]$ and
\begin{equation*}
    g_1(0)=\int_{-\infty}^{\infty} g_0(x)^2 dx>0.
\end{equation*} 
Let $\widehat{h_1}:=g_1/g_1(0)$. 
Since $h_1=\widehat{\widehat{h_1}}$, we have for $x\in \RR$,
\begin{equation}
    h_1(x)=\frac{1}{g_1(0)}\widehat{g_0}(x)^2\geq 0,
\end{equation}
and $h_1(0)>0$. 
Note that if $x\in \RR$, then $h_1(x)\in \RR$ and $h_1(ix)\in \RR$. Thus, there is $\delta>0$ such that $h_1(x)$ and $h_1(ix)$ are positive on $x\in [-\frac{\delta}{2},\frac{\delta}{2}]$. 
Let $h(z):=h_1(\delta z)$. Then, $\widehat{h}$ and $h$ satisfy the conditions (1) $\sim$ (7).
Using the Abel inversion formula, we obtain the following lemma. 

\begin{lem}\label{lem 5}
    Let $h:\mathbb{C}\to \mathbb{C}$ be a function satisfying the conditions $(1)\sim (7)$. 
    Assume that a function $\phi_{v} : \G(F_v)\to \mathbb{C}$ satisfies the conditions $(a) \sim (c)$ such that 
    \begin{equation*}
        h_{\phi_{v}}=h.
    \end{equation*}
    Then, when $F_v=\RR$, we have for $r\in \RR$,
\begin{align}
\phi_{v}\left(\bpm e^{-\frac{r}{2}} & 0\\0 & e^{\frac{r}{2}} \ebpm \right)
& = - \frac{1}{4\pi^2}\int_{\frac{r}{2\pi}}^\infty \frac{\widehat{h_{\phi_v}}'(t)}{\sqrt{\sinh^2(\pi t)-\sinh^2(r/2)}} dt \label{e:trans_phiv_hathv_R}
\\ & = \frac{1}{2\pi}\int_{\frac{r}{2\pi}}^\infty
\int_{-\infty}^\infty x h_{\phi_v}(x) \frac{\sin(2\pi xt)}{\sqrt{\sinh^2(\pi t)-\sinh^2(r/2)}} dx dt. \label{e:trans_phiv_hv_R} 
\end{align}
When $F_v=\CC$, we have for $r\in \RR$,
\begin{equation}\label{e:trans_phi_vhathv_C}
\phi_{v}\left(\bpm e^{-\frac{r}{2}} & 0 \\ 0 & e^{\frac{r}{2}} \ebpm \right)
= -\frac{\widehat{h_{\phi_{v}}}'(r/\pi)}{2\pi^{3}\sinh(r)}
= \frac{1}{\pi^2} \int_{-\infty}^\infty xh_{\phi_{v}}(x) \frac{\sin(2xr)}{\sinh(r)} dx. 
\end{equation}
\end{lem}
\begin{proof}
When $F_v=\RR$, see \cite[Proposition 4.1, pp 15-16]{HE1} and when $F_v=\CC$, see \cite[Lemma 3.5.5, pp 121]{EGM98}.
\end{proof}

Conversely, for a function $h$ satisfying the conditions $(1)\sim (7)$, we define a function $\phi_{v}$ on $\G(F_v)$ satisfying the conditions $(a)\sim (c)$ as follows. If $F_v=\RR$, then 
 \begin{equation*}
     \phi_{v}\left(z\kappa_1 \bpm e^{-\frac{r}{2}} & 0\\
     0 & e^{\frac{r}{2}}\ebpm \kappa_2 \right)=\phi_v\left(\bpm e^{-\frac{r}{2}} & 0\\
     0 & e^{\frac{r}{2}}\ebpm \right):= - \frac{1}{4\pi^2}\int_{\frac{r}{2\pi}}^\infty \frac{\widehat{h}'(t)}{\sqrt{\sinh^2(\pi t)-\sinh^2(r/2)}} dt
 \end{equation*}
and if $F_v=\CC$, then 
 \begin{equation*}
     \phi_{v}\left(z\kappa_1 \bpm e^{-\frac{r}{2}} & 0\\
     0 & e^{\frac{r}{2}}\ebpm \kappa_2 \right)=\phi_v\left(\bpm e^{-\frac{r}{2}} & 0\\
     0 & e^{\frac{r}{2}}\ebpm \right):=-\frac{\widehat{h}'(r/\pi)}{2\pi^{3}\sinh(r)}.
 \end{equation*}
Here, $z\in \Z(F_v)$, $\kappa_1,\kappa_2\in \K_{v}^{0}$ and $r\geq 0$.
By \eqref{e:hathphiv_def}, we see that $h_{\phi_v}$ is an even function.
When $F_v=\RR$, we have for $t\geq 0$,
\begin{equation}\label{eq 106}
\begin{aligned}
    \widehat{h_{\phi_v}}(t)&=2\pi\int_{2\pi t}^\infty 
\frac{\phi_v \left(\sm e^{-\frac{r}{2}}& 0 \\ 0 & e^{\frac{r}{2}}\esm\right)}
{\sqrt{\sinh^2(\frac{r}{2})-\sinh^2(\pi t))}} \sinh(r)dr\\
&=-\frac{1}{2\pi}\int_{2\pi t}^{\infty}\int_{\frac{r}{2\pi}}^{\infty} \frac{\sinh(r)}{\sqrt{\sinh^2(\frac{r}{2})-\sinh^2(\pi t)}}\cdot \frac{\widehat{h}'(x)}{\sqrt{\sinh^2(\pi x)-\sinh^2(\frac{r}{2})}} dx dr\\
&=-\frac{1}{2\pi}\int_{t}^{\infty}\widehat{h}'(x) \int_{2\pi t}^{2\pi x} \frac{\sinh(r)}{\sqrt{\sinh^2(\frac{r}{2})-\sinh^2(\pi t)}\sqrt{\sinh^2(\pi x)-\sinh^2(\frac{r}{2})}} dr dx.
\end{aligned}
\end{equation}
The last equality holds by Fubini's theorem. 
By changing the variable $y=\sinh^2(\frac{r}{2})$, we have 
\begin{equation*}
    \int_{2\pi t}^{2\pi x} \frac{\sinh(r)}{\sqrt{\sinh^2(\frac{r}{2})-\sinh^2(\pi t)}\sqrt{\sinh^2(\pi x)-\sinh^2(\frac{r}{2})}} dr = 2\pi.
\end{equation*}
Thus, \eqref{eq 106} becomes 
\begin{equation*}
    \widehat{h_{\phi_v}}(t)=-\int_{t}^{\infty}\widehat{h}'(x)dx=\widehat{h}(t).
\end{equation*}
When $F_v=\CC$, for $t\geq 0$, we have 
\begin{equation*}
\begin{aligned}
    \widehat{h_{\phi_v}}(t)&=2\pi^2\int_{\pi t}^{\infty} \phi_v\left(\bpm e^{-\frac{r}{2}} & 0\\
    0 & e^{\frac{r}{2}}\ebpm \right) \sinh(r) dr\\
    &=-\frac{1}{\pi}\int_{\pi t}^{\infty} \widehat{h}'\left(\frac{r}{\pi}\right)dr\\
    &=\widehat{h}(t).
\end{aligned}
\end{equation*}
Since $\widehat{h}$ is smooth and compactly supported, it follows that $\phi_{v}$ is also smooth and compactly supported modulo $\Z(F_v)$. 
Hence, $\phi_{v}$ satisfies the conditions $(a)\sim (c)$ and 
\begin{equation*}
    h_{\phi_v}=h.
\end{equation*}
 
Let $(\pi,V_{\pi})$ be an irreducible cuspidal automorphic representation of $\G(\mathbb{A}_{F})$ with the trivial central character.
Assume that a smooth function $\phi$ on $\G(\mathbb{A}_{F})$ is $\mathrm{Z}(\mathbb{A}_{F})$-invariant and compactly supported modulo $\mathrm{Z}(\mathbb{A}_{F})$. 
We define an operator $\pi(\phi)$ on $V_{\pi}$ by 
\begin{equation*}
    \pi(\phi)v:=\int_{\overline{\G}(\mathbb{A}_{F})}\phi(g)\pi(g)v dg.
\end{equation*}
Then, the operator $\pi(\phi)$ is a trace class. 

For $v\in S_{F,\fin}$, let $\phi_v$ be a function on $\G(F_v)$ defined by the characteristic function of $\Z(F_{v})\K_{v,\val_{v}(J)}$.
From the definition of $\phi_{v}$, we see that 
\begin{equation*}
    \int_{\bar{\G}(F_v)} \phi_v(g) dg=\frac{1}{[\K_v : \K_{v,\val_v(J)}]}.
\end{equation*}
We define a constant $A_{J}$ by 
\begin{equation}\label{eq 119}
\begin{aligned}
    A_{J}:&=\prod_{v\in S_{F,\fin}} \int_{\bar{\G}(F_{v})}\phi_{v}(g) dg=\prod_{v\in S_{F,\fin}}\frac{1}{[\K_v : \K_{v,\val_v(J)}]}=\frac{1}{[\SL_2(\mathcal{O}_{F}) : \Gamma_0(J)]}. 
\end{aligned}
\end{equation}
The following lemma provides the expression of $\tr \pi(\phi)$ in terms of $h_{\phi_v}$ for $v\in S_{F,\infty}$.
\begin{lem}\label{lem : lem of trace}
    Let $F$ be a number field and $\pi=\otimes_{v\in S_{F}}\pi_v\in \mathfrak{X}_{F}$. 
    Let $r_1$ be the number of archimedean places $v$ of $F$ with $F_v=\RR$.
    Let $J$ be an ideal of $\mathcal{O}_{F}$ and for each $v\in S_{F,\fin}$, $\phi_{v}$ is defined as the characteristic function of $\Z(F_v)\K_{v,\val_{v}(J)}$.
    Assume that for each $v\in S_{F,\infty}$, $\phi_{v}$ satisfies the conditions $(a)\sim(c)$. 
    Let $\phi$ be a function on $\G(\mathbb{A}_{F})$ defined by $\phi:=\prod_{v\in S_{F}} \phi_{v}$.
    Then, we have
    \begin{equation*}
        \tr \pi(\phi)= \frac{A_{J}}{2^{r_1}}\cdot \prod_{v\in S_{F,\infty}} h_{\phi_v}(\nu(\pi_v)/i) \cdot \dim \pi^{\K(J)}.
    \end{equation*}
\end{lem}
\begin{proof}
For any $\mathbf{v}\in V_{\pi}$ and $\kappa\in \K(J)$, we have 
\begin{equation*}
    \begin{aligned}
        \pi(\kappa)\left(\pi(\phi)\mathbf{v}\right)&=\pi(\kappa)\left(\int_{\bar{\G}(\mathbb{A}_{F})} \phi(g)\pi(g)\mathbf{v} dg \right)\\
        &=\int_{\bar{\G}(\mathbb{A}_{F})}\phi(g)\pi(\kappa g)\mathbf{v} dg\\
        &=\int_{\bar{\G}(\mathbb{A}_{F})} \phi(\kappa^{-1}g)\pi(g)\mathbf{v}dg.
    \end{aligned}
\end{equation*}
Since $\phi(\kappa^{-1}g)=\phi(g)$ for any $\kappa\in \K(J)$ and $g\in \G(\mathbb{A}_{F})$, it follows that $\pi(\phi)\mathbf{v}\in \pi^{\K(J)}$. 
Thus, the trace of $\pi(\phi)$ is the same as the trace of the restriction of $\pi(\phi)$ to the space $\pi^{\K(J)}$.
To complete the proof of Lemma \ref{lem : lem of trace}, we prove that for any $\mathbf{v}\in \pi^{\K(J)}$, 
\begin{equation}\label{eq : eq of trace of pi}
    \pi(\phi)(\mathbf{v})=\left(\frac{A_{J}}{2^{r_1}}\cdot \prod_{v\in S_{F,\infty}} h_{\phi_v}(\nu(\pi_v)/i)\right)\cdot \mathbf{v}.
\end{equation}

For each place $v$ of $F$, let $(\pi_v,V_{\pi_v})$ be the local representation of $\pi$ at $v$. 
Assume $\mathbf{v}=\otimes_{v\in S_{F}} \mathbf{v}_v\in \pi^{\K(J)}$. 
Then, for each $v\in S_{F,\fin}$, we have
\begin{equation}\label{eq : eq of trace-non-arch}
\begin{aligned}
    \pi_v(\phi_v)\mathbf{v}_v&=\int_{\bar{\G}(F_v)}\phi_v(g)\pi_v(g)\mathbf{v}_v dg\\ 
    &=\int_{\overline{\K_{v,\val_{v}(J)}}} \mathbf{v}_v dg\\
    &=\mathrm{Vol}\left(\overline{\K_{v,\val_{v}(J)}}\right)\cdot \mathbf{v}_v,
\end{aligned}
\end{equation}
where $\overline{\K_{v,\val_{v}(J)}}$ is the image of $\K_{v,\val_{v}(J)}$ under the canonical projection modulo $\mathrm{Z}(F_v)$.
Following the proof of \cite[Theorems 4.1 and 5.1]{P16}, if $v\in S_{F,\infty}$, then we have
\begin{equation}\label{eq 107}
    \pi_{v}(\phi_v)\mathbf{v}_v=\begin{cases}
        \frac{1}{2}h_{\phi_v}(\nu(\pi_v)/i) \mathbf{v}_{v} \quad &\text{if }F_v=\RR,\\
        h_{\phi_v}(\nu(\pi_v)/i)\mathbf{v}_{v} \quad &\text{if }F_v=\CC.
    \end{cases}
\end{equation}
Combining \eqref{eq : eq of trace-non-arch} and \eqref{eq 107}, we complete the proof of Lemma \ref{lem : lem of trace}.
\end{proof}

Note that for $v\in S_{F,\infty}$ and $z\in \RR$, we have 
\begin{equation}\label{eq 98}
    h_{\phi_{v}}(iz)=\int_{\RR}\widehat{h_{\phi_{v}}}(x)e^{2\pi z x} dx = 2\int_{0}^{\infty}\widehat{h_{\phi_{v}}}(x)\cosh(2\pi z x) dx \geq 2\int_{0}^{\infty}\widehat{h_{\phi_{v}}}(x)dx=h_{\phi_{v}}(0).
\end{equation}
Thus, we obtain the following lemma which provides the relation between $\sum_{\pi\in \mathfrak{X}_{F}} \tr \pi(\phi)$ and $\mathcal{N}(J)$.

\begin{lem}\label{lem : lem in sec 2}
Let $J$ be an ideal of $\mathcal{O}_{F}$.
Let $\phi$ be a function on $\G(\mathbb{A}_{F})$ defined as in Lemma \ref{lem : lem of trace}. 
Then, we have
    \begin{equation*}
        \sum_{\pi\in \mathfrak{X}_{F}} \tr \pi(\phi)\geq \frac{A_{J}\cdot\mathcal{N}(J)}{2^{r_1}}\left(\prod_{v\in S_{F,\infty}}h_{\phi_v}(0)\right).
    \end{equation*} 
\end{lem}
\begin{proof}
By Lemma \ref{lem : lem of trace}, we have  
\begin{equation}
    \sum_{\pi\in \mathfrak{X}_{F}} \tr \pi(\phi)= \frac{A_{J}}{2^{r_1}}\cdot \sum_{\pi\in \mathfrak{X}_{F}}  \left(\prod_{v\in S_{F,\infty}} h_{\phi_v}(\nu(\pi_v)/i) \cdot \dim \pi^{\K(J)}\right).
\end{equation}
Assume that $\pi\in \mathfrak{X}_{F}$ and $v\in S_{F,\infty}$.
Note that $\nu(\pi_{v})\in [-\frac{1}{2},\frac{1}{2}]$ or $\nu(\pi_{v})/i\in \RR^{\times}$.
By the assumption of $h_{\phi_v}$ and \eqref{eq 98}, if $x\in \RR$ or $x/i\in [-\frac{1}{2},\frac{1}{2}]$, then $h_{\phi_v}(x)\geq 0$. It follows that
\begin{equation}
    \begin{aligned}
        \sum_{\pi\in \mathfrak{X}_{F}} \left(\prod_{v\in S_{F,\infty}}h_{\phi_v}(\nu(\pi_v)/i) \cdot \dim \pi^{\K(J)}\right) &\geq \sum_{\pi \in \mathfrak{X}_{F,\mathrm{ex}}} \left(\prod_{v\in S_{F,\infty}}h_{\phi_v}(\nu(\pi_v)/i) \cdot \dim \pi^{\K(J)}\right)\\
        &\geq \left(\prod_{v\in S_{F,\infty}} h_{\phi_v}(0)\right)\cdot \mathcal{N}(J).
    \end{aligned}
\end{equation}
\end{proof}

By the Arthur-Selberg trace formula, we have
\begin{equation}\label{eq 8}
    \sum_{\pi\in \mathfrak{X}_{F}}\tr \pi(\phi)=
- S_{\mathrm{one}}(\phi) + S_{\mathrm{id}}(\phi)
+ S_{\rm ell}(\phi) + S_{\rm hyp}(\phi) + S_{\rm par}(\phi)
+ S_{\rm Eis}(\phi) + S_{\rm Res}(\phi).
\end{equation}
Here, the definitions of $S_{\mathrm{one}}(\phi), S_{\mathrm{id}}(\phi), S_{\rm ell}(\phi), S_{\rm hyp}(\phi), S_{\rm par}(\phi), S_{\rm Eis}(\phi)$, and $S_{\rm Res}(\phi)$ are in Section \ref{s : geometric side}.
Then, Lemma \ref{lem : lem in sec 2} implies that  
\begin{equation*}
\begin{aligned}
    \mathcal{N}(J)&\leq \frac{2^{r_1}}{A_{J}\cdot \prod_{v\in S_{F,\infty}} h_{\phi_v}(0)}\\
    &\times\left(|S_{\mathrm{one}}(\phi)| + |S_{\mathrm{id}}(\phi)|
+ |S_{\rm ell}(\phi)| + |S_{\rm hyp}(\phi)| + |S_{\rm par}(\phi)|
+ |S_{\rm Eis}(\phi)| + |S_{\rm Res}(\phi)|\right).
\end{aligned}
\end{equation*}
Therefore, we have to compute an upper bound for 
\begin{equation*}
    \frac{|S_{\mathrm{one}}(\phi)| + |S_{\mathrm{id}}(\phi)|
+ |S_{\rm ell}(\phi)| + |S_{\rm hyp}(\phi)| + |S_{\rm par}(\phi)|
+ |S_{\rm Eis}(\phi)| + |S_{\rm Res}(\phi)|}{\prod_{v\in S_{F,\infty}} h_{\phi_v}(0)}.
\end{equation*}

\section{Geometric side}\label{s : geometric side}
In this section, we compute the geometric side of the Arthur-Selberg trace formula.
For two non-negative functions $X$ and $Y$ of $n$ variables $t_1,\dots, t_n$, we write $X\ll_{t_1,\dots, t_m} Y$ if there exists a constant $C(t_1,\dots,t_m)$, dependent on $t_1,\dots, t_m$ such that 
\begin{equation*}
    C(t_1,\dots,t_m)\cdot X\leq Y.  
\end{equation*}
Here, $n$ and $m$ are positive integers with $m\leq n$. 
For simplicity, if there is a constant $C$ which is independent on $t_1,\dots, t_n$ such that 
\begin{equation*}
    C\cdot X\leq Y, 
\end{equation*}
then we write $X\ll Y$.
We follow the notation in Section \ref{s : Pre}.
Throughout this section, we fix an ideal $J$ of $\mathcal{O}_{F}$ and $e_v:=\val_v(J)$ for each $v\in S_{F,\fin}$.
Moreover, $\phi$ is a function on $\G(\mathbb{A}_{F})$ defined as in Lemma \ref{lem : lem of trace}.
\subsection{One-dimensional representations}
In \cite{GJ}, $S_{\mathrm{one}}(\phi)$ is defined by
\begin{equation*}
    S_{\mathrm{one}}(\phi):=\sum_{\substack{\chi^2=1\\ \chi \text{ idele class character }}} \int_{\bar{\rm G}(\mathbb{A})} \phi(g) \chi(\det g) dg.
\end{equation*}
Then, we compute $S_{\mathrm{one}}(\phi)$ in the following lemma.

\begin{lem}\label{lem : one-dim} 
Let $J$ be an ideal of $\mathcal{O}_{F}$. 
Then, we have 
\begin{equation*}
    S_{\mathrm{one}}(\phi)=A_{J}\cdot \prod_{v\in S_{F,\infty}} h_{\phi_{v}}\left(\frac{i}{2}\right).
\end{equation*}
\end{lem}
\begin{proof}
    Assume that $\chi:F^{\times}\backslash \mathbb{A}_{F}^{\times}\to \mathbb{C}^{\times}$ is an idele class character with $\chi^2=1$ such that 
\begin{equation}\label{eq 1}
    \int_{\bar{\G}(\mathbb{A}_{F})} \phi(g)\chi(\det g)dg\neq 0.
\end{equation}
Let $\chi=\otimes_{v\in S_F} \chi_v$, where $\chi_{v} : F_{v}^{\times}\to \mathbb{C}^{\times}$ is a character. 
Then, we have 
\begin{equation*}
    \int_{\bar{\G}(\mathbb{A}_{F})} \phi(g)\chi(\det g) dg=\left(\prod_{v\in S_{F,\fin}} \int_{\bar{\G}(F_v)} \phi_v(g)\chi_{v}(\det g) dg \right)\cdot \left(\prod_{v\in S_{F,\infty}}\int_{\bar{\G}(F_v)}\phi_{v}(g)\chi_{v}(\det g ) dg\right).
\end{equation*}

For each $v\in S_{F,\fin}$, we assume $g_v\in \Z(F_{v})\K_{v,\val_v(J)}$. 
Since $\phi_v$ is the characteristic function of $\Z(F_v)\K_{v,\val_v(J)}$, it follows that $\phi_v(g g_v)=\phi_v(g)$ for any $g\in \G(F_v)$.
Then, we have 
\begin{equation*}
    \begin{aligned}
    \int_{\bar{\G}(F_v)} \phi_v(g)\chi_{v}(\det g) dg&=\int_{\bar{\G}(F_v)} \phi_v(gg_v)\chi_{v}(\det (gg_v)) d(gg_v)\\
    &=\chi_v(\det g_v)\int_{\bar{\G}(F_v)} \phi_v(g)\chi_{v}(\det g) dg.
    \end{aligned}
\end{equation*}
Hence, we obtain that $\chi_v(\det g_v)=1$ for $g_v\in \Z(F_{v})\K_{v,\val_v(J)}$.
Since  the set of determinants of all elements in $\Z(F_{v})\K_{v,\val_v(J)}$ contains $F_v^{\times}$, it follows that $\chi_v=1$.

For each $v\in S_{F,\infty}$, if $F_{v}=\CC$ and $\chi_{v}^2=1$, then $\chi_{v}$ is the trivial character since the image of $\chi_{v}$ is connected. 
If $F_{v}=\RR$, then the characters $\chi_{v}:\RR^{\times}\to \CC^{\times}$ satisfying $\chi_{v}^2=1$ are only the trivial character and the sign character.
Since the support of $\phi_{v}$ is contained in $\GL_2^{+}(\RR)$ when $F_v=\RR$, we obtain that
\begin{equation*}
    \begin{aligned}
         S_{\mathrm{one}}(\phi)&= \left(\prod_{v\in S_{F,\fin}} \int_{\bar{\G}(F_v)}\phi_v(g)dg\right)\\
         &\times \prod_{F_v=\RR}\left(\int_{\bar{\G}(F_v)} \phi_{v}(g) dg + \int_{\bar{\G}(F_v)} \phi_{v}(g)\sgn(\det g) d_{v}g  \right)\cdot \prod_{F_v=\CC}\left(\int_{\bar{\G}(F_v)} \phi_{v}(g) dg\right)  \\
        &=2^{r_1}A_{J}\cdot \prod_{F_v=\RR}\left(\int_{\overline{\GL_2^{+}}(\RR)} \phi_{v}(g) dg\right)\cdot \prod_{F_v=\CC}\left(\int_{\bar{\G}(F_v)} \phi_{v}(g) dg\right),
    \end{aligned}
\end{equation*}
where $r_1$ denotes the number of archimedean places $v$ of $F$ satisfying $F_v=\RR$, and $\overline{\GL_2^{+}}(\RR)$ denotes the image of $\GL_2^{+}(\RR)$ modulo $\Z(\RR)$.
Following the proof of Lemma \ref{lem : lem of trace}, we obtain that if $F_v=\RR$, then 
\begin{equation*}
    \int_{\overline{\GL_2^{+}}(\RR)} \phi_{v}(g) dg=\frac{1}{2}h_{\phi_{v}}\left(\frac{i}{2}\right),
\end{equation*}
and if $F_v=\CC$, then 
\begin{equation*}
    \int_{\bar{\G}(F_v)} \phi_{v}(g) dg=h_{\phi_{v}}\left(\frac{i}{2}\right).
\end{equation*}
\end{proof}

\subsection{Identity contribution}
Let $I_2$ be the identity matrix of size $2$.
In \cite{GJ}, $S_{\mathrm{id}}(\phi)$ is defined by
\begin{equation}\label{eq : id}
     S_{\mathrm{id}}(\phi):=\vol(\bar{\G}(F)\backslash \bar{\G}(\mathbb{A}_{F}))\phi(I_2).
\end{equation}
Note that $\phi_{v}(I_2)=1$ for each $v\in S_{F,\fin}$. 
Then, we have the following lemma. 

\begin{lem}\label{lem : id}
For each $v\in S_{F,\infty}$, let $\epsilon_v:=[F_v:\RR]-1$. 
Then, we have
\begin{equation*}
   \phi(I_2)=\prod_{v\in S_{F,\infty}}\left(\frac{1+\epsilon_v}{(2-\epsilon_v)^{2(1-\epsilon_v)} \pi^{1+\epsilon_v}} \int_{-\infty}^\infty x^{1+\epsilon_v} h_v(x) \tanh^{1-\epsilon_v}(\pi x) \, dx\right).
\end{equation*}
\end{lem}
\begin{proof}
By \eqref{e:trans_phiv_hv_R}, when $F_v=\mathbb{R}$, we get
\begin{equation*}
\phi_v(I_2)
= \frac{1}{2\pi} \int_{-\infty}^{\infty} xh_{\phi_v}(x) \left(\int_0^\infty \frac{\sin(2\pi x t)}{\sinh(\pi t)} \, dt\right) \, dx
= \frac{1}{4\pi} \int_{-\infty}^\infty xh_{\phi_v}(x)\tanh(\pi x) dx
\end{equation*}
since by \cite[3.986.2]{GR},  
\begin{equation*}
    \int_0^\infty \frac{\sin(2\pi x t)}{\sinh(\pi t)} dt = \frac{1}{2}\tanh(\pi x).
\end{equation*}
When $F_v=\mathbb{C}$, we use  \eqref{e:trans_phi_vhathv_C}. 
Then, we get
\begin{equation*}
\phi_v(I_2)=\frac{2}{\pi^2}\int_{-\infty}^{\infty} x^2 h_{\phi_v}(x) dx.
\end{equation*}
\end{proof}

\subsection{Elliptic contribution}\label{ss : ell}
For $\gamma\in \bar{\G}(F)$, let $\bar{\G}_\gamma(F)$ (resp. $\bar{\G}_{\gamma}(\mathbb{A}_{F})$) be the centralizer of $\gamma$ in $\bar{\G}(F)$ (resp. $\bar{\G}(\mathbb{A}_{F})$). 
For an elliptic matrix $\gamma\in \bar{\G}(F)$, the conjugacy class $[\gamma]$ of $\gamma$ is defined by  
\begin{equation*}
[\gamma] := \left\{\delta^{-1}\gamma \delta\;:\; 
\delta\in \bar{\G}_{\gamma}(F)\backslash \bar{\G}(F)\right\}.
\end{equation*}
Let $\G_{\gamma}(F)$ (resp. $\G_{\gamma}(\mathbb{A}_F)$) be the centralizer of $\tilde{\gamma}$ in $\G(F)$ (resp. $\G(\mathbb{A}_{F})$), where $\tilde{\gamma}\in \G(F)$ is a representative of $[\gamma]$.
Let $\overline{\G_{\gamma}(F)}:=\Z(F)\backslash \G_{\gamma}(F)$ and $\overline{\G_{\gamma}(\mathbb{A}_{F})}:=\Z(\mathbb{A}_F)\backslash \G_{\gamma}(\mathbb{A}_F)$.
Then, by \cite[Lemma 3.4]{AK23}, we see that 
\[ [\bar{\G}_{\gamma}(F) : \overline{\G_{\gamma}(F)}]=\begin{cases}
     1 & \text{if } \tr \gamma\neq 0,\\
     2 &\text{otherwise.}
\end{cases} \]
In \cite[pp. 244]{GJ}, $S_{\mathrm{ell}}(\phi)$ is defined by
\begin{equation*}
    S_{\mathrm{ell}}(\phi):=\sum_{\text{elliptic }\gamma\in \bar{\G}(F)} \int_{\bar{\G}(F)\backslash \bar{\G}(\mathbb{A}_F)} \phi(g^{-1}\gamma g) dg.
\end{equation*}
Since $g^{-1}\delta^{-1}\gamma \delta g =g^{-1}\gamma g$ for $\delta\in \bar{\G}_{\gamma}(\mathbb{A}_{F})$ and $\phi=\prod_{v\in S_{F}}\phi_{v}$ is factorizable, we obtain that
\begin{equation*}
\begin{aligned}
S_{\mathrm{ell}}(\phi)&=\sum_{[\gamma]\text{ elliptic }} 
\int_{\bar{\G}(F) \backslash \bar{\G}(\mathbb{A}_{F})} 
\sum_{\delta\in \bar{\G}_{\gamma}(F)\backslash \bar{\G}(F)} \phi(g^{-1}\delta^{-1}\gamma \delta g) dg\\
&= \sum_{[\gamma]\text{ elliptic }} 
\int_{\bar{\G}_{\gamma}(F) \backslash \bar{\G}(\mathbb{A}_{F})} 
\phi(g^{-1}\gamma g) dg\\
&=\sum_{[\gamma]\text{ elliptic }} C_{\gamma}
\int_{\overline{\G_{\gamma}(F)}\backslash \bar{\G}(\mathbb{A}_{F})} 
\phi(g^{-1}\gamma g) dg\\
&=\sum_{[\gamma]\text{ elliptic }} 
\vol(\overline{\G_{\gamma}(F)} \backslash \overline{\G_{\gamma}(\mathbb{A}_{F})})\cdot S_{[\gamma]}(\phi),
\end{aligned}
\end{equation*}
where for each conjugacy class $[\gamma]$,  
\begin{equation}\label{eq 114}
S_{[\gamma]}(\phi):= \prod_{v\in S_{F,\fin}} \int_{\overline{\G_{\gamma}(F_v)} \backslash \bar{\G}(F_v)} \phi_v(g^{-1} \gamma g) dg \cdot \prod_{v\in S_{F,\infty}} \int_{\overline{\G_{\sigma_v(\gamma)}(F_v)} \backslash \bar{\G}(F_v)} \phi_v(g^{-1} \sigma_v(\gamma) g) dg.
\end{equation}
Here, $C_{\gamma}:=[\bar{\G}_{\gamma}(F) : \overline{\G_{\gamma}(F)}]^{-1}$ and $\sigma_{v} : F\to F_{v}$ is an embedding corresponding to $v\in S_{F,\infty}$.

Assume that $S_{[\gamma]}(\phi)\neq 0$.
Then, for $v\in S_{F,\fin}$, we have
\begin{equation}\label{eq 110}
    \int_{\overline{\G_{\gamma}(F_v)} \backslash \bar{\G}(F_v)} 
\phi_v(g^{-1} \gamma g) dg\neq 0.
\end{equation}
Since the centralizer $\G_{\tilde{\gamma}}(F_v)$ of $\tilde{\gamma}$ in $\G(F_v)$ contains $\Z(F_v)$, we see that 
\begin{equation*}
    \int_{\overline{\G_{\gamma}(F_v)}\backslash \bar{\G}(F_v)} \phi_v(g^{-1}\gamma g)dg=\int_{\G_{\tilde{\gamma}}(F_v)\backslash \G(F_v)} \phi_v(g^{-1}\tilde{\gamma}g)dg.
\end{equation*}
Note that the integral $\int_{\G_{\tilde{\gamma}}(F_v)\backslash \G(F_v)} \phi_v(g^{-1}\tilde{\gamma}g)dg$ is independent of the choice of the representative $\tilde{\gamma}$ of $[\gamma]$ since a Haar measure is invariant under the right multiplication.
By \eqref{eq 110}, there are $g_v\in \G(F_v)$ and $\alpha_v\in \Z(F_v)$ such that $\alpha_v(g_v^{-1}\tilde{\gamma} g_v)\in \G(\mathcal{O}_{F_v})$.
It follows that $\det(\tilde{\gamma})\in \mathcal{O}_{F_v}^{\times}F_v^{\times^2}$ for each $v\in S_{F,\fin}$.
Since the entries of $\tilde{\gamma}$ are in $F$, we have $\det(\tilde{\gamma})\in \mathcal{O}_{F}^{\times}F^{\times^2}$.
By multiplying a non-zero element in $F$, we can assume that the determinant of $\tilde{\gamma}$ is in $\mathcal{O}_{F}^{\times}$.
Since $\alpha_v(g_v^{-1}\tilde{\gamma} g_v)\in \G(\mathcal{O}_{F_v})$, it follows that $\alpha_v\in \Z(\mathcal{O}_{F_v})$ for every $v\in S_{F,\fin}$.
From this, we obtain that $g_v^{-1}\tilde{\gamma}g_v\in \G(\mathcal{O}_{F_v})$.
Thus, $\tr(\tilde{\gamma})\in \mathcal{O}_{F_v}$ for each $v\in S_{F,\fin}$, and then $\tr(\tilde{\gamma})\in \mathcal{O}_{F}$.
Note that $\tilde{\gamma}$ is similar to $\sm 0 & 1\\
-\det \tilde{\gamma}  & \tr \tilde{\gamma} \esm $ over $F$.
Therefore, there is a unique pair $(m,u)\in \mathcal{O}_{F}\times \left(\mathcal{O}_{F}^{\times}/\mathcal{O}_{F}^{\times^2}\right)$ such that 
$\sm 0 & 1 \\
-u & m \esm$ is in $[\gamma]$.

Let $U_{F}$ be the set of representatives of $\mathcal{O}_{F}^{\times}/ \mathcal{O}_{F}^{\times^2}$.
Dirichlet's unit theorem implies that $U_F$ is a finite set. 
Let $\gamma_{m,u}:=\sm 0 & 1\\
-u & m \esm$ with $(m,u)\in \mathcal{O}_{F}\times U_{F}$.
For a ring $R$ containing $\mathcal{O}_{F}$, the centralizer $\G_{\gamma_{m,u}}(R)$ of $\gamma_{m,u}$ in $\G(R)$ is equal to 
\begin{equation*}
    \G_{\gamma_{m,u}}(R)=\left\{\bpm a & -cu^{-1}\\
    c & a-cmu^{-1}\ebpm : a,c\in R \text{ and } a^2-acmu^{-1}+c^2u^{-1}\in R^{\times} \right\}.
\end{equation*}
Here, $R^{\times}$ denotes the set of units in $R$. 
Let $E:=F(\sqrt{m^2-4u})$.
Since $\gamma_{m,u}$ is an elliptic matrix over $F$, it follows that $[E:F]=2$.
We define an isomorphism $\psi:\G_{\gamma_{m,u}}(F)\to E^{\times}$ by
\begin{equation}\label{eq 40}
    \psi(aI_2-cu^{-1}\gamma_{m,u}):=a-cu^{-1}\left(\frac{m+\sqrt{m^2-4u}}{2}\right).
\end{equation}
The isomorphism \eqref{eq 40} deduces that
\begin{equation}\label{eq 41}
    \G_{\gamma_{m,u}}(\mathbb{A}_{F})\cong \mathbb{A}_{E}^{\times}. 
\end{equation}
Note that a Haar measure on $\mathbb{A}_{E}^{\times}$ is given by an isomorphism \eqref{eq 41}.
From this isomorphism \eqref{eq 41}, we compute the volume of $\overline{\G_{\gamma_{m,u}}(F)}\backslash \overline{\G_{\gamma_{m,u}}(\mathbb{A}_{F})}$ in the following lemma.

For a number field $K$, let $\mathbb{A}_{K}^{1}$ (resp. $\mathbb{A}_{K,\infty}^{1}$) be the subset of $\mathbb{A}_{K}$ (resp. $\mathbb{A}_{K,\infty}$) consisting of $(k_v)_{v\in S_{K}}$ (resp. $(k_v)_{v\in S_{K,\infty}}$) such that $\prod_{v\in S_{K}} |k_v|_{v} = 1$ (resp. $\prod_{v\in S_{K,\infty}} |k_v|_{v}=1$).

\begin{lem}\label{lem 4}
    Let $(m,u)\in \mathcal{O}_{F}\times \mathcal{O}_{F}^{\times}$ and $E:=F(\sqrt{m^2-4u})$.
    Let $\mathrm{cl}(E)$ be the class number of $E$.
    Let $i_{\infty}:E\to \mathbb{A}_{E,\infty}$ be an inclusion map defined by
    \begin{equation*}
        i_{\infty}(e):=(\sigma_v(e))_{v\in S_{E,\infty}}.
    \end{equation*}
    Assume that $[E:F]=2$. 
    Then, we have
    \[ \vol\left(\overline{\G_{\gamma_{m,u}}(F)}\backslash \overline{\G_{\gamma_{m,u}}(\mathbb{A}_{F})}\right)=\frac{2\cdot\mathrm{cl}(E)\cdot\vol\left(i_{\infty}(\mathcal{O}_{E}^{\times})\backslash \mathbb{A}^{1}_{E,\infty} \right)\cdot\prod_{w\in S_{E,\mathrm{fin}}}\vol(\mathcal{O}_{E_w}^{\times})}{\vol\left(F^{\times}\backslash \mathbb{A}_{F}^{1} \right)}. \]
\end{lem}

\begin{rmk}
    The exact volume of $\overline{\G_{\gamma_{m,u}}(F)}\backslash \overline{\G_{\gamma_{m,u}}(\mathbb{A}_{F})}$ is obtained by \cite[Theorem 4.16]{AK23}.  
    For the calculation of the orbital integral $\int_{\overline{\G_{\gamma}(F_v)} \backslash \bar{\G}(F_v)} \phi_v(g^{-1} \gamma g) dg$, we construct a Haar measure of each set in Section \ref{ss : ell}. 
    Note that the difference between these constructed Haar measures and the Haar measures in \cite{AK23} is a constant multiple.
\end{rmk}

\begin{proof}[Proof of Lemma \ref{lem 4}]
    By an isomorphism \eqref{eq 41}, $\overline{\G_{\gamma_{m,u}}(F)}\backslash \overline{\G_{\gamma_{m,u}}(\mathbb{A}_{F})}= \Z(\mathbb{A}_{F})\G_{\gamma_{m,u}}(F)\backslash \G_{\gamma_{m,u}}(\mathbb{A}_{F})$ can be identified $\left(\mathbb{A}_{F}^{\times}\cdot E^{\times}\right)\backslash \mathbb{A}_{E}^{\times}$.
By (4.31) in \cite{AK23}, we have
\begin{equation*}
    \vol\left(\left(\mathbb{A}_{F}^{\times}\cdot E^{\times}\right)\backslash \mathbb{A}_{E}^{\times} \right)=\frac{\vol\left(E^{\times}\backslash \mathbb{A}_{E}^{1} \right)}{\frac{1}{2}\cdot\vol\left(F^{\times}\backslash \mathbb{A}_{F}^{1} \right)}=\frac{2\cdot\vol\left(E^{\times}\backslash \mathbb{A}_{E}^{1} \right)}{\vol\left(F^{\times}\backslash \mathbb{A}_{F}^{1} \right)}.
\end{equation*}
Following the proof of \cite[Theorem 4.3.2]{T65}, we conclude that
\[ \vol\left(\overline{\G_{\gamma_{m,u}}(F)}\backslash \overline{\G_{\gamma_{m,u}}(\mathbb{A}_{F})}\right)=\frac{2\cdot\mathrm{cl}(E)\cdot\vol\left(i_{\infty}(\mathcal{O}_{E}^{\times})\backslash \mathbb{A}^{1}_{E,\infty} \right)\cdot\prod_{w\in S_{E,\mathrm{fin}}}\vol(\mathcal{O}_{E_w}^{\times})}{\vol\left(F^{\times}\backslash \mathbb{A}_{F}^{1} \right)}. \]
\end{proof}

To compute the volume of $\mathcal{O}^{\times}_{E_w}$ for $w\in S_{E,\fin}$, we define a Haar measure on $E_{w}^{\times}$ as follows.
For $v\in S_{F,\fin}$, a Haar measure $dg$ on $\G_{\gamma_{m,u}}(F_v)$ is defined by
\begin{equation}\label{eq 109}
    dg:=\frac{q_v^3}{(q_v-1)^2(q_v+1)}\cdot\frac{d_v a \, d_v c}{|a^2-acmu^{-1}+c^2u^{-1}|_{v}},
\end{equation}
where $g=\sm a& -cu^{-1}\\
c & a-cmu^{-1}\esm\in \G_{\gamma_{m,u}}(F_v)$.
Here, $d_v a$ is an additive measure on $F_{v}$ satisfying $\mathrm{Vol}(\mathcal{O}_{F_v})=1$.
If there exists an element $x\in F_{v}$ such that $x^2 = m^2 - 4u$, then the prime ideal $\mathfrak{p}_{v}$ corresponding to $v$ splits completely in $E$.
Thus, there are exactly two places $w\in S_{E,\fin}$ lying over $v$, and then there is an isomorphism $\psi_{v}:\G_{\gamma_{m,u}}(F_v)\to \prod_{w\mid v} E_{w}^{\times}$ defined by
\begin{equation}\label{eq 88}
    \psi_{v}(aI_2-cu^{-1}\gamma_{m,u}):=\left(a-cu^{-1}\left(\frac{m+x_1}{2} \right), a-cu^{-1}\left(\frac{m+x_2}{2} \right) \right),
\end{equation}
where $x_1$ and $x_2$ are roots of $x^2=m^2-4u$ in $F_v$. 
If there is no $x\in F_{v}$ satisfying $x^2=m^2-4u$, then there is a unique $w\in S_{F,\fin}$ lying over $v$. 
Thus, there is an isomorphism $\psi_{v}:\G_{\gamma_{m,u}}(F_v)\to E_{w}^{\times}$ defined by
\begin{equation}\label{eq 89}
    \psi_{v}(aI_2-cu^{-1}\gamma_{m,u}):=a-cu^{-1}\left(\frac{m+\sqrt{m^2-4u}}{2}\right).
\end{equation}
Hence, a Haar measure on $\prod_{w\mid v} E_{w}^{\times}$ is given by the isomorphism $\psi_{v}$ for each $v\in S_{F,\fin}$.
Thus, the volume of $\prod_{w\mid v}\mathcal{O}_{E_w}^{\times}$ is equal to the volume of the preimage of it under $\psi_{v}$, i.e.,
\begin{equation*}
    \vol\left(\prod_{w\mid v}\mathcal{O}_{E_w}^{\times} \right)=\vol\left(\psi_{v}^{-1}\left( \prod_{w\mid v}\mathcal{O}_{E_w}^{\times}\right) \right).
\end{equation*}
The following lemma provides an upper bound for $\vol(\prod_{w\mid v}\mathcal{O}_{E_w}^{\times})$ in terms of $q_v$ when $v\in S_{F,\fin}$.

\begin{lem}\label{lem : ell-1}
    Assume that $v\in S_{F,\fin}$ and that $(m,u)\in \mathcal{O}_{F_v}\times \mathcal{O}_{F_v}^{\times}$.
    Let $N_{v}$ be defined by the number of $t\in \mathcal{O}_{F_v}/ \varpi_{v}\mathcal{O}_{F_v}$ satisfying $t^2-mt+u=0$.
    Let $f_{v}:=\val_{v}(m^2-4u)$.
    If $f_v>0$, then 
    \begin{equation*}
        \mathrm{Vol}\left(\prod_{w\mid v}\mathcal{O}_{E_w}^{\times}\right)\leq \frac{q_v^{3+2\cdot[\frac{f_v+1}{2}]}}{(q_v-1)^2(q_v+1)},
    \end{equation*}
    and if $f_v=0$, then 
    \begin{equation*}
        \mathrm{Vol}\left(\prod_{w\mid v}\mathcal{O}_{E_w}^{\times}\right)\leq \frac{q_v(q_v+1-N_v)}{q_v^2-1}.
    \end{equation*}
\end{lem}
\begin{proof}
By the definition of $\psi_{v}$, we have  
    \begin{equation*}
    \begin{aligned}
        \psi_{v}^{-1}&\left(\prod_{w\mid v} \mathcal{O}_{E_w}^{\times} \right)\\
        &\subset \left\{\bpm a & -cu^{-1} \\
        c & a-cmu^{-1}\ebpm :\min\left(\val_{v}(a),\val_{v}(c) \right)\geq -\frac{f_v}{2} \text{ and } \val_v(a^2-acmu^{-1}+c^2u^{-1})=0 \right\}.
    \end{aligned}
    \end{equation*}
Recall that the Haar measure on $\G_{\gamma}(F_v)$ is given by \eqref{eq 109}.
If $f_{v}>0$, then we have 
\begin{equation*}
\begin{aligned}
    \vol\left(\psi_{v}^{-1}\left(\prod_{w\mid v} \mathcal{O}_{E_w}^{\times} \right) \right)&\leq \frac{q_{v}^3}{(q_v-1)^2(q_v+1)}\int_{\varpi_{v}^{-[\frac{f_v+1}{2}]}\mathcal{O}_{F_v}}\int_{\varpi_{v}^{-[\frac{f_v+1}{2}]}\mathcal{O}_{F_v}} d_v a \, d_v c\\
    &= \frac{q_v^{3}}{(q_v-1)^2(q_v+1)}\cdot q_v^{2\cdot[\frac{f_v+1}{2}]}.
\end{aligned}
\end{equation*}

Assume that $f_v=0$.
Since $a^2-acmu^{-1}+c^2u^{-1}=c^2u^{-2}((au/c)^2-m(au/c)+u)$, it follows that     
\begin{equation*}
\begin{aligned}
    &\{(a,c)\in \mathcal{O}_{F_v}\times \mathcal{O}_{F_v} : \val_v(a^2-acmu^{-1}+c^2u^{-1})=0\}\\
    &=\mathcal{O}_{F_v}\times \mathcal{O}_{F_v} - \{(ctu^{-1}+a_1,c) : c\in \mathcal{O}_{F_v}^{\times}, a_1\in \varpi_{v}\mathcal{O}_{F_v} \text{ and } \varpi_v\mid t^2-mt+u\}-\varpi_{v}\mathcal{O}_{F_v} \times \varpi_{v}\mathcal{O}_{F_v} .
\end{aligned}
\end{equation*}
Therefore, we conclude that
\begin{equation*}
\begin{aligned}
    \vol\left(\prod_{w\mid v}\mathcal{O}_{E_w}^{\times}\right)&\leq \frac{q_v^3}{(q_v-1)^2(q_v+1)}\left(1-\frac{N_v (q_v-1)}{q_v^2}-\frac{1}{q_v^2} \right)\\
    &=\frac{q_v(q_v+1-N_v)}{q_v^2-1}.
\end{aligned}
\end{equation*}
\end{proof}

Since $\G_{\gamma_{m,u}}(F_v)$ contains $\Z(F_v)$, it follows that $\overline{\G_{\gamma_{m,u}}(F_v)}\backslash \bar{\G}(F_v)$ is isomorphic to $\G_{\gamma_{m,u}}(F_v)\backslash \G(F_v)$. 
In the following lemma, we compute the set of representatives of $\G_{\gamma_{m,u}}(F_v)\backslash \G(F_v)$.

\begin{lem}\label{lem : ell-2}
    Assume that  $v\in S_{F,\fin}$ and that $(m,u)\in \mathcal{O}_{F}\times \mathcal{O}_{F}^{\times}$ satisfies $m^2-4u\not\in F^{2}$.
    If there is $\alpha\in F_v$ such that $\alpha^2-m\alpha+u=0$, then  
    \begin{equation*}
        \left\{\bpm 1 & u^{-1}\alpha_i\\
        \alpha_i & 1\ebpm :i\in \{1,2\} \right\} \cup \left\{\bpm 1 & x\\
        \alpha_i  & 0\ebpm : x\in F_{v}^{\times} \text{ and }i\in \{1,2\} \right\} \cup \left\{\bpm 1 & x\\
        0 & y \ebpm  : x\in F_v \text{ and }y\in F_{v}^{\times}\right\}
    \end{equation*}
    is a set of representatives of $\G_{\gamma_{m,u}}(F_v)\backslash \G(F_v)$. 
    Here, $\alpha_1$ and $\alpha_2$ are roots of $x^2-mx+u=0$ in $F_v$. 
    If there is no $\alpha\in F_v$ satisfying $\alpha^2-m\alpha+u=0$, then
    \begin{equation*}
         \left\{\bpm 1 & x\\
        0 & y \ebpm  : x\in F_v \text{ and }y\in F_{v}^{\times}\right\}
    \end{equation*} 
    is a set of representatives of $\G_{\gamma_{m,u}}(F_v)\backslash \G(F_v)$.
\end{lem}
\begin{proof}
    For $\sm a & b\\
    c & d\esm\in \G(F_v)$ with $a^2-macu^{-1}+c^2u^{-1}\neq 0$, we have 
    \begin{equation}\label{eq 53}
        \bpm a & b\\
        c & d\ebpm = \bpm a & -cu^{-1}\\
        c & a-cmu^{-1}\ebpm \bpm 1 & \frac{ab+cdu^{-1}-bcmu^{-1}}{a^2-acmu^{-1}+c^2u^{-1}}\\
        0 & \frac{ad-bc}{a^2-acmu^{-1}+c^2u^{-1}}\ebpm. 
    \end{equation}
    Moreover, if $(x_1,y_1)\neq (x_2,y_2)$, then 
    \begin{equation*}
        \bpm 1 & x_1\\
        0 & y_1\ebpm \bpm 1 & x_2\\
        0 & y_2\ebpm^{-1} =\frac{1}{y_2}\bpm y_2 &x_1-x_2\\
        0 & y_1\ebpm \not\in \G_{\gamma_{m,u}}(F_v).
    \end{equation*}
    Thus, if there is no root of $x^2-mx+u$ in $F_v$, then
    \begin{equation*}
         \left\{\bpm 1 & x\\
        0 & y \ebpm  : x\in F_v \text{ and }y\in F_{v}^{\times}\right\}
    \end{equation*} 
    is a set of representatives of $\G_{\gamma_{m,u}}(F_v)\backslash \G(F_v)$.

    Assume that there is $\alpha\in F_{v}$ satisfying $\alpha^2-m\alpha+u=0$.  
    Let $\alpha_1$ and $\alpha_2$ be roots of $x^2-mx+u=0$ in $F_v$.
    Note that if $a^2-acmu^{-1}+c^2u^{-1}=0$, then $c=a\alpha_i$ for $i\in \{1,2\}$.
    Thus, we assume that $c=a\alpha$ with $\alpha \in \{\alpha_1,\alpha_2\}$.
    If $b\neq du^{-1}\alpha$, then 
    \begin{equation}\label{eq 54}
    \begin{aligned}
        \bpm a & b\\
        a\alpha & d\ebpm &= \bpm ab(b-du^{-1}\alpha)^{-1} & -adu^{-1}(b-du^{-1}\alpha)^{-1} \\
        ad(b-du^{-1}\alpha)^{-1} & a(b-mu^{-1}d)(b-du^{-1}\alpha)^{-1}\ebpm
        \bpm 1 & a^{-1}(b-du^{-1}\alpha)\\
        \alpha & 0\ebpm \\
        &\in \G_{\gamma_{m,u}}(F_v) \bpm 1 & a^{-1}(b-du^{-1}\alpha)\\
        \alpha & 0\ebpm. 
    \end{aligned}
    \end{equation}
    Note that for any $a$ and $d$ in $F_{v}^{\times}$, we get
    \begin{equation*}
        \begin{aligned}
            a-du^{-1}\alpha^2-m(a-d)u^{-1}\alpha&=a-du^{-1}\alpha^2-m(a-d)u^{-1}\alpha+(d-a)u^{-1}(\alpha^2-m\alpha+u)\\
            &=d-au^{-1}\alpha^2
        \end{aligned}
    \end{equation*}
    and
    \begin{equation}\label{eq 55}
        \bpm a & du^{-1}\alpha\\
        a\alpha & d\ebpm = \frac{1}{1-u^{-1}\alpha^2} \bpm a-du^{-1}\alpha^2 & -(a-d)u^{-1}\alpha\\
        (a-d)\alpha & -au^{-1}\alpha^2+d
        \ebpm \bpm 1 & u^{-1}\alpha\\
        \alpha & 1\ebpm
        \in \G_{\gamma}(F_v)\bpm 1 & u^{-1}\alpha\\
        \alpha & 1\ebpm. 
    \end{equation}

    Let $S$ be the collection of $\sm 1 & u^{-1}\alpha_i \\
    \alpha_i & 1\esm$, $\sm 1 & y\\
    \alpha_i & 0\esm$, and $\sm 1 & x \\
        0 & y\esm$, where $i\in \{1,2\}$, $x\in F_{v}$ and $y\in F_v^{\times}$.
    By \eqref{eq 53}, \eqref{eq 54} and \eqref{eq 55}, we see that for any $g\in \G(F_v)$, there is $g_0\in S$ such that 
    \begin{equation}\label{eq 56}
        g\in \G_{\gamma_{m,u}}(F_v)\cdot g_0.
    \end{equation}
    Moreover, by a calculation, we deduce that there is a unique $g_0$ in $S$ satisfying \eqref{eq 56}.
    Therefore, the set $S$ is a set of representatives of $\G_{\gamma_{m,u}}(F_v)\backslash \G(F_v)$.
\end{proof}

Since $\G_{\gamma_{m,u}}(F_v)$ is an abelian group and $\G(F_v)$ is a unimodular group, there is the unique right $\G(F_v)$-invariant measure on the homogeneous space $\G_{\gamma_{m,u}}(F_v)\backslash \G(F_v)$ up to constant multiplication.
Thus, Haar measures on $\G(F_v)$ and $\G_{\gamma_{m,u}}(F_v)$ naturally induce a right $\G(F_v)$-invariant Haar measure on $\G_{\gamma_{m,u}}(F_v)\backslash \G(F_v)$ as follows.
By using Lemma \ref{lem : ell-2}, a measure $dg$ on $\G_{\gamma_{m,u}}(F_v)\backslash \G(F_v)$ is defined by
\begin{equation}\label{eq 111}
    dg:=\begin{cases}
        \frac{d_vx d_vy}{|y|_v^2} \quad &\text{if }g=\G_{\gamma_{m,u}}(F_v) \sm 1 & x \\
0 & y\esm,\\
0 \quad &\text{otherwise,}
    \end{cases}
\end{equation}
where $x\in F_{v}$ and $y\in F_{v}^{\times}$.
We have to check whether Haar measures on $\G(F_v)$, $\G_{\gamma_{m,u}}(F_v)$, and $\G_{\gamma_{m,u}}(F_v)\backslash \G(F_v)$ are compatible. 
Let $h_1:=aI_2-cu^{-1}\gamma_{m,u}\in \G_{\gamma_{m,u}}(F_v)$, and let $h_2$ be an element in the set of representatives of $\G_{\gamma_{m,u}}(F_v)\backslash \G(F_v)$ as given in Lemma \ref{lem : ell-2}.
Let $g:=h_1h_2\in \G(F_v)$. 
Assume that $h_2=\sm  1 & x\\
0 & y\esm$ with $x\in F_v$ and $y\in F_v^{\times}$.
By the definitions of Haar measures on $\G_{\gamma_{m,u}}(F_v)$ and $\G_{\gamma_{m,u}}(F_v)\backslash \G(F_v)$, we have 
\[ dh_1\, dh_2 = \frac{q_v^3}{(q_v-1)^2(q_v+1)}\cdot \frac{d_v a \, d_v c}{|a^2-acmu^{-1}+c^2u^{-1}|_{v}}\cdot \frac{d_v x \, d_v y}{|y|_v^2}. \]
In another way, since $g=\sm  a & ax-cu^{-1}y \\
c & ay+cx-cmu^{-1}y \esm$, it follows that 
\begin{equation*}
    \begin{aligned}
        dg=\frac{q_v^3}{(q_v-1)^2(q_v+1)}\cdot \frac{d_v a \, d_v c \, d_v(ax-cu^{-1}y) \, d_v(ay+cx-cmu^{-1}y)}{|y(a^2-acmu^{-1}+c^2u^{-1})|_{v}^2}=dh_1\, dh_2.
    \end{aligned}
\end{equation*}
Similarly, when $h_2$ does not have the form $\sm 1 & x\\
0 & y\esm$, we can obtain $dg=dh_1\,dh_2$.

In the following lemma, we obtain an upper bound for the integral
\begin{equation*}
    \int_{\overline{\G_{\gamma_{m,u}}(F_v)} \backslash \bar{\G}(F_v)} 
\phi_v(g^{-1} \gamma_{m,u} g) dg.
\end{equation*}
\begin{lem}\label{lem : ell-3}
Assume that $v\in S_{F,\fin}$. 
Let $m$, $u$, $f_v$ and $N_v$ be defined as in Lemma \ref{lem : ell-1}.
Then, we have
\begin{equation*}
    \left|\int_{\overline{\G_{\gamma_{m,u}}(F_v)} \backslash \bar{\G}(F_v)} 
\phi_v(g^{-1} \gamma_{m,u} g) dg\right|\leq \begin{cases}
    \frac{1}{q_v}(q_v-1+N_v) \quad &\text{if } e_v=f_v=0,\\
    4\cdot q_v^{-e_v+\frac{f_v}{2}} \quad &\text{otherwise}.
\end{cases}
\end{equation*}
Here, the measure on $\overline{\G_{\gamma_{m,u}}(F_v)} \backslash \bar{\G}(F_v)$ is given by \eqref{eq 111}.
\end{lem}
\begin{proof}
Assume that $g=\G_{\gamma_{m,u}}(F_v)\sm 1 & x\\
0 & y\esm\in \G_{\gamma_{m,u}}(F_v) \backslash \G(F_v)$ with $x\in F_v$ and $y\in F_v^{\times}$ satisfies $\phi_v(g^{-1}\gamma_{m,u} g)\neq 0$. 
Then, we get
    \begin{equation}\label{eq 57}
        \begin{aligned}
            g^{-1}\gamma_{m,u} g=\bpm 1 & x  \\
            0 & y \ebpm ^{-1} \bpm 0 & 1\\
            -u & m\ebpm  \bpm 1 & x  \\
            0 & y \ebpm=\frac{1}{y}\bpm xu & x^2u-mxy+y^2 \\
            -u & -ux+my\ebpm\in \Z(F_{v})\cdot \K_{v,e_v}.
        \end{aligned}
    \end{equation}
    Since $\det(g^{-1}\gamma_{m,u} g)\in \mathcal{O}_{F_v}^{\times}$, it follows that $g^{-1}\gamma_{m,u} g\in \K_{v,e_v}$.
    By \eqref{eq 57}, if $\phi_v(g^{-1}\gamma_{m,u} g)\neq 0$, then we obtain that $(x,y)$ is in 
    \begin{equation}\label{eq 58}
    \left\{(x,y)\in F\times F^{\times} : y^{-1}\in \varpi_{v}^{e_v}\mathcal{O}_{F_v}, \quad xy^{-1}\in \mathcal{O}_{F_v}, \text{ and } (x^2u-mxy+y^2)y^{-1}\in \mathcal{O}_{F_v} \right\}. 
    \end{equation}
    If $e_v=0$ and $y\in \mathcal{O}_{F_v}^{\times}$, then any $x\in \mathcal{O}_{F_v}$ satisfies the condition \eqref{eq 58}. 
    If $\val_{v}(y)<0$ and $\val_{v}(x)>\val_{v}(y)$, then we have 
    \begin{equation*}
        \val_{v}\left((x^2u-mxy+y^2)y^{-1}\right)=\val_{v}(y)<0.
    \end{equation*}
    Thus, if $(x,y)$ is in the set \eqref{eq 58} with $\val_v(y)<0$, then $\val_v(x)=\val_v(y)$.
    From this, we obtain that if $e_v=0$, then the set \eqref{eq 58} is equal to 
    \begin{equation*}
        \{(x,y)\in \mathcal{O}_{F_v}\times \mathcal{O}_{F_v}^{\times}\} \cup \{(\varpi_{v}^{-\alpha}x_0, \varpi_{v}^{-\alpha}y_0) : \alpha>0, \quad  x_0, y_0\in \mathcal{O}_{F_v}^{\times} \text{ and } x_0^2u-mx_0y_0+y_0^2\in \varpi_{v}^{\alpha}\mathcal{O}_{F_v}\},
    \end{equation*}
    and if $e_v>0$, then the set \eqref{eq 58} is equal to
    \begin{equation*}
        \{(\varpi_{v}^{-\alpha}x_0, \varpi_{v}^{-\alpha}y_0) : \alpha\geq e_v, \quad  x_0, y_0\in \mathcal{O}_{F_v}^{\times} \text{ and } x_0^2u-mx_0y_0+y_0^2\in \varpi_{v}^{\alpha}\mathcal{O}_{F_v}\}.
    \end{equation*}
    
    For a positive integer $\alpha$, let
    \begin{equation*}
        S_{v,\alpha,m,u}:=\{t\in \mathcal{O}_{F_v}/\varpi_{v}^{\alpha}\mathcal{O}_{F_v} : t^2-mt+u=0 \}.
    \end{equation*}
    Then, we have 
    \begin{equation*}
    \begin{aligned}
     &\{(\varpi_{v}^{-\alpha}x_0, \varpi_{v}^{-\alpha}y_0) : x_0, y_0\in \mathcal{O}_{F_v}^{\times} \text{ and } x_0^2u-mx_0y_0+y_0^2\in \varpi_{v}^{\alpha}\mathcal{O}_{F_v}\}\\
     &= \{(\varpi_{v}^{-\alpha}x_0, \varpi_{v}^{-\alpha}x_0t_{\alpha}+y) : x_0\in \mathcal{O}_{F_v}^{\times}, y\in \mathcal{O}_{F_v} \text{ and }t_{\alpha}\in S_{v,\alpha,m,u}\}\\
     &=\{(\varpi_{v}^{-\alpha}t_{\alpha}^{-1}(x_0t_{\alpha}+\varpi_v^{\alpha}y)-t_{\alpha}^{-1}y, \varpi_{v}^{-\alpha}(x_0 t_{\alpha} + \varpi_{v}^{\alpha}y)) : x_0\in \mathcal{O}_{F_v}^{\times}, y\in \mathcal{O}_{F_v}\text{ and }t_{\alpha}\in S_{v,\alpha,m,u} \}\\
     &=\{(\varpi_{v}^{-\alpha}t_{\alpha}^{-1}y_1+x_1, \varpi_{v}^{-\alpha} y_1) : x_1\in \mathcal{O}_{F_v}, y_1\in \mathcal{O}_{F_v}^{\times}\text{ and }t_{\alpha}\in S_{v,\alpha,m,u}\}.
    \end{aligned}
    \end{equation*}
    Hence, the set of $g=\sm 1& x\\
    0 & y\esm$ satisfying $\phi_{v}(g^{-1}\gamma_{m,u} g)\neq 0$ is equal to  
    \begin{equation}\label{eq 59}
        \begin{aligned}
           &\left\{\bpm 1 & \varpi_{v}^{-\alpha}t_{\alpha}^{-1}y_1+x_1 \\
        0 & \varpi_{v}^{-\alpha}y_1 \ebpm : \alpha\geq 1, x_1\in \mathcal{O}_{F_v}, y_1\in \mathcal{O}_{F_v}^{\times} \text{ and } t_{\alpha}\in S_{v,\alpha,m,u} \right\} \\
        &\cup \left\{\bpm 1 & x\\
        0 & y\ebpm : x\in \mathcal{O}_{F_v} \text{ and } y\in \mathcal{O}_{F_v}^{\times}  \right\}
        \end{aligned}
    \end{equation}
    if $e_v=0$, and 
    \begin{equation}\label{eq 7}
        \left\{\bpm 1 & \varpi_{v}^{-\alpha}t_{\alpha}^{-1}y_1+x_1 \\
        0 & \varpi_{v}^{-\alpha}y_1 \ebpm : \alpha\geq e_v, x_1\in \mathcal{O}_{F_v}, y_1\in \mathcal{O}_{F_v}^{\times} \text{ and } t_{\alpha}\in S_{v,\alpha,m,u} \right\}
    \end{equation}
    if $e_v>0$.
    
    For $t_{\alpha}\in S_{v,\alpha,m,u}$, the volume of
    \begin{equation*}
        \left\{\bpm 1 & \varpi_{v}^{-\alpha}t_{\alpha}^{-1}y_1+x_1 \\
        0 & \varpi_{v}^{-\alpha}y_1\ebpm : x_1\in \mathcal{O}_{F_v} \text{ and } y_1\in \mathcal{O}_{F_v}^{\times} \right\}
    \end{equation*}
    is computed by 
    \begin{equation}\label{eq 11}
        \begin{aligned}
            \int_{\varpi_{v}^{-\alpha}\mathcal{O}_{F_v}^{\times}}\int_{\mathcal{O}_{F_v}} \frac{d_v x d_v y}{|y|^{2}_{v}}=q_v^{-\alpha}\int_{\mathcal{O}_{F_v}^{\times}}\int_{\mathcal{O}_{F_v}} d_{v}x d_v y = q_v^{-\alpha-1}(q_v-1).
        \end{aligned}
    \end{equation}
    Similarly, we get 
    \begin{equation*}
    \vol\left(\left\{\bpm 1 & x\\
        0 & y\ebpm : x\in \mathcal{O}_{F_v} \text{ and } y\in \mathcal{O}_{F_v}^{\times}  \right\}\right)=q_v^{-1}(q_v-1).       
    \end{equation*}

    To complete the proof of Lemma \ref{lem : ell-3}, we compute the number of elements in $S_{v,\alpha,m,u}$.
    Assume that $f_v=0$ and $\alpha\geq 1$. 
    For each $t\in S_{v,\alpha,m,u}$, there is $t_0\in \mathcal{O}_{F_v}$ such that the reduction of $t_0$ modulo $\varpi_{v}^{\alpha}$ is equal to $t$. 
    Since $t\in S_{v,\alpha,m,u}$ and $f_v=0$, there is $s\in \mathcal{O}_{F_v}$ such that
    \begin{equation*}
        t_0^2-mt_0+u=\varpi_{v}^{\alpha}s,
    \end{equation*}
    and then $2t_0-m\in \mathcal{O}_{F_v}^{\times}$.
    Note that for $a\in \mathcal{O}_{F_v}$, we have 
    \begin{equation*}
        (t_0+a\varpi_v^{\alpha})^{2}-m(t_0+a\varpi_{v}^{\alpha})+u=\varpi_{v}^{\alpha}(s+a(2t_0-m))+\varpi_{v}^{2\alpha}a^2.
    \end{equation*}
    Since $2t_0-m$ is a unit in $\mathcal{O}_{F_v}$, the reduction of $a$ modulo $\varpi_{v}$ is equal to the reduction of $-(2t_0-m)^{-1}s$ modulo $\varpi_{v}$ if and only if the reduction of $t_0+a\varpi_{v}^{\alpha}$ modulo $\varpi_{v}^{\alpha+1}$ is in $S_{v,\alpha+1,m,u}$.
    Thus, if $f_v=0$ and $\alpha\geq 1$, then $\# S_{v,\alpha,m,u}=\# S_{v,\alpha+1,m,u}$.

    If $e_v=f_v=0$, then $\#S_{v,\alpha,m,u}=\#S_{v,1,m,u}=N_{v}$ for $\alpha\geq 1$. 
    Thus, we have
    \begin{equation*}
    \begin{aligned}
        \left|\int_{\overline{\G_{\gamma_{m,u}}(F_v)} \backslash \bar{\G}(F_v)} 
    \phi_v(g^{-1} \gamma_{m,u} g) dg\right|&\leq q_{v}^{-1}(q_v-1)+\sum_{\alpha\geq 1}N_{v}q_{v}^{-\alpha-1}(q_v-1)\\
    &=\frac{q_v-1+N_v}{q_v}.
    \end{aligned}
    \end{equation*}
    For any $v\in S_{F,\fin}$ and $\alpha\geq 0$, if $t\in S_{v,\alpha,m,u}$, then 
    \begin{equation*}
        \left(t-\frac{m}{2}\right)^2-\frac{m^2-4u}{4}\in \varpi_{v}^{\alpha}\mathcal{O}_{F_v}.
    \end{equation*}
    This implies that 
    \begin{equation*}
        \#S_{v,\alpha,m,u}\leq 4q_{v}^{\frac{f_v}{2}}.
    \end{equation*}
    Therefore, if $e_v+f_v>0$, then we conclude that 
    \begin{equation*}
        \begin{aligned}
            \left|\int_{\overline{\G_{\gamma_{m,u}}(F_v)} \backslash \bar{\G}(F_v)} 
    \phi_v(g^{-1} \gamma_{m,u} g) dg\right|&\leq \sum_{\alpha\geq e_v} \# S_{v,\alpha,m,u} \cdot q_v^{-\alpha-1}(q_v-1)\\
    &\leq 4q_v^{-e_v+\frac{f_v}{2}}.
        \end{aligned}
    \end{equation*}
\end{proof}
Now, we assume that $v\in S_{F,\infty}$.
Suppose that there are two roots $\alpha_1$ and $\alpha_2$ of $x^2-\sigma_v(m)x+\sigma_v(u)=0$ in $F_v$. 
Since $\sm \alpha_1 & 0\\
0 & \alpha_2 \esm$ is similar to $\sm 0 & 1\\
-\sigma_v(u) & \sigma_v(m)\esm$ over $F_v$, we take $\gamma=\sm \alpha_1 & 0\\
0 & \alpha_2 \esm$ as a representative of the conjugacy class $[\gamma_{m,u}]$. 
Then, we have 
\begin{equation*}
    \G_{\gamma}(F_v)=\left\{\bpm z_1 & 0\\
    0 & z_2 \ebpm  : z_1, z_2\in F_v^{\times}\right\}
\end{equation*}
and 
\begin{equation*}
    \G_{\gamma}(F_v)\backslash \G(F_v)=\left\{\bpm 1 & x\\
    0 & 1\ebpm \overline{\kappa} : x\in F_{v} \text{ and } \overline{\kappa}\in \left(\Z(F_v)\cap \K_v \right)\backslash \K_{v} \right\}.
\end{equation*}
Thus, there is an isomorphism 
$\psi_{v} : \G_{\gamma}(F_v) \to \prod_{w\mid v}E_{w}^{\times}$ given by 
\begin{equation}\label{eq 90}
    \psi_{v}\left(\bpm z_1 & 0\\
    0 & z_2\ebpm \right):=(z_1,z_2).
\end{equation}
An isomorphism \eqref{eq 90} induces a Haar measure $dg$ on $\prod_{w\mid v}E_{w}^{\times}$ defined by
\begin{equation}\label{eq 92}
    dg:=d_{w}^{\times} z_1 \, d_{w}^{\times}z_2, 
\end{equation}
where $g=(z_1,z_2)\in \prod_{w\mid v} E_{w}^{\times}$. 
In this case, $F_v$ is isomorphic to $E_{w}$ for $w\mid v$.
Also, a Haar measure $dg$ on $\G_{\gamma}(F_v)\backslash \G(F_v)$ is defined by
\begin{equation}\label{eq 126}
    dg:=d_v x \, d_{v}\overline{\kappa},
\end{equation}
where $g=\sm 1 & x\\
0 & 1\esm \overline{\kappa}\in \G_{\gamma}(F_v)\backslash \G(F_v)$. 
Here, $d_{v}\overline{\kappa}$ is a Haar measure on $(\Z(F_v)\cap \K_{v}) \backslash \K_{v}$.

Assume that there is no root of $x^2-\sigma_v(m)x + \sigma_v(u)=0$ in $F_v$.
This only occurs when $F_v=\RR$ and $\sigma_v(m^2-4u)<0$. 
Then, $\gamma=\sm r & 0\\
0 & r\esm  \sm \cos \theta & -\sin \theta\\
\sin \theta & \cos \theta\esm$ is similar to $\sm 0 & 1 \\
-\sigma_v(u) & \sigma_v(m)\esm$ over $\RR$, where 
\begin{equation}\label{eq 96}
    r:=\sqrt{\sigma_v(u)}, \text{ and } \theta:=\tan^{-1}\left(\frac{\sqrt{\sigma_v(4u-m^2)}}{\sigma_v(m)} \right).
\end{equation}
Thus, we obtain that
\begin{equation*}
    \G_{\gamma}(F_v)=\left\{\bpm r & 0\\
    0 & r\ebpm \cdot \kappa : r>0 \text{ and } \kappa\in \mathrm{SO}(2)\right\}
\end{equation*}
and that 
\begin{equation*}
    \G_{\gamma}(F_v)\backslash \G(F_v)=\left\{\bpm y & 0\\
    0 & 1\ebpm \bpm 1 & x\\
    0 & 1\ebpm  : y\in \RR^{\times} \text{ and } x\in \RR \right\}.
\end{equation*}
By the Iwasawa decomposition, for $g\in \G(F_v)$, we have
\begin{equation*}
    g=\bpm r_1 & 0\\
    0 & r_1\ebpm \bpm y_1 & 0\\
    0 & 1\ebpm \bpm 1 & x_1\\
    0 & 1\ebpm \kappa_1= \bpm r_2 & 0\\
    0 & r_2 \ebpm \kappa_2 \bpm y_2 & 0\\
    0 & 1\ebpm \bpm 1 & x_2 \\
    0 & 1\ebpm.
\end{equation*}
Since a Haar measure on $\G(\RR)$ is a unique up to constant multiple, there is a constant $C>0$ such that 
\begin{equation}\label{eq 112}
    dg=d^{\times}_{v} r_1 \, d_{v}^{\times}y_1 \,  d_{v} x_1\,  d_v \kappa_1 = C\cdot d^{\times}_{v} r_2 \, d_{v} \kappa_2 \, d_{v}^{\times}y_2 \, d_{v} x_2.
\end{equation}
Note that there is an isomorphism $\psi_{v} : \G_{\gamma}(F_v) \to \prod_{w\mid v}E_{w}^{\times}$ given by
\begin{equation}\label{eq 91}
    \psi_{v}\left(\bpm r & 0\\
    0 & r \ebpm \cdot \kappa_{\theta} \right):=r\cdot e^{i\theta}.
\end{equation}
An isomorphism \eqref{eq 91} induces a Haar measure $dg$ on $E_{w}^{\times}$ defined by 
\begin{equation}\label{eq 93}
    dg:=C \cdot d_{v}^{\times}r \, d\kappa_{\theta}=C\cdot\frac{dr \, d\theta}{2\pi r}=C\cdot\frac{d^{\times}_{\mathbb{C}}g}{2},
\end{equation}
where $C$ is given as in \eqref{eq 112} and $g=re^{i\theta}$.
Also a Haar measure $dg$ on $\G_{\gamma}(F_v)\backslash \G(F_v)$ is defined by
\begin{equation}\label{eq 127}
    dg:= d_v^{\times} y \, d_{v} x,
\end{equation}
where $g=\sm y & 0\\
0 & 1\esm \sm 1 & x\\
0 & 1\esm\in \G_{\gamma}(F_v)\backslash \G(F_v)$.

By \eqref{eq 92}, \eqref{eq 93} and the definition of the regulator of a number field, we have
\begin{equation*}
    \vol\left(i_{\infty}(\mathcal{O}_{E}^{\times})\backslash \mathbb{A}^{1}_{E,\infty} \right)=\left(\frac{C}{2}\right)^{a}\cdot \frac{\mathrm{reg}(E)}{w_E}.
\end{equation*}
Here, $w_{E}$ denotes the number of roots of unity in $E$, $\mathrm{reg}(E)$ denotes the regulator of $E$, and $a$ is the number of $v\in S_{F,\infty}$ such that $F_v=\RR$ and $\sigma_v(m^2-4u)<0$. 
It follows that 
\begin{equation}\label{eq 128}
    \vol\left(i_{\infty}(\mathcal{O}_{E}^{\times})\backslash \mathbb{A}^{1}_{E,\infty} \right) \ll_{F} \mathrm{reg}(E).
\end{equation}
Combining Lemmas \ref{lem 4}, \ref{lem : ell-1}, \ref{lem : ell-3} and \eqref{eq 128}, we obtain
\begin{equation*}
    \begin{aligned}
        &\vol\left(\overline{\G_{\gamma_{m,u}}(F)}\backslash \overline{\G_{\gamma_{m,u}}(\mathbb{A}_{F})}\right)\prod_{v\in S_{F,\fin}}\left|\int_{\overline{\G_{\gamma_{m,u}}(F_v)}\backslash \bar{\G}(F_v)} \phi_v(g^{-1}\gamma_{m,u} g) dg\right|\\
        &\ll_{F}\mathrm{cl}(E)\cdot \mathrm{reg}(E)\cdot \prod_{\substack{v\in S_{F,\fin}\\ f_v>0}} \left(\frac{q_v^{2[\frac{f_v+1}{2}]+3}}{(q_v-1)^2(q_v+1)}\right) \cdot \prod_{\substack{v\in S_{F,\fin}\\ f_v=0}} \left(\frac{q_v(q_v+1-N_v)}{q_v^2-1}\right)\\
        &\times \prod_{\substack{v\in S_{F,\fin}\\ e_v+f_v=0}}\left(\frac{q_v-1+N_v}{q_v}\right)  \cdot \prod_{\substack{v\in S_{F,\fin}\\ e_v+f_v>0}}\left(4q_v^{-e_v+\frac{f_v}{2}}\right).
    \end{aligned}
\end{equation*}
Note that we have
\begin{equation*}
    \prod_{\substack{v\in S_{F,\fin},\\ f_v>0}} \frac{q_v^{2[\frac{f_v+1}{2}]+3}}{(q_v-1)^2(q_v+1)}\leq \prod_{\substack{v\in S_{F,\fin},\\ f_v>0}} q_v^{5f_v}=\left|N_{F/\Q}(m^2-4u) \right|^5.
\end{equation*}
Let $n:=[F:\Q]$ and $\omega(m)$ be the number of prime divisors of an integer $m$.
Since there are at most $n$ prime ideals in $\mathcal{O}_{F}$ lying over a prime $p\in \ZZ$, the number of $v\in S_{F,\fin}$ such that $e_v>0$ is less than or equal to $n\cdot \omega(|N_{F/\Q}(J)|)$.
Thus, we obtain
\begin{equation*}
    \begin{aligned}
        &\prod_{\substack{v\in S_{F,\fin}\\ f_v=0}} \left(\frac{q_v(q_v+1-N_v)}{q_v^2-1}\right)\cdot \prod_{\substack{v\in S_{F,\fin}\\ e_v+f_v=0}}\left(\frac{q_v-1+N_v}{q_v}\right)\\
        &=\prod_{\substack{v\in S_{F,\fin}\\ f_v=0}} \left(\frac{q_v(q_v+1-N_v)}{q_v^2-1}\cdot \frac{q_v-1+N_v}{q_v} \right)\cdot \prod_{\substack{v\in S_{F,\fin}\\ f_v=0, \, e_v>0}} \left(\frac{q_v-1+N_v}{q_v}\right)^{-1}\\
        &\leq 2^{n\cdot \omega(|N_{F/\Q}(J)|)}.
    \end{aligned}
\end{equation*}
The last inequality holds since $\frac{q_v}{q_v-1+N_v}\leq 2$.
Similarly, we get 
\begin{equation*}
    \prod_{\substack{v\in S_{F,\fin}\\ e_v+f_v>0}}4q_v^{-e_v+\frac{f_v}{2}} \leq 4^{n\cdot \omega(|N_{F/\Q}(J)|\cdot |N_{F/\Q}(m^2-4u)|)}\cdot|N_{F/\Q}(J)|^{-1}\cdot|N_{F/\Q}(m^2-4u)|^{\frac{1}{2}}.
\end{equation*}
By \cite[Theorem 1]{L1}, we also have
\begin{equation}\label{eq 60}
    \mathrm{cl}(E)\cdot \mathrm{reg}(E) \leq \frac{w_{E}}{2}\cdot \left(\frac{2}{\pi}\right)^{r_{2,E}}\cdot \left(\frac{e \log \Delta_{E}}{4(2n-1)} \right)^{2n-1}\sqrt{\Delta_{E}},
\end{equation}
where $e$ denotes the Euler number, $r_{2,E}$ denotes the number of conjugate pairs of complex embeddings of $E$, and $\Delta_{E}$ denotes the discriminant of $E$.
Note that the number of roots of unity of a number field $K$ is bounded if the extension degree $[K:\Q]$ is fixed. 
Since $E:=F(\sqrt{m^2-4u})$, it follows that $\Delta_{E}=\Delta_{F}^{2}\cdot N_{F/\Q}(m^2-4u)$.
Note that for a fixed positive integer $M$ and for any $\epsilon>0$, we have 
\begin{equation*}
    M^{\omega(N)}\ll_{\epsilon} N^{\epsilon}, \quad N\to \infty.
\end{equation*} 
Thus, we conclude that 
\begin{equation}\label{eq 66}
\begin{aligned}
        \vol\left(\overline{\G_{\gamma_{m,u}}(F)}\backslash \overline{\G_{\gamma_{m,u}}(\mathbb{A}_{F})}\right)&\prod_{v\in S_{F,\fin}}\left|\int_{\overline{\G_{\gamma_{m,u}}(F_v)}\backslash \bar{\G}(F_v)} \phi_v(g^{-1}\gamma_{m,u} g) dg\right|\\
        &\ll_{F} |N_{F/\Q}(J)|^{-\frac{1}{2}}\cdot |N_{F/\Q}(m^2-4u)|^{7}.
\end{aligned}
\end{equation}

To complete the calculation of $S_{\mathrm{ell}}(\phi)$, we compute the integral 
\begin{equation}\label{eq 105}
    \int_{\overline{\G_{\sigma_v(\gamma)}(F_v)} \backslash \bar{\G}(F_v)} \phi_v(g^{-1} \sigma_v(\gamma) g) dg
\end{equation}
for $v\in S_{F,\infty}$.
Since the Haar measure on $\overline{\G_{\sigma_v(\gamma)}(F_v)} \backslash \bar{\G}(F_v)$ is defined differently in the two cases, we compute \eqref{eq 105} in the following two lemmas.

\begin{lem}\label{lem : ell-4}
    Assume that $v\in S_{F,\infty}$ and that there are two roots of $x^2-\sigma_v(m)x+\sigma_v(u)=0$ in $F_v$.
    Let 
    \begin{equation*}
        \sigma_{v}(\alpha_{\gamma}):=\frac{\sigma_v(m)+\sqrt{\sigma_v(m^2-4u)}}{\sigma_v(m)-\sqrt{\sigma_v(m^2-4u)}}.
    \end{equation*}
    Then, we have
    \begin{equation}\label{eq 95}
        \left|\int_{\overline{\G_{\sigma_v(\gamma_{m,u})}(F_v)}\backslash \bar{\G}(F_v)}\phi_v(g^{-1}\sigma_v(\gamma_{m,u})g)dg  \right|\leq \frac{(1+\epsilon_v)}{2\pi}\left|\frac{\sigma_v(u)}{\sigma_v(m^2-4u)} \right|^{\frac{1+\epsilon_v}{2}}.
    \end{equation}
\end{lem}
\begin{proof}
Since $\sigma_{v}(\alpha_{\gamma})$ is the ratio of two roots of $x^2-\sigma_v(m)x+\sigma_v(u)=0$ in $F_v$, there are $g_0\in \G(F_v)$ and $z_0\in \Z(F_v)$ such that 
\begin{equation*}
     \gamma_{v}:=\bpm \sigma_v(\alpha_{\gamma}) & 0\\
    0 & 1\ebpm =z_0g_0\gamma_{m,u} g_0^{-1}.
\end{equation*}
Thus, we have 
\begin{equation*}
    \int_{\overline{\G_{\sigma_v(\gamma_{m,u})}(F_v)}\backslash \bar{\G}(F_v)}\phi_v(g^{-1}\sigma_v(\gamma_{m,u})g)dg = \int_{\overline{\G_{\gamma_v}(F_v)}\backslash \bar{\G}(F_v)}\phi_v(g^{-1}\gamma_v g)dg. 
\end{equation*}
By the Iwasawa decomposition, 
\begin{equation*}
\overline{\rm G_{\gamma_v}(F_v)}\backslash \bar{\rm G}(F_v)
\cong \left\{\bpm 1 & x\\ 0 & 1\ebpm : x\in F_v\right\} \overline{{\rm K}_v}. 
\end{equation*}
Since $\phi_{v}$ is bi-$\K_v$-invariant and $\vol(\overline{\K_v})=1$, we have
\begin{equation}\label{eq 94}
\int_{\overline{\rm G_{\gamma_v}(F_v)} \backslash \bar{\rm G}(F_v)} \phi_v(g^{-1}\gamma_v g) dg
= \int_{F_v} \phi_v\left(\bpm 1 & -x\\ 0 & 1\ebpm \bpm \sigma_v(\alpha_{\gamma}) & 0\\0 & 1\ebpm \bpm 1 & x\\ 0 & 1\ebpm \right) d_v x. 
\end{equation}

Assume that $F_v=\RR$. 
Since the support of $\phi_v$ is contained in $\GL_2^{+}(\RR)$, it follows that 
\begin{equation*}
\phi_v\left(\bpm 1 & -x \\
0 & 1 \ebpm \bpm \sigma_v(\alpha_{\gamma}) & 0\\
0 & 1\ebpm \bpm 1 & x \\
0 & 1\ebpm \right)=\phi_v\left(\bpm \sigma_v(\alpha_{\gamma}) & x(\sigma_v(\alpha_{\gamma})-1)\\
0 & 1\ebpm\right)=0    
\end{equation*}
unless $\sigma_v(\alpha_{\gamma})>0$. 
If $\sigma_v(\alpha_{\gamma})>0$, then by the Cartan decomposition, there exist $r>0$ and $\kappa_1, \kappa_2\in \K_{v}$ such that 
\begin{equation}\label{eq 115}
g:=\bpm \sigma_v(\alpha_{\gamma}) & x(\sigma_v(\alpha_{\gamma})-1)\\ 0 & 1\ebpm = \sigma_v(\alpha_{\gamma})^{\frac{1}{2}} \kappa_1 \bpm e^{-\frac{r}{2}} & 0 \\ 0 & e^{\frac{r}{2}} \ebpm \kappa_2. 
\end{equation}
Then, we have the following relation between $x$ and $r$: 
\begin{equation*}
{\rm tr}(gg^{*}) = \sigma_v(\alpha_{\gamma})^2 + x^2(\sigma_v(\alpha_{\gamma})-1)^2 + 1
=2\sigma_v(\alpha_{\gamma})\cosh(r),
\end{equation*}
where $g^{*}$ denotes the conjugate transpose of $g$.
Since $\phi_{v}$ is $\Z(F_v)$-invariant and bi-$\K_{v}$-invariant, it follows that 
\begin{equation*}
    \phi_{v}(g)=\phi_{v}\left(\bpm e^{-\frac{r}{2}} & 0\\
     0 & e^{\frac{r}{2}}\ebpm \right).
\end{equation*}
Let $g_1:=\sm \sigma_v(\alpha_{\gamma}) & |x(\sigma_v(\alpha_{\gamma})-1)| \\
    0 & 1  \esm$. Since the trace of $gg^{*}$ is equal to the trace of $g_1 g_1^{*}$, we have
\begin{equation}\label{eq 116}
    \phi_v\left(\bpm \sigma_v(\alpha_{\gamma}) & x(\sigma_v(\alpha_{\gamma})-1) \\
    0 & 1\ebpm\right) = \phi_v\left(\bpm \sigma_v(\alpha_{\gamma}) & |x(\sigma_v(\alpha_{\gamma})-1)| \\
    0 & 1\ebpm\right). 
\end{equation}
Hence, by changing the variable $x^2(\sigma_v(\alpha_{\gamma})-1)^2 =2\sigma_v(\alpha_{\gamma})\cosh(r)-\sigma_v(\alpha_{\gamma})^2-1$, we have
\begin{equation}\label{eq 122}
    \begin{aligned}
\int_{F_v} &\phi_v\left(\bpm \sigma_v(\alpha_{\gamma}) & x(\sigma_v(\alpha_{\gamma})-1) \\ 0 & 1\ebpm \right) dx\\
& = \frac{\sigma_v(\alpha_{\gamma})^{\frac{1}{2}}}{|\sigma_v(\alpha_{\gamma})-1|}\int_{\log\sigma_v(\alpha_{\gamma})}^\infty \frac{\phi_v\left(\sm e^{-\frac{r}{2}} & 0 \\ 0 & e^{\frac{r}{2}}\esm \right)\sinh(r) }{\sqrt{\sinh^2(r/2)-\sinh^2\left(\frac{\log\sigma_v(\alpha_{\gamma})}{2}\right)}} dr
\\ & = \frac{\sigma_v(\alpha_{\gamma})^{\frac{1}{2}}}{2\pi|\sigma_v(\alpha_{\gamma})-1|} \widehat{h_{\phi_v}}\left(\frac{\log \sigma_v(\alpha_{\gamma}) }{2\pi}\right).
\end{aligned}
\end{equation}
The last equality holds by \eqref{e:hathphiv_def}.

Assume that $F_v=\CC$. Similarly, for $x\in \CC$, we have 
\begin{equation*}
    \phi_v\left(\bpm \sigma_v(\alpha_{\gamma}) & x(\sigma_v(\alpha_{\gamma})-1) \\
    0 & 1 \ebpm \right)=\phi_v\left(\bpm |\sigma_v(\alpha_{\gamma})| & |x||\sigma_v(\alpha_{\gamma})-1| \\
    0 & 1 \ebpm \right).
\end{equation*}
Then, we have 
\begin{equation}\label{eq 117}
\begin{aligned}
    \int_{F_v} \phi_v\left(\bpm \sigma_v(\alpha_{\gamma}) & x(\sigma_v(\alpha_{\gamma})-1) \\
    0 & 1 \ebpm\right)dx&=\int_{F_v}\phi_v\left(\bpm |\sigma_v(\alpha_{\gamma})| & |x||\sigma_v(\alpha_{\gamma})-1|\\
    0 & 1\ebpm \right)dx\\
    &=2\pi \int_{0}^{\infty} \phi_{v}\left(\bpm |\sigma_v(\alpha_{\gamma})| & y|\sigma_v(\alpha_{\gamma})-1| \\
    0 & 1 \ebpm\right) y dy.
\end{aligned}
\end{equation}
Similarly, by changing the variable $y^2|\sigma_v(\alpha_{\gamma})-1|^2=2|\sigma_v(\alpha_{\gamma})|\cosh(r)-|\sigma_v(\alpha_{\gamma})|^2-1$, we have 
\begin{equation}\label{eq 123}
\begin{aligned}
    2\pi \int_{0}^{\infty} \phi_{v}\left(\bpm |\sigma_v(\alpha_{\gamma})| & y|\sigma_v(\alpha_{\gamma})-1| \\
    0 & 1 \ebpm\right) y dy&=\frac{2\pi|\sigma_v(\alpha_{\gamma})|}{|\sigma_v(\alpha_{\gamma})-1|^2}\int_{\log |\sigma_v(\alpha_{\gamma})|}^{\infty} \phi_{v}\left(\bpm e^{-\frac{r}{2}} & 0\\
    0 & e^{\frac{r}{2}}\ebpm \right) \sinh(r)dr\\
    &=\frac{|\sigma_v(\alpha_{\gamma})|}{\pi|\sigma_v(\alpha_{\gamma})-1|^2}\widehat{h_{\phi_v}}\left(\frac{\log|\sigma_v(\alpha_{\gamma})|}{\pi}\right).
\end{aligned}
\end{equation}

Since $h_{v}$ is a non-negative even function on $\RR$, for any $x\in \RR$, we have 
    \begin{equation*}
    \begin{aligned}
        \left|\widehat{h_v}(x) \right|=\left|\int_{\RR} h_v(z)e^{-2\pi i x z} dz \right|
        =\left|\int_{\RR} h_v(z)\cos(2\pi x z) dz \right|
        \leq \left|\int_{\RR} h_v(z) dz \right| = 1.
    \end{aligned}
    \end{equation*}
    Thus, we complete the proof of Lemma \ref{lem : ell-4}.
\end{proof}

\begin{lem}\label{lem : ell-5}
    Assume that $v\in S_{F,\infty}$ and that there is no root of $x^2-\sigma_v(m)x+\sigma_v(u)=0$ in $F_v$.
    Then, we have  
    \begin{equation*}
        \left|\int_{\overline{\G_{\sigma_v(\gamma_{m,u})}(F_v)}\backslash \bar{\G}(F_v)} \phi_{v}(g^{-1}\sigma_v(\gamma_{m,u}) g) dg \right|\leq \frac{1}{\sqrt{2}}\left|\frac{\sigma_v(u)}{\sigma_v(4u-m^2)}\right|\cdot h_{\phi_v}(0).
    \end{equation*}
\end{lem}
\begin{proof}
Following the proof of the case for $D<0$ of \cite[Lemma 2.1]{BL17}, we have
\begin{equation}\label{eq 108}
        \int_{\overline{\G_{\sigma_v(\gamma_{m,u})}(F_v)}\backslash \bar{\G}(F_v)} \phi_{v}(g^{-1}\sigma_v(\gamma_{m,u}) g) dg =\frac{1}{4\sqrt{2}}\int_{-\infty}^{\infty} \widehat{h_{\phi_{v}}}(x)\frac{\cosh(\pi x)}{\sinh^2(\pi x)+\sigma_v\left(\frac{4u-m^2}{4u}\right)}dx.
    \end{equation}
Since 
    \begin{equation*}
        \begin{aligned}
            \left|\sinh^2(\pi x) + \sigma_v\left(\frac{4u-m^2}{4u}\right)\right|\geq \left|\sigma_v\left(\frac{4u-m^2}{4u}\right)\cdot \cosh(\pi x)\right|,
        \end{aligned}
    \end{equation*}
we conclude that 
    \begin{equation*}
    \begin{aligned}
        \left|\frac{1}{4\sqrt{2}}\int_{-\infty}^{\infty} \widehat{h_{\phi_{v}}}(x)\frac{\cosh(\pi x)}{\sinh^2(\pi x)+\sigma_v\left(\frac{4u-m^2}{4u}\right)}dx\right|&\leq \frac{1}{4\sqrt{2}}\left|\frac{\sigma_v(4u)}{\sigma_v(4u-m^2)}\right|\cdot \left|\int_{-\infty}^{\infty} \widehat{h_{\phi_v}}(x) dx \right|\\
        &=\frac{1}{\sqrt{2}}\left|\frac{\sigma_v(u)}{\sigma_v(4u-m^2)}\right|\cdot h_{\phi_v}(0).
    \end{aligned}        
    \end{equation*}
\end{proof}

\subsection{Hyperbolic contribution}
For each place $v$ of $F$, we define a function $H_{v}$ on $\G(F_{v})$ by 
\begin{equation}\label{eq 113}
    H_{v}\left(\begin{pmatrix}
    a & x \\
    0 & b
    \end{pmatrix}\kappa\right):=\left|\frac{a}{b}\right|_{v},
\end{equation}
where $a,b,x\in F_v$ and $\kappa_{v}\in \K_{v}$. 
Let $H$ be a function on $\G(\mathbb{A}_{F})$ defined by 
\begin{equation*}
    H:=\prod_{v\in S_{F}} H_{v}.
\end{equation*}
In \cite[(6.35)]{GJ}, $S_{\mathrm{hyp}}(\phi)$ is defined by
\begin{equation}\label{eq : eq hyp-1}
S_{\mathrm{hyp}}(\phi):= -\frac{{\vol}(F^\times \backslash\mathbb{A}_{F}^1) }{2} \sum_{\substack{\alpha\in F^\times \\ \alpha\neq 1}} J_{\alpha}(\phi),
\end{equation}
where for each $\alpha\in F^{\times}$ with $\alpha\neq 1$, 
\begin{equation}\label{eq 19}
    J_{\alpha}(\phi):=\int_{\K}\int_{\mathbb{A}_{F}} \phi\left(\kappa^{-1}\bpm 1 & -x\\ 0 & 1\ebpm \bpm \alpha & 0 \\ 0 & 1\ebpm \bpm 1 & x\\ 0 & 1\ebpm \kappa\right) \sum_{v\in S_{F}}\log H_v\left(\bpm 0 & 1 \\ -1 & 0 \ebpm \bpm 1 & x\\ 0 & 1\ebpm \kappa\right) dx \, d\kappa.
\end{equation}
If $J_{\alpha}(\phi)\neq 0$, then there are $\kappa_{v}\in \K_{v}$ and $x_v\in F_v$ such that $\phi_{v}\left(\kappa_v^{-1} \sm 1 & -x_v\\
    0 & 1\esm \sm \alpha & 0\\
    0 & 1\esm \sm 1 & x_v\\
    0 & 1\esm \kappa_v \right)$ is non-zero for each $v\in S_{F,\fin}$.
Since $\phi_{v}$ is a characteristic function of $\Z(F_v)\K_{v,e_v}$, we have 
\begin{equation*}
    \kappa_{v}^{-1}\bpm 1 & -x_v\\
    0 & 1 \ebpm \bpm \alpha & 0\\
    0 & 1 \ebpm \bpm 1 & x_v\\
    0 & 1 \ebpm \kappa_{v}\in \Z(F_v)\K_{v,e_v}\subset \Z(F_v)\K_{v}.
\end{equation*}
This implies that 
\begin{equation*}
    \bpm 1 & -x_v\\
    0 & 1 \ebpm \bpm \alpha & 0\\
    0 & 1 \ebpm \bpm 1 & x_v\\
    0 & 1 \ebpm=\bpm \alpha & x_v(\alpha-1)\\
    0 & 1\ebpm\in \Z(F_v)\K_{v}.
\end{equation*}
Hence, if $J_{\alpha}(\phi)\neq 0$, then $\alpha\in \mathcal{O}_{F_v}^{\times}$ for each $v\in S_{F,\fin}$, and so $\alpha\in \mathcal{O}_{F}^{\times}$.

For $v,w\in S_{F}$, if $v\neq w$, then let 
\begin{equation*}
    J_{\alpha,v}^{w}(\phi_v):=\int_{\K_v}\int_{F_v}\phi_{v}\left(\kappa^{-1} \bpm 1 & -x\\
    0 & 1\ebpm \bpm \alpha & 0\\
    0 & 1\ebpm \bpm 1 & x\\
    0 & 1\ebpm \kappa \right)d_v x \, d_v \kappa,
\end{equation*}
and if $v=w$, then
\begin{equation*}
    J_{\alpha,v}^{v}(\phi_v):=\int_{\K_v}\int_{F_v}\phi_{v}\left(\kappa^{-1} \bpm 1 & -x\\
    0 & 1\ebpm \bpm \alpha & 0\\
    0 & 1\ebpm \bpm 1 & x\\
    0 & 1\ebpm \kappa \right)\log H_{v}\left(\bpm 0 & 1\\
    -1 & 0\ebpm \bpm 1 & x\\
    0 & 1\ebpm \kappa \right) d_v x \, d_v \kappa.
\end{equation*}
Let $\sigma_{v} : F\to F_{v}$ be an embedding corresponding to an archimedean place $v$ of $F$. 
By \eqref{eq 19}, we have 
\begin{equation}\label{eq 21}
    J_{\alpha}(\phi)=\sum_{w\in S_{F}}\prod_{v\in S_{F,\fin}}J_{\alpha,v}^{w}(\phi_v)\prod_{v\in S_{F,\infty}}J_{\sigma_v(\alpha),v}^{w}(\phi_v).
\end{equation}

\begin{lem}\label{lem 1}
    Let $\alpha\in \mathcal{O}_{F}^{\times}$ and $v\in S_{F,\fin}$.
    If $\val_{v}(\alpha-1)=0$, then $J_{\alpha,v}^{v}(\phi_v)=0$.
\end{lem}
\begin{proof}
    Assume that $\kappa_{v}\in \K_{v}$ and $x_v\in F_v$ satisfying 
    \begin{equation*}
        \phi_{v}\left(\kappa_v^{-1} \bpm 1 & -x_v\\
    0 & 1\ebpm \bpm \alpha & 0\\
    0 & 1\ebpm \bpm 1 & x_v\\
    0 & 1\ebpm \kappa_v \right)\neq 0.
    \end{equation*}
    Since the determinant of $\kappa_v^{-1} \sm 1 & -x_v\\
    0 & 1\esm \sm \alpha & 0\\
    0 & 1\esm \sm 1 & x_v\\
    0 & 1\esm \kappa_v$ is in $\mathcal{O}_{F}^{\times}$, we have 
    \begin{equation*}
        \kappa_v^{-1} \bpm 1 & -x_v\\
    0 & 1\ebpm \bpm \alpha & 0\\
    0 & 1\ebpm \bpm 1 & x_v\\
    0 & 1\ebpm \kappa_v\in \K_{v,e_v}.
    \end{equation*}
    Let $\kappa_v=\sm a_v & b_v \\
    c_v & d_v\esm\in \K_v$. 
    By a computation, we have 
    \begin{equation*}
    \begin{aligned}
        &\kappa_v^{-1} \bpm 1 & -x_v\\
    0 & 1\ebpm \bpm \alpha & 0\\
    0 & 1\ebpm \bpm 1 & x_v\\
    0 & 1\ebpm \kappa_v\\
    &=(\det \kappa_{v})^{-1}\bpm \alpha a_v d_v + (\alpha-1) x_v c_v d_v-b_v c_v & \alpha b_v d_v+(\alpha-1) x_v d_v^2- b_vd_v\\\
    -(\alpha-1)c_v(a_v+ x_v c_v) & -\alpha b_v c_v - (\alpha-1) x_v c_v d_v + a_v d_v\ebpm.
    \end{aligned}
    \end{equation*}
    Thus, we obtain that 
    \begin{equation}\label{eq 20}
        (\alpha-1) x_v c_v d_v \in \mathcal{O}_{F_v}, \quad (\alpha-1) x_v d_v^2\in \mathcal{O}_{F_v}, \text{ and } 
        (\alpha-1) c_v(a_v+x_vc_v)\in \varpi_{v}^{e_v}\mathcal{O}_{F_v}.
    \end{equation}
    From the assumption that $\alpha-1\in \mathcal{O}_{F_v}^{\times}$, the conditions \eqref{eq 20} are equivalent to 
    \begin{equation*}
        x_v c_v d_v \in \mathcal{O}_{F_v}, \quad x_v d_v^2\in \mathcal{O}_{F_v}, \text{ and } 
        c_v(a_v+x_vc_v)\in \varpi_{v}^{e_v}\mathcal{O}_{F_v}.
    \end{equation*}

    Since $\det(\kappa_v)=a_v d_v-b_v c_v\in \mathcal{O}_{F_v}^{\times}$, it follows that at least one of $c_v$ and $d_v$ is in $\mathcal{O}_{F_v}^{\times}$.
    If $c_v\in \mathcal{O}_{F_v}^{\times}$, then $x_v\in \mathcal{O}_{F_v}$ because $c_v(x_vc_v+a_v)\in \mathcal{O}_{F_v}$.
    Otherwise, if $d_v\in \mathcal{O}_{F_v}^{\times}$, then we also have $x_v\in \mathcal{O}_{F_v}$ by $x_v d_v^2\in \mathcal{O}_{F_v}$.
    Thus, we obtain that $x_v\in \mathcal{O}_{F_v}$.
    This implies that 
    \begin{equation}\label{eq 27}
        \bpm 0 & 1 \\
        -1 & 0\ebpm \bpm 1 & x_v\\
        0 & 1\ebpm \kappa_v\in \K_{v}.
    \end{equation}
    Therefore, we have 
    \begin{equation*}
        \log H_v\left( \bpm 0 & 1 \\
        -1 & 0\ebpm \bpm 1 & x_v\\
        0 & 1\ebpm \kappa_v\right)=0.
    \end{equation*}
\end{proof}
By \eqref{eq 21} and Lemma \ref{lem 1}, we can rewrite $J_{\alpha}(\phi)$ as 
\begin{equation}\label{eq 30}
    J_{\alpha}(\phi)=\sum_{\substack{w\in S_{F,\fin}, \\ \val_{w}(\alpha-1)>0 }}\prod_{v\in S_{F,\fin}} J_{\alpha,v}^{w}(\phi_v)\prod_{v\in S_{F,\infty}}J_{\sigma_v(\alpha),v}^{w}(\phi_v) + \sum_{w\in S_{F,\infty}}\prod_{v\in S_{F,\fin}}J_{\alpha,v}^{w}(\phi_v)\prod_{v\in S_{F,\infty}}J_{\sigma_v(\alpha),v}^{w}(\phi_v).
\end{equation}
The following lemma provides an upper bound for $|J^{w}_{\alpha,v}(\phi_v)|$ when $v\in S_{F,\fin}$.

\begin{lem}\label{lem 2}
    Assume that $v\in S_{F,\fin}$ and $\alpha\in \mathcal{O}_{F}^{\times}$.
    Let $g_{v}:=\val_v(\alpha-1)$.
    Then, 
    \begin{equation*}
        \left|J_{\alpha,v}^{w}(\phi_v) \right|\leq \begin{cases}
            q_{v}^{g_v-\frac{e_v}{2}} \quad &\text{if }w\neq v,\\
            2f_v\cdot \log q_v \cdot q_v^{g_v-\frac{e_v}{2}} \quad &\text{if }w=v.
        \end{cases}
    \end{equation*}
\end{lem}
\begin{proof}
    First, we assume that $w\neq v$.
    Let $R_{v}$ be the collection of all $(\kappa_v,x_v)\in \K_v\times F_v$ satisfying the condition $\phi_{v}\left(\kappa_v^{-1} \sm 1 & -x_v\\
    0 & 1\esm \sm \alpha & 0\\
    0 & 1\esm \sm 1 & x_v\\
    0 & 1\esm \kappa_v \right)\neq 0$.
    Since $\phi_{v}$ is a characteristic function of $\Z(F_v)\K_{v,e_v}$, it follows that $|J_{\alpha,v}^{w}(\phi_v)|$ is equal to the volume of $R_{v}$.
    By \eqref{eq 20}, we obtain that $\left(\sm a_v & b_v\\
    c_v & d_v\esm, x_v \right)\in R_{v}$ if and only if
    \begin{equation}\label{eq 22}
         x_v c_v d_v \in \varpi_{v}^{-g_v}\mathcal{O}_{F_v}, \quad x_v d_v^2\in \varpi_{v}^{-g_v}\mathcal{O}_{F_v}, \text{ and } 
        c_v(a_v+x_vc_v)\in \varpi_{v}^{e_v-g_v}\mathcal{O}_{F_v}.
    \end{equation}

    By dividing the set $R_{v}$ according to $\val_{v}(c_v)$, we have
    \begin{equation}\label{eq 23}
        R_{v}=\left(R_{v}\cap \left(\K_{v,[\frac{e_v}{2}]+1}\times F_v\right) \right) \bigcup \left(\bigcup_{i=0}^{[\frac{e_v}{2}]}  R_{v}\cap \left( \left(\K_{v,i}\backslash \K_{v,i+1}\right)\times F_v\right)\right).
    \end{equation}
    If $\val_v(c_v)\geq [\frac{e_v}{2}]+1>0$, then $d_v\in \mathcal{O}_{F_v}^{\times}$.
    Then, we have $x_v\in \varpi_{v}^{-g_v}\mathcal{O}_{F_v}$ by \eqref{eq 22}.
    Thus, we get 
    \begin{equation}\label{eq 24}
    \begin{aligned}
        \vol\left(R_{v}\cap \left(\K_{v,[\frac{e_v}{2}]+1}\times F_v\right) \right)&\leq \vol\left(\left\{\left(\sm a_v & b_v\\
        c_v & d_v\esm, x_v\right) : \val_{v}(c_v)\geq [\frac{e_v}{2}]+1 \text{ and } x_{v}\in \varpi_{v}^{-g_v}\mathcal{O}_{F_v}  \right\} \right)\\
        &= \vol\left(\K_{v,[\frac{e_v}{2}]+1} \right)\times \vol\left(\varpi_{v}^{-g_v}\mathcal{O}_{F_v} \right).
    \end{aligned}
    \end{equation}
    Now, we assume that $\val_{v}(c_v)\leq [\frac{e_v}{2}]$.
    Since $c_v(a_v+c_vx_v)\in \varpi_v^{e_v-g_v}\mathcal{O}_{F_v}$, it follows that $x_v=-a_vc_v^{-1}+y_v$, where $y_v\in \varpi_{v}^{e_v-g_v-2\val_v(c_v)}\mathcal{O}_{F_v}$. 
    Thus, for each integer $i\in \{0,\dots [\frac{e_v}{2}]\}$, we have
    \begin{equation}\label{eq 25}
        \begin{aligned}
            &\vol\left(R_{v}\cap \left(\left(\K_{v,i}\backslash \K_{v,i+1}\right)\times F_v \right) \right)\\
            &\leq \vol\left(\left\{\left(\sm a_v & b_v\\
        c_v & d_v\esm, x_v\right) : \val_{v}(c_v)=i \text{ and } x_{v}=-a_vc_v^{-1}+y_v \text{ for some }y_v\in \varpi_{v}^{e_v-g_v-2i}\mathcal{O}_{F_v} \right\} \right)\\
        &= \vol\left(\K_{v,i}\backslash \K_{v,i+1} \right)\times \vol\left(\varpi_{v}^{e_v-g_v-2i}\mathcal{O}_{F_v} \right).
        \end{aligned}
    \end{equation}
    Combining \eqref{eq 23} $\sim$ \eqref{eq 25}, we obtain that 
    \begin{equation*}
        \vol(R_v)\leq \vol\left(\K_{v,[\frac{e_v}{2}]+1} \right)\times \vol\left(\varpi_{v}^{-g_v}\mathcal{O}_{F_v} \right)+ \sum_{i=0}^{[\frac{e_v}{2}]} \vol\left(\K_{v,i}\backslash \K_{v,i+1} \right)\times \vol\left(\varpi_{v}^{e_v-g_v-2i}\mathcal{O}_{F_v} \right).
    \end{equation*}

    Similar to deduce $[\SL_2(\ZZ):\Gamma_0(N)]=N\prod_{p\mid N}(1+\frac{1}{p})$, we have
    \begin{equation}\label{eq 35}
        \vol(\K_{v,i})=\frac{1}{[\G(\mathcal{O}_{F_v}) : \K_{v,i}]}=\begin{cases}
            1 \quad &\text{if }i=0,\\
            q_{v}^{1-i}(q_v+1)^{-1} \quad &\text{if }i>0,
        \end{cases}
    \end{equation}
    and 
    \begin{equation*}
        \vol(\varpi_{v}^{i}\mathcal{O}_{F_v})=q_v^{-i}.
    \end{equation*}
    Then, we have
    \begin{equation}\label{eq 26}
        \begin{aligned}
            &\vol\left(\K_{v,[\frac{e_v}{2}]+1} \right)\times \vol\left(\varpi_{v}^{-g_v}\mathcal{O}_{F_v} \right)+ \sum_{i=0}^{[\frac{e_v}{2}]} \vol\left(\K_{v,i}\backslash \K_{v,i+1} \right)\times \vol\left(\varpi_{v}^{e_v-g_v-2i}\mathcal{O}_{F_v} \right)\\
            &=q_{v}^{g_v}(q_v+1)^{-1}\left(q_{v}^{1-e_v+[\frac{e_v}{2}]}+q_{v}^{-[\frac{e_v}{2}]} \right)\\
            &\leq q_{v}^{g_v-\frac{e_v}{2}}.
        \end{aligned}
    \end{equation}
    Thus, we complete the proof of Lemma \ref{lem 2} for the case when $w\neq v$.

    For the case when $w=v$, it is enough to show that for $(\kappa_v,x_v)\in R_v$, $|\log H_v(\sm 0 & 1\\-1 & 0\esm \sm 1 & x_v \\
    0& 1\esm \kappa_{v})|$ is less than or equal to $2g_v \log q_v$.
    If $x_v\in \mathcal{O}_{F_v}$, then by the definition of $H_v$ in \eqref{eq 113} we have 
     \begin{equation*}
        \log H_v\left( \bpm 0 & 1 \\
        -1 & 0\ebpm \bpm 1 & x_v\\
        0 & 1\ebpm \kappa_v\right)=0.
    \end{equation*}
    If $x_v=\varpi_{v}^{-m} u$ for some $m>0$ and $u\in \mathcal{O}_{F_v}^{\times}$, then we have
    \begin{equation}\label{eq 28}
        \bpm 0 & 1\\
        -1 & 0\ebpm \bpm 1 & \varpi_{v}^{-m}u\\
        0 & 1\ebpm = \bpm \varpi_{v}^{m} & -u^{-1}\\
        0 & \varpi_{v}^{-m}\ebpm \bpm -u^{-1} & 0\\
        -\varpi_{v}^{m} & -u\ebpm.  
    \end{equation}
    From \eqref{eq 28}, we obtain 
    \begin{equation}\label{eq 29}
        \log H_v\left( \bpm 0 & 1 \\
        -1 & 0\ebpm \bpm 1 & x_v\\
        0 & 1\ebpm \kappa_v\right)=-2m \log q_{v}.
    \end{equation}
    Assume that $(\kappa_v,x_v):=\left(\sm a_v & b_v\\
    c_v & d_v\esm, x_v \right)\in R_{v}$. 
    If $\val_{v}(c_v)=0$, then by \eqref{eq 22} we have $a_v+x_vc_v \in \varpi_{v}^{e_v-g_v}\mathcal{O}_{F_v}$.
    It follows that $\val_{v}(x_v)\geq \min(0, e_v-g_v)$.
    If $\val_{v}(c_v)>0$, then $d_{v}\in \mathcal{O}_{F_v}^{\times}$. 
    Again, by \eqref{eq 22}, we obtain that $\val_{v}(x_v)\geq -g_v$. 
    Thus, we conclude that if $(\kappa_v,x_v)\in R_v$, then $\val_{v}(x_v)\geq -g_v$.
    Combining \eqref{eq 29}, we obtain for $(\kappa_v,x_v)\in R_{v}$,
    \begin{equation*}
        |\log H_v(\sm 0 & 1\\-1 & 0\esm \sm 1 & x_v \\
    0& 1\esm \kappa_{v})|\leq 2g_v \log q_{v}.
    \end{equation*}
\end{proof}

Note that  
\begin{equation*}
    \prod_{v\in S_{F,\fin}} q_{v}^{g_v-\frac{e_v}{2}}=\left|N_{F\backslash \Q}(\alpha-1)\right|\cdot \left|N_{F\backslash \Q}(J) \right|^{-\frac{1}{2}}.
\end{equation*}
Combining \eqref{eq 30} and Lemma \ref{lem 2}, we have
\begin{equation}\label{eq 80}
\begin{aligned}
     |J_{\alpha}(\phi)|&\leq \left|N_{F\backslash \Q}(\alpha-1)\right|\cdot\left|N_{F\backslash \Q}(J) \right|^{-\frac{1}{2}}\\
     &\times\left(2\log |N_{F\backslash \Q}(\alpha-1)|\cdot \left|\prod_{v\in S_{F,\infty}}J_{\sigma_{v}(\alpha),v}^{w_0}(\phi_v)\right| + \sum_{w\in S_{F,\infty}}\left|\prod_{v\in S_{F,\infty}}J_{\sigma_v(\alpha),v}^{w}(\phi_v)\right| \right),
\end{aligned}
\end{equation}
where $w_0$ is an arbitrary non-archimedean place of $F$.
In the following lemma, we compute $J_{\alpha,v}^{w}(\phi_v)$ when $v$ is an archimedean place of $F$.

\begin{lem}\label{lem 3}
Assume that $v\in S_{F,\infty}$. 
Fix $\alpha\in F_v^\times$ and $\alpha\neq 1$. 
When $w\neq v$,
\begin{equation}\label{e:R_hyp_int}
\left|J_{\alpha, v}^w(\phi_v)\right|
\leq \frac{(1+\epsilon_v)\cdot|\alpha|^{\frac{1+\epsilon_v}{2}}}{2\pi|\alpha-1|^{1+\epsilon_v}}.
\end{equation}
When $w=v$ and $F_v=\mathbb{R}$, 
\begin{equation}\label{e:Fv_hyp_log_int}
 \left|J^{v}_{\alpha,v}(\phi_v)\right|\leq \frac{2|\alpha|}{|\alpha-1|^2}\int_{\frac{\log |\alpha|}{2\pi}}^{\infty} \widehat{h_{\phi_v}}(t) \cosh(\pi t) dt.
\end{equation}
When $w=v$ and $F_v=\mathbb{C}$, 
\begin{equation}\label{e:C_hyp_log_int}
\left|J_{\alpha, v}^v(\phi_v)\right| 
\leq - \frac{2|\alpha|}{|\alpha-1|^2} \int_{\frac{\log|\alpha|}{\pi}}^\infty \widehat{h_{\phi_v}}(t)\frac{\sinh(\pi t)}{\cosh(\pi t)-\frac{\Re(\alpha)}{|\alpha|}}\, dt.
\end{equation}
\end{lem}

\begin{proof}
Following the proof of Lemma \ref{lem : ell-4}, we obtain \eqref{e:R_hyp_int}.
By the Iwasawa decomposition, we have
\begin{equation*}
\bpm 0 & 1\\ -1 & -x\ebpm \in \bpm \frac{1}{|x|^2+1} & -\frac{\bar{x}}{|x|^2+1} \\ 0 & 1\ebpm \Z(F_v) {\rm K}_v. 
\end{equation*}
Then, we get
\begin{equation*}
\log H_v\left(\bpm 0 & 1\\-1 & -x\ebpm \right)
= -(1+\epsilon_v)\log(|x|^2+1).
\end{equation*}
By \eqref{eq 116} and \eqref{eq 117}, we have
\begin{equation*}
    \begin{aligned}
        J_{\alpha, v}^v(\phi_v) &= -(1+\epsilon_v) \int_{F_v}\phi_v\left(\bpm \alpha & x(\alpha-1)\\
        0 & 1 \ebpm \right) \log(|x|^2+1) dx
        \\ & = -2\pi^{\epsilon_v}(1+\epsilon_v) \int_{0}^\infty \phi_v\left(\bpm |\alpha| & x|\alpha-1|\\ 0 & 1\ebpm \right) \log (x^2+1) x^{\epsilon_v} dx
\\ &= -\frac{2\pi^{\epsilon_v}(1+\epsilon_v)}{|\alpha-1|^{1+\epsilon_v}} \int_{0}^\infty \phi_v\left(\bpm |\alpha| & x \\ 0 & 1\ebpm \right) \log (x^2|\alpha-1|^{-2}+1) x^{\epsilon_v} dx.
    \end{aligned}
\end{equation*}
As in the proof of Lemma \ref{lem : ell-4}, there are $r>0$, $\kappa_1,\kappa_2\in \K_v$ such that 
\begin{equation*}
    \bpm |\alpha| & x\\
    0  &  1 \ebpm = |\alpha|^{\frac{1}{2}} \kappa_1 \bpm e^{-\frac{r}{2}} & 0\\
    0 & e^{\frac{r}{2}}\ebpm \kappa_2,
\end{equation*}
and that 
\begin{equation}\label{eq 118}
    |\alpha|^2+x^2+1=2|\alpha|\cosh(r).
\end{equation}

Assume that $F_v=\RR$ and $\alpha>0$. 
Since $\phi_{v}$ is $\Z(F_v)$-invariant and bi-$\K_v$-invariant, we have by Lemma \ref{lem 5}
\begin{equation*}
\begin{aligned}
     \phi_v\left(\bpm \alpha & x\\
    0 & 1\ebpm\right)=\phi_v\left(\bpm e^{-\frac{r}{2}} & 0\\
    0 & e^{\frac{r}{2}}\ebpm \right)=-\frac{1}{4\pi^2}\int_{\frac{r}{2\pi}}^{\infty} \frac{\widehat{h_{\phi_v}}'(t)}{\sqrt{\sinh^2(\pi t)-\sinh^2\left(\frac{r}{2} \right)}} dt.
\end{aligned}
\end{equation*}
Then, we have 
\begin{equation*}
    J^{v}_{\alpha,v}(\phi_v)=\frac{1}{2\pi^2|\alpha-1|}\int_{0}^{\infty}\int_{\frac{r}{2\pi}}^{\infty} \frac{\widehat{h_{\phi_v}}'(t)}{\sqrt{\sinh^2(\pi t)-\sinh^2\left(\frac{r}{2} \right)}}\log(x^2|\alpha-1|^{-2}+1)dt\,dx.
\end{equation*}
Note that $\frac{r}{2\pi}\leq t$ implies 
\begin{equation*}
    x^2=2|\alpha|\cosh(r)-|\alpha|^2-1\leq 2|\alpha|\cosh(2\pi t) - |\alpha|^2-1=4|\alpha|\left(\sinh^2(\pi t)-\sinh^2\left(\frac{\log |\alpha|}{2} \right) \right).
\end{equation*}
Using the relation $\sinh^2(\frac{r}{2})=\sinh^2(\frac{\log |\alpha|}{2})+\frac{x^2}{4|\alpha|}$ and Fubini's theorem, we have  
\begin{equation*}
\begin{aligned}
    &J^{v}_{\alpha,v}(\phi_v)\\
    &=\frac{1}{2\pi^2|\alpha-1|}\int_{\frac{\log |\alpha|}{2\pi}}^{\infty} \widehat{h_{\phi_v}}'(t)
\int_0^{2\sqrt{|\alpha|}\sqrt{\sinh^2(\pi t) - \sinh^2\left(\frac{\log|\alpha|}{2}\right)}}
\frac{\log \bigg(\frac{x^2}{|\alpha-1|^{2}}+1\bigg)}{\sqrt{\sinh^2(\pi t)- \sinh^2\left(\frac{r}{2}\right)}}  dx \, dt\\
&=\frac{1}{2\pi^2|\alpha-1|}\int_{\frac{\log |\alpha|}{2\pi}}^{\infty} \widehat{h_{\phi_v}}'(t)
\int_0^{2\sqrt{|\alpha|}\sqrt{\sinh^2(\pi t) - \sinh^2\left(\frac{\log|\alpha|}{2}\right)}}
\frac{\log \bigg(\frac{x^2}{|\alpha-1|^{2}}+1\bigg)}{\sqrt{\sinh^2(\pi t)- \sinh^2\left(\frac{\log|\alpha|}{2}\right)-\frac{x^2}{4|\alpha|}}}  dx \, dt\\
&=\frac{\sqrt{\alpha|}}{\pi^2|\alpha-1|}\int_{\frac{\log |\alpha|}{2\pi}}^{\infty} \widehat{h_{\phi_v}}'(t)
\int_0^{\sqrt{\sinh^2(\pi t) - \sinh^2\left(\frac{\log|\alpha|}{2}\right)}}
\frac{\log \bigg(\frac{4|\alpha|x^2}{|\alpha-1|^{2}}+1\bigg)}{\sqrt{\sinh^2(\pi t)- \sinh^2\left(\frac{\log|\alpha|}{2}\right)-x^2}}  dx \, dt.
\end{aligned}
\end{equation*}
By \cite[4.295(38)]{GR}, and then applying the integration by parts, we have
\begin{equation*}
    \begin{aligned}
        J_{\alpha,v}^{v}(\phi_v)&= \frac{\sqrt{|\alpha|}}{\pi^2|\alpha-1|} \int_{\frac{\log|\alpha|}{2\pi}}^\infty 
\widehat{h_{\phi_v}}'(t)
\pi \log\bigg(\frac{1+\sqrt{1+\frac{4|\alpha|}{|\alpha-1|^2}\left(\sinh^2(\pi t)-\sinh^2\left(\frac{\log|\alpha|}{2}\right)\right)}}{2}\bigg)
dt
\\ &= - \frac{\sqrt{|\alpha|}}{\pi|\alpha-1|}\int_{\frac{\log|\alpha|}{2\pi}}^\infty \widehat{h_{\phi_v}}(t) \frac{\pi\frac{4|\alpha|}{|\alpha-1|^2} \sinh(\pi t) \cosh(\pi t)}{\left(1+\sqrt{1+\frac{4|\alpha|}{|\alpha-1|^2}f_{\alpha}(t)}\right)\sqrt{1+\frac{4|\alpha|}{|\alpha-1|^2}f_{\alpha}(t)}}dt,
    \end{aligned}
\end{equation*}
where 
\begin{equation*}
    f_{\alpha}(t):=\sinh^2(\pi t)-\sinh^2\left(\frac{\log|\alpha|}{2}\right).
\end{equation*}
Since $|\alpha-1|^2\geq (|\alpha|-1)^2$, it follows that  
\begin{equation*}
    1+\frac{4|\alpha|}{|\alpha-1|^2}f_{\alpha}(t)\geq \frac{4|\alpha|}{|\alpha-1|^2}\sinh^2(\pi t) 
\end{equation*}
and that
\begin{equation*}
    \left|\left(1+\sqrt{1+\frac{4|\alpha|}{|\alpha-1|^2}f_{\alpha}(t)}\right)\sqrt{1+\frac{4|\alpha|}{|\alpha-1|^2}f_{\alpha}(t)}\right|>\frac{2\sqrt{|\alpha|}}{|\alpha-1|}\sinh(\pi t).
\end{equation*}
Hence, we get \eqref{e:Fv_hyp_log_int}.

Assume that $F_v=\CC$. Then, we have by Lemma \ref{lem 5} 
\begin{equation*}
    \phi_{v}\left(\bpm |\alpha| & x\\
    0 & 1 \ebpm\right)= \phi_v\left(\bpm e^{-\frac{r}{2}} & 0\\
    0 & e^{\frac{r}{2}}\ebpm \right)=-\frac{1}{2\pi^3\sinh(r)}\widehat{h_{\phi_v}}'\left(\frac{r}{\pi}\right),
\end{equation*}
where $|\alpha|^2+x^2+1=2|\alpha|\cosh(r)$.
By changing the variable $x^2=2|\alpha|\cosh(r)-|\alpha|^2-1$, we obtain
\begin{equation*}
\begin{aligned}
    J^{v}_{\alpha,v}(\phi_v)&=\frac{2|\alpha|}{\pi^2|\alpha-1|^2}\int_{\log |\alpha|}^{\infty} \widehat{h_{\phi_v}}'\left(\frac{r}{\pi}\right) \log\left(\frac{2|\alpha|\cosh(r)-|\alpha|^2-1+|\alpha-1|^2 }{|\alpha-1|^2} \right)dr\\
    &=\frac{2|\alpha|}{\pi |\alpha-1|^2}\int_{\frac{\log |\alpha|}{\pi}}^{\infty} \widehat{h_{\phi_v}}'(t) \log\left(\frac{2|\alpha|\cosh(\pi t)-|\alpha|^2-1+|\alpha-1|^2 }{|\alpha-1|^2} \right)dt
\end{aligned}
\end{equation*}
Applying the integration by parts, we get \eqref{e:C_hyp_log_int}.
\end{proof}

\subsection{Parabolic contribution}
For $s\in \mathbb{C}$ with $\mathrm{Re}(s)>1$, $Z(s,\phi)$ is defined by
\begin{equation}\label{eq 16}
    Z(s,\phi):=\prod_{v\in S_{F}}Z_{v}(s,\phi_v),
\end{equation}
where for each $v\in S_{F}$, 
\begin{equation*}
    Z_{v}(s,\phi_v):=\int_{\K_v}\int_{F_{v}^{\times}}\phi_v\left(\kappa^{-1}\sm 1 & y\\
    0 & 1\esm \kappa\right)|y|_{v}^{s} d^{\times}_{v}y \, d_{v}\kappa.
\end{equation*}
In \cite[(6.34)]{GJ}, $S_{\mathrm{par}}(\phi)$ is defined by 
\begin{equation*}
    S_{\mathrm{par}}(\phi):=\lim_{s\to 1}\bigg(Z(s, \phi) - \text{ principal part of $Z(s,\phi)$ at $s=1$} \bigg).
\end{equation*}
In the following lemma, we compute $Z_v(s,\phi_v)$ for $v\in S_{F,\fin}$. 

\begin{lem}\label{lem : par}
    For each $v\in S_{F,\fin}$, let
    \begin{equation*}
        g_{v}(s):=\frac{1}{q_{v}+1}\left(q_{v}^{1-e_v s} +(q_v-1)q_v^{(2-e_v)s-1}\frac{q_v^{(2s-1)[\frac{e_v}{2}]}-1}{q_{v}^{2s-1}-1}+q_v^{-[\frac{e_v}{2}]} \right).
    \end{equation*}
    Then, we have
    \begin{equation*}
        Z_{v}(s,\phi_v)=\frac{g_v(s)}{1-q_{v}^{-s}}.
    \end{equation*}
\end{lem}
\begin{proof}
    Suppose that $e_v=0$. 
    If $\phi_v\left(\kappa^{-1}\sm  1 & y\\
    0 & 1\esm \kappa\right)$ is non-zero for some $\kappa\in \G(\mathcal{O}_{F_v})$ and $y\in F_{v}^{\times}$, then $\kappa^{-1}\sm  1 & y\\
    0 & 1\esm \kappa\in \Z(F_{v})\cdot \G(\mathcal{O}_{F_v})$. 
    This implies that $y\in \mathcal{O}_{F_v}-\{0\}$.
    Conversely, if $y\in \mathcal{O}_{F_v}-\{0\}$ and $\kappa\in \K_{v}$, then we have $\phi_v\left(\kappa^{-1}\sm  1 & y\\
    0 & 1\esm \kappa\right)$ is $1$.
    Thus, we have 
    \begin{equation*}
    \begin{aligned}
    Z_{v}(s,\phi_v)&=\mathrm{Vol}(\K_{v})\int_{\mathcal{O}_{F_v}-\{0\}} |y|^{s}_v d_{v}^{\times}y\\
    &=\sum_{m=0}^{\infty} \int_{\varpi_{v}^m\mathcal{O}_{v}^{\times}} |y|_v^{s} d_v^{\times}y\\
    &=\frac{1}{1-q_v^{-s}}.
    \end{aligned}
    \end{equation*}
    
    Now, we assume that $e_v$ is a positive integer.
    Let $\kappa=\sm a & b\\
    c & d \esm\in \K_v$.
    Since $\phi_{v}$ is the characteristic function of $\Z(F_v)\K_{v,e_v}$, it follows that $\phi_v\left(\kappa^{-1}\sm  1 & y\\
    0 & 1\esm \kappa\right)$ is non-zero if and only if
    \begin{equation}\label{eq 31}
        \kappa^{-1}\bpm 1 & y\\
        0 & 1\ebpm \kappa = \frac{1}{ad-bc}\bpm cdy+ad-bc & d^2y\\
        -c^2y & -cdy+ad-bc\ebpm\in \Z(F_v)\cdot\K_{v,e_v}.
    \end{equation}
    Since the determinant of $\kappa^{-1}\bpm 1 & y\\
        0 & 1\ebpm \kappa$ is $1$, the condition \eqref{eq 31} is equivalent to
    \begin{equation}\label{eq 32}
        cdy\in \mathcal{O}_{F_v}, \quad d^2y\in \mathcal{O}_{F_v}, \text{ and } c^2y\in \varpi_{v}^{e_v}\mathcal{O}_{F_v}.
    \end{equation}
    Note that at least one of $c$ and $d$ is in $\mathcal{O}_{F_v}^{\times}$ since $\kappa\in \K_{v}$. 
    Thus, \eqref{eq 32} is equivalent to 
    \begin{equation}\label{eq 33}
        y\in \mathcal{O}_{F_v} \text{ and } c^2y\in \varpi_v^{e_v}\mathcal{O}_{F_v}.
    \end{equation}

    As in the proof of Lemma \ref{lem 2}, we consider the condition \eqref{eq 33} according to $\val_v(c)$.
    If $\val_v(c)\leq [\frac{e_v}{2}]$, then $y\in \varpi_{v}^{e_v-2\val_v(c)}\mathcal{O}_{F_v}$, and if $\val_v(c)>[\frac{e_v}{2}]$, then $y\in \mathcal{O}_{F_v}$.
    Thus, we have
    \begin{equation}\label{eq 36}
        \begin{aligned}
            Z_v(s,\phi_v)&=\sum_{i=0}^{[\frac{e_v}{2}]} \vol(\K_{v,i}\backslash \K_{v,i+1})\cdot \int_{\varpi^{e_v-2i}\mathcal{O}_{F_v}-\{0\}}|y|_{v}^{s} d_{v}^{\times}y + \vol(\K_{v,[\frac{e_v}{2}]+1})\cdot \int_{\mathcal{O}_{F_v}-\{0\}} |y|_{v}^{s} d_{v}^{\times}y.
        \end{aligned}
    \end{equation}
    For each integer $i$, we have
    \begin{equation}\label{eq 34}
            \int_{\varpi_v^{e_v-2i}\mathcal{O}_{F_v}-\{0\}} |y|_{v}^{s} d_v^{\times}y=\frac{q_v^{-(e_v-2i)s}}{1-q_v^{-s}}.
    \end{equation}
    Combining \eqref{eq 35}, \eqref{eq 36} and \eqref{eq 34}, we complete the proof of Lemma \ref{lem : par}.
\end{proof}

By Lemma \ref{lem : par}, we have 
\begin{equation*}
    Z(s,\phi)=\zeta_{F}(s)\cdot \prod_{v\in S_{F,\fin}}g_v(s)\cdot \prod_{v\in S_{F,\infty}} Z_{v}(s,\phi_v),
\end{equation*}
where $\zeta_{F}$ denotes the Dedekind zeta function of $F$. 
Let 
\begin{equation*}
    g(s):=\prod_{v\in S_{F,\fin}}g_v(s)
\end{equation*}
and
\begin{equation*}
    Z_{\infty}(s,\phi_{\infty}):=\prod_{v\in S_{F,\infty}}Z_v(s,\phi_v).
\end{equation*}
Note that the Laurent series of $\zeta_{F}$ at $s=1$ has the form
\begin{equation*}
        \zeta_{F}(s)=\frac{\lambda_{-1}}{s-1}+\lambda_0+\sum_{n=1}^{\infty} \lambda_n (s-1)^{n}.
\end{equation*}
By the definition of $g_v$, we see that $g_v(s)$ is holomorphic at $s=1$ for each $v\in S_{F,\fin}$.
Also, $Z_{\infty}(s,\phi_{\infty})$ is holomorphic at $s=1$ (see \cite[pp. 242]{GJ}).
Thus, we have 
    \begin{equation*}
    \begin{aligned}
        Z(s,\phi)&=\left(\frac{\lambda_{-1}}{s-1}+\lambda_0+\cdots \right)\left(g(1)+g'(1)(s-1)+\cdots\right)\left(Z_{\infty}(1,\phi_{\infty})+Z_{\infty}'(1,\phi_{\infty})(s-1)+\cdots \right)\\
        &=\frac{\lambda_{-1}\cdot g(1) \cdot Z_{\infty}(1,\phi_{\infty})}{s-1}+\left(\lambda_0 \cdot g(1) \cdot Z_{\infty}(1,\phi_{\infty})+\lambda_{-1}\left(g(1)\cdot Z'_{\infty}(1,\phi_{\infty})+g'(1) \cdot Z_{\infty}(1,\phi_{\infty}) \right)\right)+ \cdots.
    \end{aligned}
    \end{equation*}
   Then, we get
    \begin{equation}\label{eq 37}
        S_{\mathrm{par}}(\phi)=\lambda_0 \cdot g(1) \cdot Z_{\infty}(1,\phi_{\infty})+\lambda_{-1}\left(g(1)\cdot Z'_{\infty}(1,\phi_{\infty})+g'(1) \cdot Z_{\infty}(1,\phi_{\infty}) \right).
    \end{equation} 
    
From Lemma \ref{lem : par}, we have for each $v\in S_{F,\fin}$, 
\begin{equation*}
    |g_v(1)|=\left|(q_v+1)^{-1}\left(q_{v}^{[\frac{e_v}{2}]+1-e_v}+q_v^{-[\frac{e_v}{2}]} \right) \right|\leq q_{v}^{-\frac{e_v}{2}}
\end{equation*}
and 
\begin{equation*}
    \begin{aligned}
        |g_v'(1)|&=\left|\log q_v \cdot(q_v+1)^{-1}\cdot \left(\left(2\left[\frac{e_v}{2}\right]-e_v \right)\cdot q_v^{[\frac{e_v}{2}]-e_v+1}-2(q_v-1)^{-1}\left(q_v^{[\frac{e_v}{2}]-e_v+1}-q_v^{1-e_v}\right) \right) \right|\\
        &\leq 6\log q_v \cdot q_{v}^{-\frac{e_v}{2}}.
    \end{aligned}
\end{equation*}
Then, we get
\begin{equation*}
    |g(1)|=\left|\prod_{v\in S_{F,\fin}} g_v(1) \right|\leq \left|N_{F\backslash \Q}(J) \right|^{-\frac{1}{2}}
\end{equation*}
and 
\begin{equation*}
    |g'(1)|=\left|\sum_{w\in S_{F,\fin}} g_w'(1) \prod_{\substack{v\in S_{F,\fin}\\ v\neq w}} g_v(1) \right|
    \leq 6\left|N_{F\backslash \Q}(J) \right|^{-\frac{1}{2}}\log \left|N_{F\backslash \Q}(J) \right|.
\end{equation*}
Again by \eqref{eq 37}, we conclude that
\begin{equation}\label{eq 71}
    |S_{\rm par}(\phi)|\leq \left|N_{F\backslash \Q}(J) \right|^{-\frac{1}{2}}\left(\lambda_0 Z_{\infty}(1,\phi_{\infty}) + \lambda_{-1}\left(Z'_{\infty}(1,\phi_{\infty})+6\cdot Z_{\infty}(1,\phi_{\infty})\log \left|N_{F\backslash \Q}(J) \right| \right) \right).
\end{equation}
In the following lemma, we compute $Z_v(1,\phi_v)$ and $Z_v'(1,\phi_v)$ for $v\in S_{F,\infty}$.

\begin{lem}\label{lem 7}
Assume that $v\in S_{F,\infty}$. Let $\epsilon_v:=[F_v:\RR]-1$. Then, we have 
\begin{equation*}
Z_{v}(1, \phi_v) = \frac{(1+\epsilon_v)^2}{4\pi^{1+\epsilon_v}} 
\end{equation*}
and 
\begin{equation}\label{e:Zphiv'_1} 
Z_{v}'(1, \phi_v) = C_{v,1}+C_{v,2}\cdot h_{\phi_v}(0)+C_{v,3}\cdot \int_{-\infty}^\infty h_{\phi_v}(t) \frac{\Gamma'}{\Gamma}(1+(1+\epsilon_v)it) dt,
\end{equation}
where $C_{v,1}, C_{v,2}$ and $C_{v,3}$ are constants depending on $F_v$.
\end{lem}
\begin{proof}
See \cite[Theorem 6.2]{HE2} or \cite[Proposition 6.5.3(1)]{EGM98}.
\end{proof}

\subsection{Eisenstein series and residual contribution}
Let $\chi=\otimes_{v\in S_{F}} \chi_{v}$ be an idele class character of $F^{\times}\backslash \mathbb{A}_{F}^{\times}$ and $\chi_{v}$ be a character of $F_{v}^{\times}$.
For $a:=(a_v)_{v\in S_F}\in \mathbb{A}_{F}$, let $\|a\|:=\prod_{v\in S_{F}} |a_v|_{v}$.
For a complex number $s$, we define $V_{\chi,s}$ by the space of functions $\varphi$ on $\mathrm{G}(\mathbb{A}_{F})$ satisfying 
\begin{enumerate}
    \item $\varphi\left(\sm a & x\\
    0 & b\esm g \right)=\chi(ab^{-1})\left\|\frac{a}{b}\right\|^{s+\frac{1}{2}}\varphi(g)$,
    \item $\int_{\mathrm{K}}|\varphi(k)|^2 dk<\infty$.
\end{enumerate}
Similarly, for each place $v$ of $F$, we define $V_{\chi_v,s}$ by the space of functions $\varphi_v$ on $\mathrm{G}(F_v)$ satisfying 
\begin{enumerate}
    \item $\varphi_v\left(\sm a & x\\
    0 & b\esm g \right)=\chi_v(ab^{-1})\left|\frac{a}{b}\right|_{v}^{s+\frac{1}{2}}\varphi_v(g)$,
    \item $\int_{\mathrm{K}_v}|\varphi(k)|^2 dk<\infty$.
\end{enumerate}
For $\varphi\in V_{\chi,s}$ and $\varphi_{v}\in V_{\chi_{v},s}$,  $M\left(\eta_{\chi,s}\right)(\varphi)$ and $M\left(\eta_{\chi_v,s}\right)(\varphi_v)$ are defined by 
\begin{equation*}
    \left(M\left(\eta_{\chi,s}\right)(\varphi)\right)(g):=\int_{\mathbb{A}_{F}}\varphi\left(w\bpm 1 & x\\
    0 & 1\ebpm g\right)dx,
\end{equation*}
and 
\begin{equation*}
    \left(M\left(\eta_{\chi_v,s}\right)(\varphi_v)\right)(g):=\int_{F_v}\varphi_v\left(w\bpm 1 & x\\
    0 & 1\ebpm g\right)d_v x,
\end{equation*}
where $w:=\sm 0 & -1\\
1 & 0\esm$.
Let $\pi_{\chi,s}(\phi):V_{\chi,s}\to V_{\chi,s}$ and $\pi_{\chi_v,s}(\phi_v): V_{\chi_v,s}\to V_{\chi_v,s}$ be defined by 
\begin{equation*}
    \left(\pi_{\chi,s}(\phi)\varphi\right)(g):=\int_{\bar{\G}(\mathbb{A}_{F})}\phi(x)\varphi(gx) dx
\end{equation*}
and 
\begin{equation*}
    \left(\pi_{\chi_v,s}(\phi_v)\varphi_v\right)(g):=\int_{\bar{\G}(F_v)}\phi_v(x)\varphi_v(gx) d_v x.
\end{equation*}
Then, we have by \cite[pp. 243]{GJ} 
\begin{equation*}
    S_{\rm Eis}(\phi)=\frac{1}{4\pi}\sum_{\chi} \int_{-i\infty}^{i\infty}\tr\left(M(\eta_{\chi,s})^{-1}M'(\eta_{\chi,s})\pi_{\chi,s}(\phi)\right)ds,
\end{equation*}
where the sum is over all $\chi$ whose restriction to $F_{\infty}^{+}$ is trivial.
Here, $F_{\infty}^{+}$ is the subset of $\mathbb{A}_{F}$ consisting of $a:=(a_v)_{v\in S_{F}}$ such that $a_v=1$ for all $v\in S_{F,\fin}$, and $a_v>0$ for all $v\in S_{F,\infty}$.

Note that $M(\eta_{\chi_v,s})$ acts by a scalar multiplication. 
Thus, there is a constant $(a_{\chi_v,s})_{v}\in F_v$ such that for all $\varphi_v\in V_{\chi_v,s}$, 
\begin{equation*}
    \left(M(\eta_{\chi_v,s})\varphi_v \right)=(a_{\chi_v,s})_{v}\cdot \varphi_v.
\end{equation*}
Then, for $\varphi\in V_{\chi,s}$, we obtain that 
\begin{equation}\label{eq 39}
    M(\eta_{\chi,s})\varphi = \left(\prod_{v\in S_{F}} (a_{\chi_v,s})_{v}\right)\varphi
\end{equation}
and 
\begin{equation*}
    M(\eta_{\chi,s})^{-1}M(\eta_{\chi,s})'\varphi=\left(\sum_{v\in S_{F}} \frac{(a_{\chi_v,s})'_v}{(a_{\chi_v,s})_v} \right)\varphi.
\end{equation*}
It implies that 
\begin{equation}\label{eq : eq of eis-6}
    \tr\left(M(\eta_{\chi,s})^{-1}M'(\eta_{\chi,s})\pi_{\chi,s}(\phi)\right)=\left(\sum_{v\in S_{F}} \frac{(a_{\chi_v,s})'_v}{(a_{\chi_v,s})_v} \right)\tr \pi_{\chi,s}(\phi).
\end{equation}

Assume that there is $v_0\in S_{F,\infty}$ such that $F_{v_0}=\CC$ and $\chi_{v_0}$ is non-trivial on the unit circle $S^{1}$. 
Then, there is $\theta_{0}\in [0,2\pi)$ such that $\chi_{v_0}(e^{i\theta_0})\neq 1$.
Let $g_0:=(g_{0,v})_{v\in S_{F}}\in \G(\mathbb{A}_{F})$ be defined by 
\begin{equation*}
    g_{0,v}:=\begin{cases}
        \sm e^{i\theta_0} & 0\\
        0 & 1 \esm \quad &\text{if }v=v_0,\\
        I_2 \quad&\text{otherwise}.
    \end{cases}
\end{equation*}
Then, we have
\begin{equation*}
\begin{aligned}
    \left(\pi_{\chi,s}(\phi)\varphi\right)(I_2)&=\int_{\bar{\G}(\mathbb{A}_{F})} \phi(x)\varphi(x) dx\\
    &=\int_{\bar{\G}(\mathbb{A}_{F})} \phi(g_0 x)\varphi(g_0x)dx\\
    &=\chi_{v_0}(e^{i\theta_0})\cdot \left(\pi_{\chi,s}(\phi)\varphi\right)(I_2).
\end{aligned}
\end{equation*}
This implies that $\left(\pi_{\chi,s}(\phi)\varphi\right)(I_2)=0$ and that $\left(\pi_{\chi,s}(\phi)\varphi\right)(\kappa)=0$ for all $\kappa\in \prod_{v\in S_{F,\infty}} \K_{v}^{0}$ since
\begin{equation*}
    \begin{aligned}
        \left(\pi_{\chi,s}(\phi)\varphi\right)(\kappa)&=\int_{\bar{\G}(\mathbb{A}_{F})}\phi(x)\varphi(\kappa x)dx\\
        &=\int_{\bar{\G}(\mathbb{A}_{F})} \phi(\kappa^{-1}x)\varphi(x) dx\\
        &=\int_{\bar{\G}(\mathbb{A}_{F})} \phi(x)\varphi(x) dx\\
        &=\left(\pi_{\chi,s}(\phi)\varphi\right)(I_2).
    \end{aligned}
\end{equation*}
By the Iwasawa decomposition, we conclude that $\pi_{\chi,s}(\phi)=0$ if there is $v_0\in S_{F,\infty}$ such that $F_{v_0}=\CC$ and $\chi_{v_0}$ is non-trivial on $S^{1}$.

Thus, if $\tr \pi_{\chi,s}(\phi)\neq 0$ and the restriction of $\chi$ to $F^{+}_{\infty}$ is trivial, then we have
\begin{equation}\label{eq 99}
    \chi_{v}=\begin{cases}
        \chi_{0,v} \text{ or } \sgn \quad &\text{if }F_v=\RR,\\
        \chi_{0,v} \quad &\text{otherwise.}
    \end{cases}
\end{equation}
Here, $\chi_{0,v}$ is the trivial character of $F_v^{\times}$ for each $v\in S_{F}$ and $\sgn$ denotes the sign function. 
Assume that $\chi=\otimes_{v\in S_F} \chi_{v}$ is an idele class character of $F^{\times}\backslash \mathbb{A}_{F}^{\times}$ such that $\chi_{v}$ satisfies \eqref{eq 99} for each $v\in S_{F}$. 
If $\chi$ is a non-trivial idele class character, then there is $v_0\in S_{F,\infty}$ such that $\chi_{v_0}$ is a non-trivial character of $F_{v_0}^{\times}$.
By the strong approximation theorem, there is $a\in F^{\times}$ such that
$\sigma_{v_0}(a)<0$ and $\sigma_v(a)>0$ for $v\in S_{F,\infty}$ with $F_v=\RR$ and $v\neq v_0$. 
This contradicts our assumption that $\chi$ is trivial on $F^{\times}$.
Hence, we get 
\begin{equation}\label{eq : eq of Eis-1}
    S_{\rm Eis}(\phi)=\frac{1}{4\pi}\int_{-i\infty}^{i\infty}\tr\left(M(\eta_{\chi_0,s})^{-1}M'(\eta_{\chi_0,s})\pi_{\chi_0,s}(\phi)\right)ds.
\end{equation}
Here, $\chi_0$ is the trivial idele class character of $F^{\times}\backslash \mathbb{A}_{F}^{\times}$.
Let $V_{\chi_0,s}^{\K(J)}$ denote the subspace of $V_{\chi_0,s}$ consisting of $v$ such that $v=\pi_{\chi_0,s}(k)v$ for all $k\in \K(J)$. 
From \eqref{eq : eq of Eis-1}, we compute $S_{\mathrm{Eis}}(\phi)$ in the following lemma.

\begin{lem}\label{prop : Eis}
Let $r_1$ be the number of real embeddings of $F$ and $r_2$ be the number of conjugate pairs of complex embeddings of $F$.
Let $\zeta_{F}$ be the Dedekind zeta function of $F$ and $\Delta_{F}$ be the discriminant of $F$.
Let $\Lambda_{F}(s)$ be the completed zeta function defined by 
\begin{equation*}
    \Lambda_{F}(s):= |\Delta_F|^{s/2} (\pi^{-s/2}\Gamma(s/2))^{r_1} ((2(2\pi)^{-s}\Gamma(s))^{r_2} \zeta_F(s).
\end{equation*}
Let $A_{J}$ be defined as in \eqref{eq 119}.
Then, we have
\begin{equation*}
     S_{\rm Eis}(\phi)=\frac{A_{J}}{2^{r_1+2}\pi} \int_{\RR}\bigg(\frac{\Lambda_{F}'(2it)}{\Lambda_{F}(2it)} - \frac{\Lambda_{F}'(2it+1)}{\Lambda_{F}(2it+1)}\bigg)\prod_{v\in S_{F,\infty}}h_{\phi_{v}}(t) \cdot \dim V_{it}^{\K(J)}dt
\end{equation*}
\end{lem}
\begin{proof}
    By \eqref{eq : eq of Eis-1}, it is enough to show that 
    \begin{equation*}
    \tr\left(M(\eta_{\chi_0,it})^{-1}M'(\eta_{\chi_0,it})\pi_{\chi_0,it}(\phi)\right)=2^{-r_1}A_{J}\bigg(\frac{\Lambda_{F}'(2it)}{\Lambda_{F}(2it)} - \frac{\Lambda_{F}'(2it+1)}{\Lambda_{F}(2it+1)}\bigg)\prod_{v\in S_{F,\infty}}h_{\phi_{v}}(t) \cdot \dim V_{it}^{\K(J)}.
    \end{equation*}
    Assume that $v\in S_{F,\fin}$. 
    We take $\varphi_v\in V_{\chi_{0,v},s}$ such that the $\varphi_v|_{\mathrm{K}_v}=1$.
    Then, we have 
    \begin{equation}\label{eq : eq of eis-2}
    \begin{aligned}
        \left(M(\eta_{\chi_{0,v},s})\varphi_v\right)(w)&=\int_{F_v}\varphi_v\left(w\bpm 1 & x\\
        0 & 1\ebpm w\right)d_{v}x\\
        &=\int_{F_v}\varphi_v\left(\bpm 1 & 0\\
        -x & 1\ebpm \right)d_{v}x\\
        &=\sum_{n\in \ZZ}q_{v}^{-n}(1-q_v^{-1})\int_{\mathcal{O}_{F_v}^{\times}}\varphi_v\left(\bpm 1 & 0\\ -u\varpi_{v}^n & 1 \ebpm \right) d_v^{\times}u.
    \end{aligned}
    \end{equation}
    When $n\geq 0$, we have $\sm 1 & 0\\
    -u\varpi_{v}^n & 1\esm\in \mathrm{K}_v$ for all $u\in \mathcal{O}_{F_v}^{\times}$.
    Thus, we have by \eqref{eq : eq of eis-2}
    \begin{equation}\label{eq : eq of eis-3}
       \left(M(\eta_{\chi_{0,v},s})\varphi_v\right)(w)=\sum_{n\geq 0}q_{v}^{-n}(1-q_{v}^{-1})+\sum_{n<0}q_{v}^{-n}(1-q_{v}^{-1}) \int_{\mathcal{O}_{F_v}^{\times}}\varphi_v\left(\bpm 1 & 0\\ -u\varpi_{v}^n & 1 \ebpm \right) d_v^{\times}u.
    \end{equation}
    For $n<0$, we have
    \begin{equation*}
        \bpm 1 & 0\\
        -u\varpi_{v}^n & 1\ebpm = \bpm u^{-1}\varpi_{v}^{-n} & 1\\
        0 & -u\varpi_{v}^{n}\ebpm \bpm 0 & 1 \\ 1 & -u^{-1}\varpi_{v}^{-n} \ebpm.
    \end{equation*}
    Since $\sm 0 & 1 \\ 1 & -u^{-1}\varpi_{v}^{-n}\esm \in \mathrm{K}_{v}$, it follows that 
    \begin{equation}\label{eq : eq of eis-4}
        \varphi_v\left(\bpm 1& 0\\
        -u\varpi_{v}^{n} & 1\ebpm \right)= q_{v}^{2n(s+\frac{1}{2})}=q_{v}^{2ns+n}.
    \end{equation}
    Combining \eqref{eq : eq of eis-3} and \eqref{eq : eq of eis-4}, we deduce that 
    \begin{equation}\label{eq : eq of eis-5}
       (a_{\chi_{0,v},s})_v=\left(M(\eta_{\chi_{0,v},s})\varphi_v\right)(w)=\frac{1-q_{v}^{-2s-1}}{1-q_{v}^{-2s}}.
    \end{equation}

    Now, we consider an archimedean place $v$ of $F$. 
    When $F_v=\RR$, since $\varphi_{v}$ is a bi-$\mathrm{SO}(2)$-invariant function, we have 
    \begin{equation*}
    \begin{aligned}
        \left(M(\eta_{\chi_{0,v},s})\varphi_{v}\right)(I_2)&=\int_{\RR}\varphi_{v}\left(w\bpm 1 & x\\
        0 & 1\ebpm\right)dx\\
        &=\int_{\RR}\varphi_{v}\left(w\bpm 1 & x\\
        0 & 1\ebpm w^{-1}\right)dx.
    \end{aligned}
    \end{equation*}
    Note that there is $\kappa\in \mathrm{SO}(2)$ such that 
    \begin{equation*}
        w\bpm 1 & x \\
        0 & 1 \ebpm w^{-1}=\bpm \frac{1}{\sqrt{x^2+1}} & -\frac{x}{\sqrt{x^2+1}} \\ 0 & \sqrt{x^2+1} \ebpm  \kappa.
    \end{equation*}
    Thus, for $v\in S_{F,\infty}$ with $F_v=\RR$, we have by \cite[3.251(2)]{GR}
    \begin{equation}\label{eq : eq of eis-7}
(a_{\chi_{0,v},s})_{v}=\left(M(\eta_{\chi_{0,v},s})\varphi_{v}\right)(I_2)
= \int_{\RR} \varphi_{v}\left(\sm \frac{1}{\sqrt{x^2+1}} & -\frac{x}{\sqrt{x^2+1}} \\0& \sqrt{x^2+1} \esm\right) dx
= \int_{\RR} (x^2+1)^{-s-\frac{1}{2}}
dx
= \frac{\sqrt{\pi} \Gamma\left(s\right)}
{\Gamma\left(s+\frac{1}{2}\right)}.
\end{equation}

When $F_v=\CC$, by taking $z=re^{i\theta}$, we have
\begin{equation*}
(M((\eta_{\chi_{0,v},s})_v)\varphi_v)(I_2)
= \int_{\CC} \varphi_v\left(w\sm 1 & z\\ 0 & 1\esm \right) dz
= \int_0^{2\pi} \int_0^\infty
\varphi_v\left(\sm w\sm 1 & re^{i\theta} \\ 0 & 1\esm w^{-1} \esm \right) rdr \, d\theta. 
\end{equation*}
Similarly, we have 
\begin{equation*}
w\sm 1 & re^{i\theta} \\ 0 & 1\esm w^{-1} 
= \sm e^{-i\theta}& \\ & 1 \esm
\sm 1 & 0 \\ -r & 1\esm 
\sm e^{i\theta} & \\ & 1\esm 
\in \sm e^{-i\theta}& \\ & 1 \esm
\sm \frac{1}{\sqrt{r^2+1}} & \frac{-r}{\sqrt{r^2+1}} \\ 0 & \sqrt{r^2+1} \esm \mathrm{U}(2).
\end{equation*}
Then, we get by \cite[3.251(2)]{GR}
\begin{equation*}
\begin{aligned}
    (M((\eta_{\chi_{0,v},s})_v)f_v)(I_2)
&= \int_0^{2\pi} \int_0^\infty
\varphi_v\left(
\sm \frac{e^{-i\theta}}{\sqrt{r^2+1}} & \frac{-re^{-i\theta}}{\sqrt{r^2+1}} \\ 0 & \sqrt{r^2+1} \esm\right) rdr \, d\theta
\\ &= 2\pi \int_0^\infty (r^2+1)^{-2s-1} r dr\\
&= \frac{\pi\Gamma\left(2s\right)}{\Gamma\left(2s+1\right)}
\end{aligned}
\end{equation*}
for $\Re(s)>-\frac{3}{4}$. 
This implies that 
\begin{equation}\label{e:aetas_archC}
a((\eta_{\chi_{0,v},s})_v)  
= \frac{\pi \Gamma\left(2s\right)}{\Gamma\left(2s+1\right)}.
\end{equation}

Combining \eqref{eq : eq of eis-6}, \eqref{eq : eq of eis-5}, \eqref{eq : eq of eis-7}, and \eqref{e:aetas_archC}, we obtain that 
\begin{equation*}
    \tr\left(M(\eta_{\chi_0,it})^{-1}M'(\eta_{\chi_0,it})\pi_{\chi_0,it}(\phi)\right)=\bigg(\frac{\Lambda_{F}'(2it)}{\Lambda_{F}(2it)} - \frac{\Lambda_{F}'(2it+1)}{\Lambda_{F}(2it+1)}\bigg) \tr\pi_{\chi_{0},it}(\phi).
\end{equation*}
By Lemma \ref{lem : lem of trace}, we complete the proof of Proposition \ref{prop : Eis}.
\end{proof}

In \cite[pp. 244]{GJ}, $S_{\mathrm{Res}}(\phi)$ is given by
\begin{equation*}
    S_{\mathrm{Res}}(\phi):=-\frac{1}{4}\sum_{\chi^2=\chi_0} \tr\left(M(\eta_{\chi,0})\pi_{\chi,0}(\phi) \right).
\end{equation*}
By a similar argument to the Eisenstein series contribution, if an idele class character $\chi$ of $F^{\times}\backslash \mathbb{A}_{F}^{\times}$ satisfies $\chi^2=\chi_0$ and $\tr(M(\eta_{\chi,0})\pi_{\chi,0}(\phi))\neq 0$, then $\chi=\chi_0$.
This implies that 
\begin{equation*}
    S_{\mathrm{Res}}(\phi)=-\frac{1}{4}\tr\left(M(\eta_{\chi_0,0})\pi_{\chi_0,0}(\phi) \right).
\end{equation*}
Combining \eqref{eq : eq of eis-5}, \eqref{eq : eq of eis-7}, and \eqref{e:aetas_archC}, we obtain the following lemma which provides the formula for $S_{\mathrm{Res}}(\phi)$.

\begin{lem}\label{prop : Res}
    With the above notation, we have 
    \begin{equation*}
        S_{\rm Res}(\phi)=-2^{-r_1-r_2-2}\cdot\sqrt{\Delta_F}\cdot\lim_{s\to 0}\frac{\Lambda_{F}(2s)}{\Lambda_{F}(2s+1)}\cdot A_{J}\cdot \prod_{v\in S_{F,\infty}} h_{\phi_v}(0) \cdot \dim V_{0}^{\K(J)}.
    \end{equation*}
\end{lem}

\section{Proof of Theorem \ref{thm : main}}\label{s : pf}
In this section, we prove Theorem \ref{thm : main}. 
We follow the notation in Sections \ref{s : Pre} and \ref{s : geometric side}. 
In the definitions of $S_{\mathrm{ell}}(\phi)$ and $S_{\mathrm{hyp}}(\phi)$, the sums seem to be infinite sums.
However, the following lemma guarantees that these sums are finite.
\begin{lem}\label{lem : lem of finiteness}
    Let $v\in S_{F,\infty}$ and $\epsilon_v:=[F_v:\RR]-1$. 
    Assume that the support of $\widehat{h_{\phi_{v}}}$ is contained in $[-a,a]$ for some positive real number $a$.
    If $\gamma\in \SL_2(F_v)$ satisfies
    \begin{equation*}
    \tr(\gamma \gamma^{*})\geq e^{-\frac{2\pi a}{1+\epsilon_v}}+e^{\frac{2\pi a}{1+\epsilon_v}},    
    \end{equation*}
    then $\phi_{v}(\gamma)=0$.
\end{lem}
\begin{proof}
    By the Cartan decomposition, there are $\kappa_1, \kappa_2\in \K_{v}^{0}$ and a non-negative real number $a(\gamma)$ such that 
    \begin{equation*}
        \gamma=\kappa_1 \bpm e^{-\frac{a(\gamma)}{2}} & 0\\
        0 & e^{\frac{a(\gamma)}{2}}\ebpm \kappa_2.
    \end{equation*}
    Then, the trace of $\gamma \gamma^{*}$ is equal to $e^{-a(\gamma)}+e^{a(\gamma)}$.
    Since $\phi_{v}$ is a bi-$\K_{v}^{0}$-invariant function, we have 
    \begin{equation}\label{eq 5}
        \phi_{v}(\gamma)=\phi_{v}\left(\bpm e^{-\frac{a(\gamma)}{2}} & 0\\
        0 & e^{\frac{a(\gamma)}{2}}\ebpm\right).
    \end{equation}
    By Lemma \ref{lem 5}, we have if $F_{v}=\RR$, then 
    \begin{equation}\label{eq 6}
        \phi_{v}\left(\bpm e^{-\frac{a(\gamma)}{2}} & 0\\
        0 & e^{\frac{a(\gamma)}{2}}\ebpm\right)=-\frac{1}{4\pi^2}\int_{\frac{a(\gamma)}{2\pi}}^{\infty} \frac{\widehat{h_{\phi_{v}}}'(t)}{\sqrt{\sinh^{2}(\pi t)-\sinh^2\left(\frac{a(\gamma)}{2}\right)}}dt,
    \end{equation}
    and if $F_v=\CC$, then 
    \begin{equation*}
        \phi_{v}\left(\bpm e^{-\frac{a(\gamma)}{2}} & 0 \\ 0 & e^{\frac{a(\gamma)}{2}} \ebpm \right)
= -\frac{\widehat{h_{\phi_{v}}}'(a(\gamma)/\pi)}{2\pi^{3}\sinh(a(\gamma))}.
    \end{equation*}
   Thus, we obtain that if $a(\gamma)\geq \frac{2\pi a}{1+\epsilon_v}$, then $\phi_{v}(\gamma)=0$.
\end{proof}

For each $v\in S_{F,\infty}$, assume that $\widehat{h_v} : \mathbb{R}\to \mathbb{R}_{\geq 0}$ and $h_v : \mathbb{C}\to \mathbb{C}$ satisfy the conditions $(1)\sim(7)$ in Section \ref{s : Pre}.
For a real positive number $t$, let
\begin{equation*}
    \widehat{h_{v,t}}(x):= \widehat{h_v}(tx).
\end{equation*}
Then, $\widehat{h_{v,t}}$ and $h_{v,t}$ also satisfy the conditions $(1)\sim(7)$ in Section \ref{s : Pre}.
Thus, we may assume that the support of $\widehat{h_{v}}$ is contained in $[-1,1]$ for all $v\in S_{F,\infty}$.
For each $v\in S_{F,\infty}$, we fix $h_{v}$. 
Let $a_{v}$ be a positive real number. 
Then, the support of $\widehat{h_{v,a_v}}$ is contained in $[-\frac{1}{a_v},\frac{1}{a_v}]$ and  
\begin{equation}\label{eq 61}
    h_{v,a_v}(z)=\frac{1}{a_v}h_{v}\left(\frac{z}{a_v}\right).
\end{equation}
Note that there is a function $\phi_{v,a_v}$ on $\G(F_v)$ satisfying the conditions $(a)\sim (c)$ in Section \ref{s : Pre} such that 
\begin{equation*}
    h_{\phi_{v,a_v}}=h_{v,a_v}.
\end{equation*}
For $\mathbf{a}:=(a_v)_{v\in S_{F,\infty}}\in \RR_{>0}^{\# S_{F,\infty}}$, let $\phi^{(\mathbf{a})}:=\prod_{v\in S_{F,\fin}}\phi_{v} \cdot\prod_{v\in S_{F,\infty}} \phi_{v,a_v}$ be a function on $\G(\mathbb{A}_{F})$.
Recalling that for each $v\in S_{F,\fin}$, $\phi_{v}$ is the characteristic function of $\Z(F_v)\K_{v,\val_{v}(J)}$.

In the following lemmas, we compute upper bounds for $S_{\mathrm{one}}(\phi^{(\mathbf{a})})$, $S_{id}(\phi^{(\mathbf{a})})$, $S_{\rm ell}(\phi^{(\mathbf{a})})$, $S_{\rm hyp}(\phi^{(\mathbf{a})})$, $S_{\rm par}(\phi^{(\mathbf{a})})$ and $S_{\rm Eis}(\phi^{(\mathbf{a})})$.

\begin{lem}\label{lem-sp,a}
With the above notation, we have
\begin{equation*}
    \begin{aligned}
        \left|S_{\mathrm{one}}(\phi^{(\mathbf{a})})\right| \leq A_{J}\cdot \prod_{v\in S_{F,\infty}}\left(e^{\frac{\pi}{a_v}}\cdot h_{v,a_v}(0)\right)
    \end{aligned}
\end{equation*}
\end{lem}
\begin{proof}
By Lemma \ref{lem : one-dim}, we have 
\begin{equation*}
    S_{\mathrm{one}}(\phi^{(\mathbf{a})})=A_{J}\cdot \prod_{v\in S_{F,\infty}} h_{v,a_v}\left(\frac{i}{2}\right).
\end{equation*}
Note that we get
\begin{equation*}
\begin{aligned}
    \left|h_{v,a_v}\left(\frac{i}{2}\right)\right|&=\frac{1}{a_v}\left|h_{v}\left(\frac{i}{2a_v}\right)\right|\\
    &=\frac{1}{a_v}\left|\int_{\RR} \widehat{h_v}(x) e^{\frac{\pi x}{a_v}}dx\right|\\
    &=\frac{2}{a_v}\left|\int_{0}^{1} \widehat{h_v}(x) \cosh\left(\frac{\pi x}{a_v}\right) dx\right|\\
    &\leq \cosh\left(\frac{\pi}{a_v}\right)\cdot h_{v,a_v}(0)\\
    &\leq e^{\frac{\pi}{a_v}}\cdot h_{v,a_v}(0).
\end{aligned}
\end{equation*}
Therefore, we conclude that
\begin{equation*}
    \begin{aligned}
        \left|S_{\mathrm{one}}(\phi^{(\mathbf{a})})\right|\leq A_{J}\cdot \prod_{v\in S_{F,\infty}}\left(e^{\frac{\pi}{a_v}}\cdot h_{v,a_v}(0)\right).
    \end{aligned}
\end{equation*}
\end{proof}

\begin{lem}\label{lem-id,a}
    With the above notation, we have 
    \begin{equation*}
        \left|S_{\rm id}(\phi^{(\mathbf{a})}) \right| \ll_{F} \prod_{v\in S_{F,\infty}} \left(a_v^{2+\epsilon_v} \cdot  h_{v,a_v}(0)\right). 
    \end{equation*}
\end{lem}
\begin{proof}
    By Lemma \ref{lem : id}, we have 
    \begin{equation*}
        \begin{aligned}
            S_{\rm id}(\phi^{(\mathbf{a})})&= \vol(\bar{\G}(F)\backslash \bar{\G}(\mathbb{A}_{F}))\cdot\prod_{v\in S_{F,\infty}}\left(\frac{1+\epsilon_v}{(2-\epsilon_v)^{2(1-\epsilon_v)} \pi^{1+\epsilon_v}} \int_{-\infty}^\infty x^{1+\epsilon_v} h_{v,a_v}(x) \tanh^{1-\epsilon_v}(\pi x) \, dx\right)\\
            &=\vol(\bar{\G}(F)\backslash \bar{\G}(\mathbb{A}_{F}))\cdot\prod_{v\in S_{F,\infty}}\left(\frac{2(1+\epsilon_v)}{(2-\epsilon_v)^{2(1-\epsilon_v)} \pi^{1+\epsilon_v}} \int_{0}^\infty x^{1+\epsilon_v} h_{v,a_v}(x) \tanh^{1-\epsilon_v}(\pi x) \, dx\right)
        \end{aligned}
    \end{equation*}
    Since $|\tanh(x)|\leq 1$, for each $v\in S_{F,\infty}$, we have
    \begin{equation*}
    \begin{aligned}
        \left|\int_{0}^\infty x^{1+\epsilon_v} h_{v,a_v}(x) \tanh^{1-\epsilon_v}(\pi x) \, dx\right|&\leq \left|\int_{0}^\infty x^{1+\epsilon_v} h_{v,a_v}(x) \, dx\right|\\
        &= \left|\int_{0}^\infty \frac{x^{1+\epsilon_v}}{a_v} \cdot h_{v}\left(\frac{x}{a_v}\right) \, dx\right|\\
        &= a_{v}^{1+\epsilon_v}\left|\int_{0}^{\infty} x^{1+\epsilon_{v}} h_{v}(x) dx \right|\\
        &=\left(\frac{1}{h_v(0)}\left|\int_{0}^{\infty} x^{1+\epsilon_{v}} h_{v}(x) dx \right|\right)a_v^{2+\epsilon}\cdot h_{v,a_v}(0).
    \end{aligned}
    \end{equation*}

\end{proof}

Now, we introduce the following lemma which is useful to compute upper bounds for $S_{\rm ell}(\phi^{(\mathbf{a})})$ and $S_{\rm hyp}(\phi^{(\mathbf{a})})$.

\begin{lem}\label{lem 6}
    For each $v\in S_{F,\infty}$, let $M_{v}$ be a positive real number. 
    Let $n:=[F:\Q]$. Then, we have
    \begin{equation*}
        \#\left\{a\in \mathcal{O}_{F} : |\sigma_v(a)|\leq M_v \text{ for all }v\in S_{F,\infty} \right\}\leq n\cdot \left(2 \cdot \prod_{v\in S_{F,\infty}} (|M_v|+1)^{1+\epsilon_v}\right)^n.
    \end{equation*}
\end{lem}
\begin{proof}
    Assume that $a\in \mathcal{O}_{F}$ satisfies $|\sigma_v(a)|\leq M_{v}$ for all $v\in S_{F,\infty}$.
    Let $f_a(x)\in \mathbb{Z}[x]$ be a monic irreducible polynomial of $a$.
    Note that the coefficients of $f_a$ can be expressed as a symmetric function of a subset of $\{\sigma_{v}(a)\}_{v\in S_{\infty}}$ and that the absolute value of each coefficient of $f_{a}$ is less than or equal to 
    \begin{equation*}
        \prod_{v\in S_{F,\infty}} (|M_v|+1)^{1+\epsilon_v}.
    \end{equation*}
    Since the degree of $f_{a}$ is less than or equal to $n$ and $f_a$ has at most $n$ roots, we complete the proof of Lemma \ref{lem 6}.
\end{proof}

For convenience, let $A:=|N_{F/\Q}(J)|$ and $B_{m,u}:=|N_{F/\Q}(m^2-4u)|$.
\begin{lem}\label{lem-ell,a}
    With the above notation, assume that $a_{v}\leq 1$ for all $v\in S_{F,\infty}$.
    Then, we have 
    \begin{equation*}
        \left|S_{\rm ell}(\phi^{(\mathbf{a})}) \right| \ll_{F} A^{-\frac{1}{2}}\cdot \prod_{v\in S_{F,\infty}} \left(e^{\frac{2\pi(n+13)}{a_v}}\cdot h_{v,a_v}(0)\right).
    \end{equation*}
\end{lem} 
\begin{proof}
Recalling \eqref{eq 114}, $S_{[\gamma]}(\phi)$ is defined by
\begin{equation*}
S_{[\gamma]}(\phi):= \prod_{v\in S_{F,\fin}} \int_{\overline{\G_{\gamma}(F_v)} \backslash \bar{\G}(F_v)} \phi_v(g^{-1} \gamma g) dg \cdot \prod_{v\in S_{F,\infty}} \int_{\overline{\G_{\sigma_v(\gamma)}(F_v)} \backslash \bar{\G}(F_v)} \phi_v(g^{-1} \sigma_v(\gamma) g) dg.
\end{equation*}
First, we find the condition of $(m,u)\in \mathcal{O}_{F}\times \mathcal{O}_{F}^{\times}$ such that $S_{[\gamma_{m,u}]}(\phi^{(\mathbf{a})})\neq 0$. 
Suppose that $S_{[\gamma_{m,u}]}(\phi^{(\mathbf{a})})\neq 0$.
For each $(m,u)\in \mathcal{O}_{F}\times \mathcal{O}_{F}^{\times}$, let $S_{m,u}^{1}$ be the subset of $S_{F,\infty}$ consisting of $v$ such that there are two roots of $x^2-\sigma_v(m)x+\sigma_v(u)=0$ in $F_v$ and $S_{m,u}^{2}:=S_{F,\infty}\backslash S_{m,u}^{1}$.
Note that $v\in S_{m,u}^{2}$ if and only if $F_{v}=\RR$ and $\sigma_v(m^2-4u)<0$.
Thus, if $v\in S_{m,u}^{2}$, then we have
\begin{equation*}
    |\sigma_v(m)|<2\sqrt{|\sigma_v(u)|}\leq 2\cdot e^{\frac{2\pi}{a_v(1+\epsilon_v)}} \sqrt{|\sigma_v(u)|} 
\end{equation*}
and 
\begin{equation*}
    |\sigma_v(m^2-4u)|\leq |\sigma_v(m^2)|+4|\sigma_v(u)|\leq 8|\sigma_v(u)|\leq 8\cdot e^{\frac{4\pi}{a_v(1+\epsilon_v)}}\cdot |\sigma_v(u)|.
\end{equation*}

If $v\in S_{m,u}^{1}$, then we have by \eqref{eq 122} and \eqref{eq 123}
\begin{equation*}
    e^{-\frac{4\pi }{a_v(1+\epsilon_v)}}<\left|\sigma_v(\alpha_{\gamma})\right|<e^{\frac{4\pi }{a_v(1+\epsilon_v)}}.
\end{equation*}
This implies that
\begin{equation}\label{eq 62}
    \left|\sigma_v(m)\pm\sqrt{\sigma_v(m^2-4u)} \right|^2<4\cdot e^{\frac{4 \pi}{a_v(1+\epsilon_v)}}\cdot |\sigma_v(u)|.
\end{equation}
Note that we have 
\begin{equation}\label{eq 64}
     \left| \sigma_v(m)+\sqrt{\sigma_v(m^2-4u)} \right|^2+\left| \sigma_v(m)-\sqrt{\sigma_v(m^2-4u)} \right|^2=2(|\sigma_v(m)|^2 + |\sigma_v(m^2-4u)|).
\end{equation}
Combining \eqref{eq 62} and \eqref{eq 64}, we get
\begin{equation*}
    |\sigma_v(m)|^2 + |\sigma_v(m^2-4u)|<4\cdot e^{\frac{4\pi}{a_v(1+\epsilon_v)}}\cdot |\sigma_v(u)|.
\end{equation*}
It immediately implies that 
\begin{equation}\label{eq 65}
    |\sigma_v(m)|<2\cdot e^{\frac{2\pi }{a_v(1+\epsilon_v)}} \sqrt{|\sigma_v(u)|} \text{ and } |\sigma_v(m^2-4u)|< 8\cdot e^{\frac{4\pi}{a_v(1+\epsilon_v)}}\cdot |\sigma_v(u)|.
\end{equation}
Hence, we conclude that if $S_{[\gamma]}(\phi^{(\mathbf{a})})$ is non-zero, then $(m,u)$ satisfies \eqref{eq 65} for all $v\in S_{F,\infty}$. 

Recalling Section \ref{s : geometric side}, let $U_{F}$ be a set of representatives of $\mathcal{O}_{F}^{\times}/\mathcal{O}_{F}^{\times^2}$.
Then, Dirichlet's unit theorem implies that $U_F$ is a finite set.  
In Section \ref{ss : ell}, we see that if $S_{[\gamma]}(\phi^{(\mathbf{a})})$ is non-zero, then there is a unique pair $(m,u)\in \mathcal{O}_{F}\times U_{F}$ corresponding to $\gamma$ such that $\gamma_{m,u}:=\sm 0 & 1\\
-u & m\esm \in [\gamma]$.
For $\mathbf{a}:=(a_v)_{v\in S_{F,\infty}}$, let $X_{\mathbf{a}}$ be the subset of $\mathcal{O}_{F}\times U_{F}$ consisting of $(m,u)$ satisfying \eqref{eq 65} for all $v\in S_{F,\infty}$.
Then, we have 
\begin{equation}\label{eq 68}
    \left| S_{\rm ell}(\phi^{(\mathbf{a})})\right|\leq \sum_{(m,u)\in X_{\mathbf{a}}} \vol(\overline{\G_{\gamma_{m,u}}(F)} \backslash \overline{\G_{\gamma_{m,u}}(\mathbb{A}_{F})})\left|S_{[\gamma_{m,u}]}(\phi^{(\mathbf{a})})\right|.
\end{equation}

Note that we have 
\begin{equation}\label{eq 67}
    B_{m,u}:=N_{F/\Q}(m^2-4u)=\prod_{v\in S_{F,\infty}} |\sigma_{v}(m^2-4u)|^{1+\epsilon_v}.
\end{equation}
Combining \eqref{eq 66}, \eqref{eq 67}, Lemmas \ref{lem : ell-4} and \ref{lem : ell-5}, for each $(m,u)\in X_{\mathbf{a}}$, we have
\begin{equation*}
    \begin{aligned}
        &\vol(\overline{\G_{\gamma_{m,u}}(F)} \backslash \overline{\G_{\gamma_{m,u}}(\mathbb{A}_{F})})\left|S_{[\gamma_{m,u}]}(\phi^{(\mathbf{a})})\right|\\
        &\ll_{F,u} A^{-\frac{1}{2}}\cdot\prod_{v\in S_{m,u}^{1}} \left|\sigma_{v}(m^2-4u) \right|^{\frac{13(1+\epsilon_v)}{2}}\prod_{v\in S_{m,u}^{2}}\left(\left|\sigma_v(m^2-4u)\right|^6 \cdot h_{v,a_v}(0)\right)\\
        &\ll_{F,u} A^{-\frac{1}{2}}\cdot\prod_{v\in S_{F,\infty}}e^{\frac{26\pi }{a_v}} \cdot \prod_{v\in S_{m,u}^{2}}h_{v,a_v}(0)\\
        &\ll_{F,u} A^{-\frac{1}{2}}\cdot\prod_{v\in S_{F,\infty}}e^{\frac{26\pi}{a_v}}\cdot \prod_{v\in S_{m,u}^{1}}h_{v,a_v}^{-1}(0) \cdot \prod_{v\in S_{F,\infty}}h_{v,a_v}(0)\\
        &\ll_{F,u} A^{-\frac{1}{2}}\cdot\prod_{v\in S_{F,\infty}}\left(e^{\frac{26\pi}{a_v}}\cdot h_{v,a_v}(0)\right).
    \end{aligned}
\end{equation*}
Here, the last inequality holds since a positive real number $a_{v}$ is less than $1$ for each $v\in S_{F,\infty}$.
Since $U_{F}$ is a finite set, it follows that
\begin{equation}\label{eq 69}
    \vol(\overline{\G_{\gamma_{m,u}}(F)} \backslash \overline{\G_{\gamma_{m,u}}(\mathbb{A}_{F})})\left|S_{[\gamma_{m,u}]}(\phi^{(\mathbf{a})})\right| \ll_{F} A^{-\frac{1}{2}}\cdot \prod_{v\in S_{F,\infty}}\left(e^{\frac{26\pi}{a_v}}\cdot h_{v,a_v}(0)\right).
\end{equation}

By the finiteness of $U_F$, Lemma \ref{lem 6} implies 
\begin{equation}\label{eq 70}
    \# X_{\mathbf{a}} \ll_{F} \prod_{v\in S_{F,\infty}} e^{\frac{2\pi n }{a_v}}.
\end{equation}
Therefore, we complete the proof of Lemma \ref{lem-ell,a} by combining \eqref{eq 68}, \eqref{eq 69} and \eqref{eq 70}.
\end{proof}

\begin{lem}\label{lem 8}
   With the above notation, assume that $a_{v}\leq 1$ for all $v\in S_{F,\infty}$. 
   Then, we have 
   \begin{equation*}
       \left|S_{\rm hyp}(\phi^{(\mathbf{a})})\right| \ll_{F} A^{-\frac{1}{2}} \cdot\prod_{v\in S_{F,\infty}}\left(e^{\frac{5\pi}{a_v}}\cdot h_{v,a_v}(0)\right).
   \end{equation*}
\end{lem}
\begin{proof}
    By \eqref{eq 21}, if $J_{\alpha}(\phi^{(\mathbf{a})})\neq 0$, then there is $w\in S_{F}$ such that
    \begin{equation*}
        \prod_{v\in S_{F,\infty}} J_{\sigma_v(\alpha),v}^{w}(\phi_{v,a_v})\neq 0.
    \end{equation*}
    Lemma \ref{lem 3} implies that for all $v\in S_{F,\infty}$, we have
    \begin{equation}\label{eq 73}
        e^{-\frac{4\pi}{a_v(1+\epsilon_v)}}<|\sigma_v(\alpha)|< e^{\frac{4\pi}{a_v(1+\epsilon_v)}}.
    \end{equation}
    Let $Y_{\mathbf{a}}$ be the subset of $\mathcal{O}_{F}^{\times}$ consisting of $\alpha$ satisfying \eqref{eq 73} for all $v\in S_{F,\infty}$.
    Then, we get 
    \begin{equation*}
        \left|S_{\rm hyp}(\phi^{(\mathbf{a})}) \right|\leq \frac{\vol(F^{\times}\backslash \mathbb{A}_{F}^{1})}{2}\sum_{\alpha\in Y_{\mathbf{a}}} \left|J_{\alpha}(\phi^{(\mathbf{a})}) \right|.
    \end{equation*}
    
Assume that $\alpha\in Y_{\mathbf{a}}$.
    Since $\alpha-1\in \mathcal{O}_{F}$, it follows that 
\begin{equation}\label{eq 75}
    1\leq \left|N_{F/\Q}(\alpha-1)\right|=\prod_{v\in S_{F,\infty}}\left|\sigma_v(\alpha-1)\right|^{1+\epsilon_v}.
\end{equation}
Then, for $w\in S_{F,\infty}$, we have 
\begin{equation}\label{eq 76}
    |\sigma_w(\alpha-1)|=|\sigma_w(\alpha)-1|\leq |\sigma_w(\alpha)|+1<e^{\frac{4\pi}{a_{w}(1+\epsilon_w)}}+1\leq 2\cdot e^{\frac{4\pi}{a_{w}(1+\epsilon_w)}}. 
\end{equation}
Combining \eqref{eq 75} and \eqref{eq 76}, for any $v\in S_{F,\infty}$, we get 
\begin{equation}
    \frac{1}{|\sigma_v(\alpha-1)|^{1+\epsilon_v}}<2^{n-1}\prod_{w\neq v}e^{\frac{4\pi}{a_{w}}}.
\end{equation}
Since $\alpha\in Y_{\mathbf{a}}$, it follows that $|\sigma_v(\alpha)|<e^{\frac{4\pi}{a_v(1+\epsilon_v)}}$ and that
\begin{equation}\label{eq 79}
    \left(\frac{|\sigma_{v}(\alpha)|^{\frac{1}{2}}}{|\sigma_{v}(\alpha)-1|}\right)^{1+\epsilon_v}<2^{n-1}e^{\frac{2\pi}{a_v}}\prod_{w\neq v} e^{\frac{4\pi }{a_{w}}}.
\end{equation}

 Note that $\frac{\Re(\alpha)}{|\alpha|}=\frac{\Re(\alpha^{-1})}{|\alpha^{-1}|}$, $\frac{|\alpha|}{|\alpha-1|^2}=\frac{|\alpha^{-1}|}{|\alpha^{-1}-1|^2}$ and $\widehat{h_{\phi_v}}$ is an even function.
By Lemma \ref{lem 3}, we may assume that $|\sigma_v(\alpha)|\geq 1$ without loss of generality.
If $v\neq w$, then Lemma \ref{lem 3} says that 
    \begin{equation}\label{eq 81}
        \left|J_{\sigma_v(\alpha),v}^{w}(\phi_{v,a_v})\right|\leq \frac{1+\epsilon_v}{2\pi}\cdot \frac{|\sigma_v(\alpha)|^{\frac{1+\epsilon_v}{2}}}{|\sigma_v(\alpha)-1|^{1+\epsilon_v}}.
    \end{equation}

Assume that $F_v=\RR$. 
By Lemma \ref{lem 3} and \eqref{eq 79}, we have
\begin{equation}\label{eq 78}
    \begin{aligned}
        \left|J_{\sigma_v(\alpha),v}^{v}(\phi_{v,a_v})\right|
&\leq \frac{2|\sigma_v(\alpha)|}{|\sigma_v(\alpha)-1|^2} \int_{\frac{\log|\sigma_v(\alpha)|}{2\pi}}^{\frac{1}{a_v}} \widehat{h_{v,a_v}}(x) \cosh(\pi x) dx\\
&\leq \frac{|\sigma_v(\alpha)|}{|\sigma_v(\alpha)-1|^2} \cdot \cosh\left(\frac{\pi}{a_{v}} \right) \cdot h_{v,a_v}(0)\\
&<\frac{|\sigma_{v}(\alpha)|^{\frac{1}{2}}}{|\sigma_{v}(\alpha)-1|}\left(2^{n-1}e^{\frac{2\pi}{a_v}}\prod_{w\neq v}e^{\frac{4\pi}{a_w}}\right)\cdot\cosh\left(\frac{\pi}{a_{v}} \right) \cdot h_{v,a_v}(0)\\
&<\frac{2^{n-1}|\sigma_{v}(\alpha)|^{\frac{1}{2}}}{|\sigma_{v}(\alpha)-1|}\prod_{w\neq v} \frac{1}{h_{w}(0)}\prod_{w\in S_{F,\infty}} e^{\frac{4\pi}{a_{w}}} \cdot h_{w,a_w}(0).
    \end{aligned}
\end{equation}
The last inequality holds since $h_{w,a_w}(0)=\frac{1}{a_w}h_w(0)\geq h_w(0)>0$.

Assume that $F_{v}=\CC$. 
Since $\widehat{h_v}$ is a smooth function with compact support, there is a positive real number $M_v$ such that 
\begin{equation*}
    \left|\widehat{h_{v,a_v}}(x) \right|\leq M_v
\end{equation*}
for all $x\in \RR$.
 By Lemma \ref{lem 3}, we have 
\begin{equation}\label{eq 120}
    \begin{aligned}
        \left|J_{\sigma_v(\alpha),v}^{v}(\phi_{v,a_v})\right|&\leq \frac{2|\sigma_v(\alpha)|}{|\sigma_v(\alpha)-1|^2}\cdot M_{v}\cdot \int_{\frac{\log (|\sigma_v(\alpha)|)}{\pi}}^{\frac{1}{a_v}} \frac{\sinh(\pi x)}{\cosh(\pi x)-\frac{\Re(\sigma_v(\alpha))}{|\sigma_v(\alpha)|}}dx\\
        &= \frac{2|\sigma_v(\alpha)|}{|\sigma_v(\alpha)-1|^2}\cdot M_{v}\cdot \frac{\left(\log s_2-\log s_1 \right)}{\pi},
    \end{aligned}
\end{equation}
where 
\begin{equation*}
    s_1:=\cosh(\log |\sigma_v(\alpha)|)-\frac{\Re(\sigma_v(\alpha))}{|\sigma_v(\alpha)|}= \frac{|\sigma_v(\alpha)-1|^2}{2|\sigma_v(\alpha)|} \text{ and }s_2:=\cosh\left(\frac{\pi}{a_v}\right)-\frac{\Re(\sigma_v(\alpha))}{|\sigma_v(\alpha)|}.
\end{equation*}
By \eqref{eq 79}, we obtain
\begin{equation}\label{eq 121}
\begin{aligned}
    \log s_2 -\log s_1&\leq \log\left(\cosh\left(\frac{\pi}{a_v}\right)+1\right)+\log\left(\frac{2|\sigma_v(\alpha)|}{|\sigma_v(\alpha)-1|^2} \right)\\
    &\leq \frac{\pi}{a_v}+\log \left(2^n e^{\frac{2\pi}{a_v}}\prod_{w\neq v} e^{\frac{4\pi}{a_w}}\right)\\
    &\leq n\log 2 + 4\pi \sum_{w\in S_{F,\infty}}\frac{1}{a_w}\\
    &\leq (4\pi+2 \log 2)\sum_{w\in S_{F,\infty}}\frac{1}{a_w}.
\end{aligned}
\end{equation}
The last inequality holds since $\sum_{w\in S_{F,\infty}}\frac{1}{a_{w}}\geq \frac{n}{2}$.
Thus, combining \eqref{eq 120} and \eqref{eq 121}, we have
\begin{equation}\label{eq 83}
\begin{aligned}
    \left|J_{\sigma_v(\alpha),v}^{v}(\phi_{v,a_v})\right|&\leq \frac{2(4\pi+2\log 2)M_v|\sigma_v(\alpha)|}{\pi |\sigma_v(\alpha)-1|^2}\sum_{w\in S_{F,\infty}} \frac{1}{a_w}\\
    &=\frac{2(4\pi+2\log 2)M_v|\sigma_v(\alpha)|}{\pi |\sigma_v(\alpha)-1|^2}\cdot \frac{\left(\prod_{w\in S_{F,\infty}}a_w\right)\left(\sum_{w\in S_{F,\infty}} \frac{1}{a_w} \right) }{\prod_{w\in S_{F,\infty}}h_{w}(0)}\prod_{w\in S_{F,\infty}} h_{w,a_w}(0) \\
    &\leq \frac{2(4\pi+2\log 2)M_v|\sigma_v(\alpha)|}{\pi |\sigma_v(\alpha)-1|^2}\cdot\frac{n}{\prod_{w\in S_{F,\infty}}h_{w}(0)}\prod_{w\in S_{F,\infty}} h_{w,a_w}(0)\\
    &\leq \frac{2(4\pi+2\log 2)M_v|\sigma_v(\alpha)|}{\pi |\sigma_v(\alpha)-1|^2}\cdot\frac{n}{\prod_{w\in S_{F,\infty}}h_{w}(0)}\prod_{w\in S_{F,\infty}}e^{\frac{4\pi}{a_w}}\cdot h_{w,a_w}(0).
\end{aligned}
\end{equation}

Recalling \eqref{eq 80}, we have
\begin{equation}\label{eq 84}
    \begin{aligned}
     \left|J_{\alpha}(\phi^{(\mathbf{a})})\right|&\leq A^{-\frac{1}{2}}\cdot\left|N_{F\backslash \Q}(\alpha-1)\right|\\
     &\times\left(2\log |N_{F\backslash \Q}(\alpha-1)|\cdot \left|\prod_{v\in S_{F,\infty}}J_{\sigma_{v}(\alpha),v}^{w_0}(\phi_{v,a_v})\right| + \sum_{w\in S_{F,\infty}}\left|\prod_{v\in S_{F,\infty}}J_{\sigma_v(\alpha),v}^{w}(\phi_{v,a_v})\right| \right)
     \end{aligned}
\end{equation}
for any $w_0\in S_{F,\fin}$.
Combining \eqref{eq 81}, \eqref{eq 78} and \eqref{eq 83}, we deduce that \eqref{eq 84} becomes
\begin{equation*}
\begin{aligned}
    \left|J_{\alpha}(\phi^{(\mathbf{a})})\right|\leq &A^{-\frac{1}{2}}\cdot \left|N_{F/\Q}(\alpha-1)\right|\cdot \prod_{v\in S_{F,\infty}} \frac{|\sigma_v(\alpha)|^{\frac{1+\epsilon_v}{2}}}{|\sigma_v(\alpha-1)|^{1+\epsilon_v}}\\
    \times &\bigg(2\cdot C_1\cdot \log\left|N_{F/\Q}(\alpha-1)\right| + \sum_{w\in S_{F,\infty}} C_{w,1} \prod_{v\in S_{F,\infty}}\left(e^{\frac{4\pi}{a_v}} \cdot h_{v,a_v}(0)\right)\bigg)\\
    =&A^{-\frac{1}{2}} \cdot \left|N_{F/\Q}(\alpha)\right|^{\frac{1}{2}}\bigg(2\cdot C_1\cdot \log\left|N_{F/\Q}(\alpha-1)\right| + \sum_{w\in S_{F,\infty}} C_{w,1} \prod_{v\in S_{F,\infty}}\left(e^{\frac{4\pi}{a_v}} \cdot h_{v,a_v}(0)\right)\bigg)
\end{aligned} 
\end{equation*}
for some constants $C_1$ and $C_{w,1}$ with $w\in S_{F,\infty}$.
Since $\alpha\in \mathcal{O}_{F}^{\times}$, the absolute value of the norm $N_{F/\Q}(\alpha)$ of $\alpha$ is $1$. 
By \eqref{eq 76}, we have 
\begin{equation*}
\begin{aligned}
    \log \left|N_{F/\Q}(\alpha-1)\right|&=\log\left(\prod_{w\in S_{F,\infty}} |\sigma_w(\alpha-1)|^{1+\epsilon_w}\right)\\
    &\leq n\log 2 + 4\pi \sum_{w\in S_{F,\infty}}\frac{1}{a_w}\\
    &\leq (4\pi+2\log 2)\sum_{w\in S_{F,\infty}} \frac{1}{a_w}\\
    &\leq \frac{n(2\log 2 + 4\pi)}{\prod_{w\in S_{F,\infty}}h_w(0)}\prod_{w\in S_{F,\infty}} h_{w,a_w}(0).
\end{aligned}
\end{equation*}

Therefore, we complete the proof of Lemma \ref{lem 8}.
\end{proof}

\begin{lem}\label{lem-par,a}
With the above notation, assume that $a_{v}\leq 1$ for all $v\in S_{F,\infty}$.
Then, we have
    \begin{equation}\label{eq 72}
        \left| S_{\rm par}(\phi^{(\mathbf{a})})\right|\ll_{F} A^{-\frac{1}{2}}\log A \cdot \prod_{v\in S_{F,\infty}} h_{v,a_v}(0).
    \end{equation}
\end{lem}
\begin{proof}
Assume that $v\in S_{F,\infty}$. 
Since $h_{v,a_v}$ is an even function, it follows that  
    \begin{equation*}
        \int_{\mathbb{R}}h_{v,a_v}(t)\cdot \frac{\Gamma'}{\Gamma}(1+(1+\epsilon_v)it)dt=2\int_{0}^{\infty}h_{v,a_v}(t)\cdot \Re\left(\frac{\Gamma'}{\Gamma}(1+(1+\epsilon_v)it) \right)dt.
    \end{equation*}
    By the asymptotic behavior of $\frac{\Gamma'}{\Gamma}(1+it)$, there are non-negative real numbers $A$ and $B$ such that 
    \begin{equation*}
        \left|\Re\left(\frac{\Gamma'}{\Gamma}(1+(1+\epsilon_v)it)\right)\right|\leq A+Bt
    \end{equation*}
    for all non-negative real numbers $t$. Thus, 
    \begin{equation*}
        \begin{aligned}
            \left|\int_{0}^{\infty} h_{v,a_v}(t) \cdot \Re\left(\frac{\Gamma'}{\Gamma}(1+(1+\epsilon_v)it)\right) dt\right|&\leq \int_{0}^{\infty}\frac{1}{a_v}\cdot h_{v}\left(\frac{t}{a_v}\right) (A+Bt) dt\\
            &= A\int_{0}^{\infty} \frac{1}{a_v}\cdot h_{v}\left(\frac{t}{a_v}\right)dt+B\int_{0}^{\infty} \frac{t}{a_v}\cdot h_{v}\left(\frac{t}{a_v}\right)dt\\
            &=A\int_{0}^{\infty}h_{v}(t)dt+a_v\cdot B\int_{0}^{\infty}th_v(t)dt.
        \end{aligned}
    \end{equation*}
    Since $a_{v}\leq 1$, we have
    \begin{equation*}
        \left|\int_{\mathbb{R}}h_{v,a_v}(t)\cdot \frac{\Gamma'}{\Gamma}(1+it)dt \right|\ll_{v} 1.
    \end{equation*}
    Then, Lemma \ref{lem 7} implies that there are positive constants $C_{v}^{(1)}$ and $C_{v}^{(2)}$ satisfying
    \begin{equation*}
        \left|Z_{v}'(1,\phi_{v,a_v})\right|\leq C_{v}^{(1)} + C_{v}^{(2)}\cdot h_{v,a_v}(0).
    \end{equation*}
    Again, since $a_{v}\leq 1$ for all $v\in S_{F,\infty}$, we obtain \eqref{eq 72} by \eqref{eq 71}.
\end{proof}

\begin{lem}\label{lem-eis,a}
    With the above notation, assume that $a_{v}\leq 1$ for all $v\in S_{F,\infty}$.
    Then, we have 
    \begin{equation*}
        \left|S_{\rm Eis}(\phi^{(\mathbf{a})}) \right|\ll_{F} A_{J}\cdot A^{\frac{1}{2}} \cdot \prod_{v\in S_{F,\infty}} h_{v,a_v}(0).
    \end{equation*}
\end{lem}
\begin{proof}
    Recalling Lemma \ref{prop : Eis}, we have 
\begin{equation*}
     S_{\rm Eis}(\phi^{(\mathbf{a})})=\frac{A_{J}}{2^{r_1+2}\pi}\cdot \int_{\RR}\bigg(\frac{\Lambda_{F}'(2it)}{\Lambda_{F}(2it)} - \frac{\Lambda_{F}'(2it+1)}{\Lambda_{F}(2it+1)}\bigg)\cdot \left(\prod_{v\in S_{F,\infty}}h_{v,a_v}(t)\right) \cdot \dim V_{it}^{\K(J)}dt.
\end{equation*}
There are only finitely many $v\in S_{F,\fin}$ such that $\val_{v}(J)>0$.
By \cite[pp. 73]{Gel}, we have 
\begin{equation}\label{eq 102}
    \dim V_{it}^{\K(J)}\leq 2^{r_1}\cdot \prod_{v\in S_{F,\infty}} (1+\val_v(J)).
\end{equation}
For each $v\in S_{F,\fin}$, let $r_{v}:=[\mathcal{O}_{F}/ \mathfrak{p}_{v} : \FF_{p}]$, where $\mathfrak{p}_{v}$ denotes a prime ideal of $\mathcal{O}_{F}$ corresponding to $v$ and $p$ denotes a prime satisfying $(p):=\mathfrak{p}_{v}\cap \mathbb{Z}$. 
Then, we have 
\begin{equation}\label{eq 103}
     N_{F/\Q}(J)=\prod_{p}\left(\prod_{\mathfrak{p}_{v}\mid p} N_{F/\Q}(\mathfrak{p}_{v})^{\val_{v}(J)}\right)
     =\prod_{p}p^{\sum_{\mathfrak{p}_{v}\mid p}r_v\cdot \val_v(J)}.  
\end{equation}
Let $d(N)$ be the number of divisors of an integer $N$. 
Let $n:=[F:\Q]$.
Since the number of prime ideals $\mathfrak{p}_{v}$ of $\mathcal{O}_{F}$ lying over $p$ is less than or equal $n$, it follows that 
\begin{equation*}
    \left(d\left(N_{F/\Q}(J)\right)\right)^{n}=\prod_{p}\left(1+\sum_{\mathfrak{p}_{v}\mid p} r_v\cdot \val_{v}(J)\right)^{n}\geq \prod_{p}\left(\prod_{\mathfrak{p}_{v}\mid p} (1+\val_{v}(J)) \right)=\prod_{v\in S_{F,\fin}} (1+\val_{v}(J)).
\end{equation*}
Combining \eqref{eq 102} and \eqref{eq 103}, we get 
\begin{equation*}
    \frac{\dim V_{it}^{\K_0(J)}}{2^{r_1}}\leq  \left(d\left(N_{F/\Q}(J)\right)\right)^{n}= d(A)^n.
\end{equation*}
Since $d(A)\ll_{\epsilon} A^{\epsilon}$ for $\epsilon>0$, it follows that
\begin{equation*}
\begin{aligned}
     \left|S_{\rm Eis}(\phi^{(\mathbf{a})}) \right|&\ll_{F} A_{J}\cdot A^{\frac{1}{2}}\cdot\int_{\RR}\left|\frac{\Lambda_{F}'(2it)}{\Lambda_{F}(2it)} - \frac{\Lambda_{F}'(2it+1)}{\Lambda_{F}(2it+1)}\right|\prod_{v\in S_{F,\infty}}h_{v,a_v}(t) dt.
\end{aligned}
\end{equation*}

By \cite[Proposition 5.7]{IK}, there are positive real numbers $C_1$ and $C_2$ such that for any $t\in \RR$, 
\begin{equation*}
    \left|\frac{\Lambda_{F}'(2it)}{\Lambda_{F}(2it)} - \frac{\Lambda_{F}'(2it+1)}{\Lambda_{F}(2it+1)}\right|\leq C_1+C_2|t|.
\end{equation*}
Thus, we obtain that 
\begin{equation*}
    \begin{aligned}
        \left|S_{\rm Eis}(\phi^{(\mathbf{a})}) \right|&\ll_{F} A_{J}\cdot A^{\frac{1}{2}}\cdot \int_{0}^{\infty}(C_1+C_2t)\prod_{v\in S_{F,\infty}}h_{v,a_v}(t) dt.
    \end{aligned}
\end{equation*}
Since $\widehat{h_{v,a_v}}$ is a non-negative function, $h_{v,a_v}(t)$ is less than or equal to $h_{v,a_v}(0)$.
Then, we have 
\begin{equation*}
    \begin{aligned}
        \left|\int_{0}^{\infty}(C_1+C_2t)\prod_{v\in S_{F,\infty}}h_{v,a_v}(t) dt\right|&\leq \prod_{v\neq w}h_{v,a_v}(0) \cdot \left|\int_{0}^{\infty} (C_1+C_2t)h_{w,a_w}(t) dt \right|\\
        &=\prod_{v\neq w}h_{v,a_v}(0)\cdot \left|C_1\cdot \int_{0}^{\infty} h_{w}(t)dt + C_2\cdot a_{w}\int_{0}^{\infty} th_{w}(t) dt \right|\\
        &\ll_{F} \prod_{v\in S_{F,\infty}} h_{v,a_v}(0).
    \end{aligned}
\end{equation*}
Here, the last inequality holds because $a_{w}$ is a positive real number with $a_{w}\leq 1$.
\end{proof}

Using the argument in the proof of Lemma \ref{lem-eis,a}, we have
\begin{equation*}
    \dim V_{0}^{\K(J)}\ll_{F} A^{\frac{1}{2}}.
\end{equation*}
Thus, Lemma \ref{prop : Res} implies the following lemma.

\begin{lem}\label{lem-res,a}
    With the above notation, we have 
    \begin{equation*}
        \left|S_{\mathrm{Res}}(\phi^{(\mathbf{a})}) \right|\ll_{F} A_{J}\cdot A^{\frac{1}{2}}\cdot \prod_{v\in S_{F,\infty}} h_{v,a_v}(0).
    \end{equation*}
\end{lem}

Now, we are ready to prove Theorem \ref{thm : main}.

\begin{proof}[Proof of Theorem \ref{thm : main}]
Let $(\mathbf{a}):=(a_v)_{v\in S_{F,\infty}}\in \RR_{>0}^{\# S_{F,\infty}}$. 
For each $v\in S_{F,\infty}$, assume that $a_{v}\leq 1$.
By Lemma \ref{lem : lem in sec 2} and \eqref{eq 8}, we have  
\begin{equation}\label{eq 85}
\begin{aligned}
    \mathcal{N}(J)
    &\leq  \frac{2^{r_1}}{A_J\cdot \prod_{v\in S_{F,\infty}}h_{v,a_v}(0)}\sum_{\pi\in \mathfrak{X}_{F}} \tr \pi(\phi^{(\mathbf{a})})\\
    &\leq  \frac{2^{r_1}}{A_J\cdot \prod_{v\in S_{F,\infty}}h_{v,a_v}(0)}\bigg(\left|S_{\mathrm{one}}(\phi^{(\mathbf{a})})\right|+\left|S_{\rm id}(\phi^{(\mathbf{a})}) \right|+\left|S_{\rm ell}(\phi^{(\mathbf{a})}) \right|+\left|S_{\rm hyp}(\phi^{(\mathbf{a})}) \right|\\
    &\quad\quad\quad\quad\quad\quad\quad\quad\quad\qquad+\left| S_{\rm par}(\phi^{(\mathbf{a})})\right|+\left|S_{\rm Eis}(\phi^{(\mathbf{a})}) \right|+ \left|S_{\mathrm{Res}}(\phi^{(\mathbf{a})}) \right|\bigg).
\end{aligned}
\end{equation}
By Lemmas \ref{lem-sp,a}, \ref{lem-id,a}, \ref{lem-ell,a}, \ref{lem 8}, \ref{lem-par,a}, \ref{lem-eis,a}, and \ref{lem-res,a}, we have
\begin{equation}\label{eq 86}
    \begin{aligned}
    \mathcal{N}(J)\ll_{F} \frac{1}{A_{J}}\bigg(&A_{J}\cdot \prod_{v\in S_{F,\infty}}e^{\frac{\pi}{a_v}} + \prod_{v\in S_{F,\infty}} a_v^{2+\epsilon_v} + A^{-\frac{1}{2}}\cdot \prod_{v\in S_{F,\infty}} e^{\frac{2\pi(n+13)}{a_v}} \\
    &+A^{-\frac{1}{2}}\cdot \prod_{v\in S_{F,\infty}}e^{\frac{5\pi}{a_v}} +A^{-\frac{1}{2}}\log A +  A_{J}\cdot A^{\frac{1}{2}} \bigg).
    \end{aligned}
\end{equation}
For each $v\in S_{F,\infty}$, let
\begin{equation*}
    a_v:=\frac{8\pi n(n+13)}{\log A}.
\end{equation*}
We see that $a_{v}\leq 1$ as $A\to \infty$.
Since $A=|N_{F/\Q}(J)|\leq A_{J}^{-1}$, the main term of the right hand side of $\eqref{eq 86}$ comes from $A_J^{-1}\prod_{v\in S_{F,\infty}} a_{v}^{2+\epsilon_v}$.
Therefore, we conclude that
\begin{equation*}
    \mathcal{N}(J)\ll_{F} \frac{1}{A_J(\log A)^{2r_1+3r_2}} = \frac{[\SL_2(\mathcal{O}_{F}) : \Gamma_0(J) ] }{(\log (N_{F/\Q}(J)))^{2r_1+3r_2}}, \quad \left|N_{F/\Q}(J)\right|\to \infty.
\end{equation*}
\end{proof}

\thispagestyle{empty}
{\footnotesize
\nocite{*}
\bibliographystyle{abbrv}
\bibliography{Selberg}
} 

\end{document}